\documentclass{article}

\usepackage{arxiv}

\usepackage[utf8]{inputenc} 
\usepackage[T1]{fontenc}    
\usepackage{hyperref}       
\usepackage{url}            
\usepackage{booktabs}       
\usepackage{amsfonts}       
\usepackage{nicefrac}       
\usepackage{microtype}      
\usepackage{lipsum}		
\usepackage{graphicx}
\usepackage{natbib}
\usepackage{doi}

\usepackage{amsthm, amsmath}
\usepackage{enumerate}
\usepackage{bbm}
\usepackage{upgreek}

\newtheorem{theorem}{Theorem}
\newtheorem{lemma}{Lemma}
\newtheorem{proposition}{Proposition}
\theoremstyle{remark}
\newtheorem{assumption}{Assumption}
\newtheorem{intassumption}{Assumption}
\numberwithin{intassumption}{assumption}

\theoremstyle{definition}
\newtheorem{definition}{Definition}
\theoremstyle{remark}
\newtheorem{remark}{Remark}

\newcommand{\zerob}{\pmb 0}
\newcommand{\R}{\mathbb{R}}
\newcommand{\N}{\mathbb{N}}
\newcommand{\Z}{\mathbb{Z}}

\newcommand{\E}{\mathbb{E}}
\newcommand{\Prob}{\mathbb{P}}
\newcommand{\var}{\text{Var}}
\newcommand{\cov}{\text{Cov}}
\newcommand{\one}{\mathbbm{1}}
\newcommand{\n}{^{(n)}}
\newcommand{\mub}{\boldsymbol{\mu}}

\newcommand{\Upsilonb}{\boldsymbol{\Upsilon}}
\newcommand{\lambdab}{\boldsymbol{\lambda}}
\newcommand{\thetab}{\boldsymbol{\theta}}
\newcommand{\taub}{\boldsymbol{\tau}}
\newcommand{\varthetab}{\boldsymbol{\vartheta}}
\newcommand{\xb}{\boldsymbol{x}}
\newcommand{\hb}{\boldsymbol{h}}

\renewcommand{\thetheorem}{\thesection.\arabic{theorem}}
\renewcommand{\thelemma}{\thesection.\arabic{lemma}}

\title{Construction of optimal tests for symmetry on the torus and their quantitative error bounds}

\author{ \href{https://orcid.org/0000-0002-7549-1348}{\includegraphics[scale=0.06]{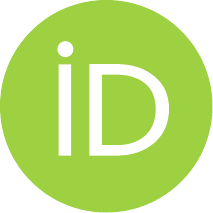}\hspace{1mm}Andreas Anastasiou} \\
    Department of Mathematics and Statistics \\
	University of Cyprus\\
	\texttt{anastasiou.andreas@ucy.ac.cy} \\
	\And
	\href{https://orcid.org/0000-0002-2290-8437}{\includegraphics[scale=0.06]{Plots/orcid.pdf}\hspace{1mm}Christophe Ley} \\
    Department of Mathematics \\
	University of Luxembourg\\
	\texttt{christophe.ley@uni.lu} \\
    \And
	\href{https://orcid.org/0009-0006-1372-7032}{\includegraphics[scale=0.06]{Plots/orcid.pdf}\hspace{1mm}Sophia Loizidou} \\
    Department of Mathematics \\
	University of Luxembourg\\
	\texttt{sophia.loizidou@uni.lu} \\
}

\renewcommand{\headeright}{}
\renewcommand{\undertitle}{To appear in Bernoulli}
\renewcommand{\shorttitle}{Construction of optimal tests for symmetry on the torus and their quantitative error bounds}

\begin{document}
\maketitle

\begin{abstract}
In this paper, we investigate the general problem of assessing symmetry in data points on the hyper-dimensional torus, a question that originally emerged in applications from bioinformatics and directional statistics.
We develop optimal tests for symmetry for both scenarios where the center of symmetry is known and where it is unknown. 
Our new tests are not only valid under a given parametric hypothesis but also under a very broad class of symmetric distributions.
The asymptotic behavior of the proposed tests is studied both under the null hypothesis and local alternatives. 
A key contribution   of our paper is that we accompany our asymptotic results with error guarantees by deriving  quantitative bounds on the distributional distance between the exact (unknown) distribution of the test statistic and its asymptotic counterpart by leveraging Stein’s method.
The finite-sample performance of the tests is evaluated through simulation studies, and their practical utility in bioinformatics is demonstrated via an application to protein folding data. 
\end{abstract}


\section{Introduction} \label{sec: Introduction}

Various complex data obtained from the real world can be viewed as data on the $d$-dimensional torus, which is the {Cartesian} product of $d$ circles.
Concrete examples include wind direction measured at different times during the day \citep{johnson_measures_1977, kato_distribution_2009, shieh_inferences_2005}, directions of steepest descent before and after an earthquake \citep{rivest_decentred_1997}, direction of animal movement \citep{mastrantonio_modelling_2022}, shared orthologous genes between circular genomes \citep{fernandez-duran_modeling_2014}, morphological data from human neurons \citep{puerta_regularized_2015} and, in marine biology, the spawning time of a particular fish and the time of the low tide \citep{kato_circularcircular_2008}. {A domain that has given rise in the past two decades to a lot of such {toroidal data} is bioinformatics, with examples including RNA data \citep{nodehi_estimation_2021} and the circadian clock of two different tissues of a mouse \citep{liu_phase_2006}. The most prominent requirement from bioinformatics however is the  protein structure prediction problem, where backbone dihedral angles from proteins can be viewed as data on the two-dimensional torus, see for instance 
\cite{kato_mobius_2015}, 
\cite{mardia_protein_2007} and  \cite{singh_probabilistic_2002}. 
A larger number of protein angles, including torsion and side-chain angles, was considered in e.g. \cite{mardia_multivariate_2008}, \cite{mardia_mixtures_2012}, and \cite{TWCM}.
Predicting the three-dimensional folding of a protein from its one-dimensional amino acid sequence remains one of the most important yet challenging problems in science. Accurate protein structure prediction has far-reaching implications for drug discovery, vaccine development, understanding disease mechanisms, human cell engineering, and enzyme design, among other applications. In recognition of their groundbreaking work, Demis Hassabis and John Jumper were awarded the 2024 Nobel Prize in Chemistry for AlphaFold2, a deep learning–based program capable of predicting protein structures with remarkable accuracy \citep{senior2019protein}. However, as highlighted in \cite{lane2023protein}, there remains a pressing need to complement AlphaFold2’s single-point predictions with rigorous assessments of statistical uncertainty, in particular well-fitting probability distributions on the (hyper-)torus.}



In the literature there exist a lot of distributions for angular data.
Firstly, for one angle, meaning circular data, there is an abundance of literature proposing models, see for example \cite{mardia_directional_2000} for an overview.
The most classical distributions are the von Mises distribution, which arises as a maximum entropy distribution, the wrapped Cauchy, the cardioid, and the wrapped Normal. 
Beyond one angle, we speak of data points on the (hyper-)torus.
Distributions for two angles include the bivariate von Mises distribution \citep{mardia_statistics_1975}, its submodels, which are the Sine \citep{singh_probabilistic_2002} and Cosine \citep{mardia_protein_2007} distributions, and the bivariate wrapped Cauchy distribution \citep{kato_mobius_2015, kato_distribution_2009}. 
The trivariate wrapped Cauchy copula (TWCC) \citep{TWCC} can model three angles, while the multivariate von Mises \citep{mardia_multivariate_2008}, multivariate wrapped normal \citep{greco2021robust} and the multivariate non-negative trigonometric sums (MNNTS) \citep{fernandez-duran_modeling_2014} models can be used for any dimension.

All aforementioned distributions, except the MNNTS, are symmetric models. 
However, many datasets involve skewed data, such as in \cite{kato_distribution_2009}, \cite{shieh2011modeling} and \cite{bekker2022}. In particular, a need for asymmetric models on the torus has been expressed in \cite{Mardia13} in the context of protein structural bioinformatics.
In order to overcome this important shortcoming of flexible asymmetric models, \cite{ameijeiras-alonso_sine-skewed_2022} proposed the sine-skewed family of distributions, building on the work of \cite{umbach_building_2009} and \cite{abe_sine-skewed_2011} which skewed one-dimensional circular distributions.
The proposed transformation, given in \eqref{eq: sine_skewed_density}, can turn any symmetric distribution, of any dimension, into an asymmetric one and has several attractive properties, such as unchanged normalizing constant, easy interpretation and simple data generating mechanism. 
It was shown in \cite{ameijeiras-alonso_sine-skewed_2022} that this transformation improves for various protein structures the fit as compared to the more classical symmetric models. The work of \cite{ameijeiras-alonso_sine-skewed_2022} is, to the best of our knowledge, the only skewing mechanism for the $d-$dimensional torus proposed in the literature.
Thanks to 
its attractive properties, simplicity and good fit to protein data, the sine-skewed construction, in particular in combination with the Sine model, has been implemented in the probabilistic programming languages Pyro and NumPyro by bioinformaticians \citep{ronning_time_efficient_2021}. 
This sine-skewing construction raises the following question: when should the simpler symmetric distribution be used and when should the more complicated, yet more flexible skewed version of it be preferred? 
This question was partly answered in \cite{ameijeiras-alonso_sine-skewed_2022}, using likelihood ratio tests which they applied on different amino acid types.
However, these parametric tests are not enough for an informed decision whether the dataset is symmetric or not.
Generally valid tests are needed to better understand the data without any restrictive parametric assumptions. {On the hypersphere, various efficient procedures for testing for symmetry have been proposed and investigated in the literature, see for instance \cite{ley2017skew, garcia2020optimal} and \cite{sobolev_tests_symmetry}. On the circle, corresponding to dimension $d=1$ in our toroidal setting}, optimal tests are proposed by \cite{ley_simple_2014}  and \cite{ameijeiras-alonso_optimal_2021} for the cases of known and unknown symmetry center, respectively. 
For  dimensions two and higher, which are more important for the bioinformatics applications mentioned above, no such tests exist.
The goal of this paper is to fill this important gap and derive tests for symmetry for $d$-dimensional toroidal data that are robust to the assumption of the underlying distribution of the data, for both known and unknown symmetry centers. Moreover, we wish our tests to be efficient against the class of sine-skewed alternatives. 

In order to reach our goals we develop  the Le Cam theory of asymptotic experiments \citep{le2012asymptotic} for hyper-toroidal settings, and optimality is to be understood in the Le Cam sense, namely asymptotically (in the sample size) and locally (against local sine-skewed deviations from  symmetry). In particular, we will establish the Uniform Local Asymptotic Normality (ULAN) property for a broad class of symmetric distributions on the hyper-torus. For both known and unknown symmetry center scenarios, we start by constructing  optimal parametric tests under a specified symmetric distribution, which we will then turn into semi-parametric tests that are valid under a very broad class of symmetric distributions on the torus satisfying mild regularity conditions. 
Each resulting semi-parametric test will not only be robust to the assumption of the underlying distribution but moreover remain optimal under the parametric distribution it was built under. 
When the center of symmetry is known, all tests are exactly of the same form, leading to a single test that is uniformly optimal to test symmetry against sine-skewed distributions.  
This is a very powerful and quite  rare result. 
The tests are easy to compute numerically, in practice they can be computed within a few seconds, {with their computational complexity increasing with the dimension of the data set 
and the difficulty of inverting the corresponding Fisher Information matrix, which is $O(d^3)$ for a $d\times d$ matrix \citep{trefethen2022numerical}, as well as the sample size. 
More specifically, the computational complexity of the algorithm in the case of the specified symmetry center case is $O(n^2d^9)$, where $d$ is the dimension of the problem and $n$ is the sample size of the dataset under consideration. 
In the case of the unspecified symmetry center, it becomes $O(n^6d^{23})$.}

A particular highlight of our paper is that we accompany our asymptotic results with error guarantees, in the sense that we measure the asymptotic approximation error at any sample size. Indeed, since the rejection rules of our new tests will be based on their asymptotic distribution, there will inevitably be approximation errors as real data naturally have a finite sample size. The larger this sample size, the smaller the approximation error between the intractable exact distribution of our test statistics and their simple asymptotic law. This is a reality for the overwhelming majority of existing hypothesis testing procedures. Unlike most of the literature, in this paper we wish to quantify this approximation error, and we will tackle this difficult problem through the suit of several new techniques. One part of our proofs will rely on Stein's Method, which is an important tool in applied and theoretical probability, whose principal aim is to provide quantitative assessments in distributional approximation problems of the form $W \approx Z$,
where $Z$ follows a simple distribution  and $W$ is the object of interest \citep{stein1972bound}. 
Stein's Method is known to be effective for numerous approximation problems, including the normal, the Poisson, the exponential and the Beta, among many others \citep{ley_reinert_swan2017, ross_stein2011}. It has successively been used to assess approximation errors for maximum likelihood estimators \citep{anastasiou2017bounds, anastasiou2017_MLE_m, anastasiou2018_MLE, anastasiou_gaunt2021},
the likelihood ratio statistic \citep{reinert_anastasiou2020}, Pearson's chi-square statistic \citep{gaunt_pickett_reinert2017}, Friedman's statistic \citep{gaunt_reinert_Friedman2023}, and neural networks \citep{favaro_quantitative_2025}, to name but a few. For a general overview, we refer the reader to \cite{anastasiou_stein2023}. 

The novelty of our paper thus does not just arise from defining symmetry tests for the first time on the hyper-torus and hereby providing new handles to model protein data in view of the protein structure prediction problem.
It is also, to the best of our knowledge, the first paper that combines Le Cam's theory of asymptotic experiments and Stein's Method in the context of hypothesis testing.
The results obtained in that direction are presented in Theorems~\ref{thm: steins method specified} and \ref{theorem: stein's method unspecified} and can pave the way to more general results combining the two methods.
Moreover, in Proposition~\ref{prop: QMD}, we provide a general result for quadratic mean differentiability of products of functions, which recovers various ad hoc results from the literature.
Finally, our paper offers a statistical confirmation that dihedral angles from proteins are asymmetric, which provides further insights in the structure of the proteins.

The outline of the paper is as follows. 
In Section~\ref{sec: ULAN property}, through Proposition~\ref{prop: QMD}, we establish the ULAN property of the family of sine-skewed distributions, as defined in \eqref{eq: sine_skewed_density}.
In  Sections~\ref{sec:knowncenter} and \ref{sec:unknowncenter}, we respectively propose optimal tests for symmetry on the $d$-dimensional torus in the cases that the symmetry center is known or unknown and needs to be estimated. 
The theoretical properties of both tests are investigated, including the asymptotic distribution under the null hypothesis of symmetry, which turns out to be chi-squared  with $d$ degrees of freedom, and under local alternatives, which is non-central chi-squared  with $d$ degrees of freedom, where $d$ is the dimensionality of the problem.
Bounds between the complicated finite-sample distribution of the test statistics and their simple limiting distributions are  derived via a suit of new techniques and leveraging Stein's Method.
Extensive numerical simulations, as well as a real data application using protein folding data can be found in Sections~\ref{sec: simulations} and \ref{sec: real data}, respectively. Section~\ref{sec: conclusion} concludes the paper with a summary of the most important findings.
{The Supplementary Material contains all proofs of the theorems and lemmas presented in the main paper, along with some further results that are used for the proofs,
as well as some further simulation results.
The implementation of the tests can be found in \url{https://doi.org/10.5281/zenodo.17224137} and \url{https://github.com/Sophia-Loizidou/Symmetry_test_on_hypertorus}.}

\section{Family of distributions and the ULAN property} \label{sec: ULAN property}

\subsection{A technical result about Quadratic Mean Differentiability}\label{sec:QMD}

A common analytical condition in order to establish a ULAN property is quadratic mean differentiability (QMD) of the square root of densities (see Section~\ref{sec:ULAN} for details in our context). Since sine-skewed densities correspond to a product or composition of two functions, we provide in this section a general statement for the QMD property under function composition. This result goes beyond the mere setting of this paper and actually holds for very general settings. We therefore also consider a general support $\R^m$ for the pdfs.
The proof can be found in Section \ref{appendix: proof prop QMD} of the Supplementary Material.
For the proposition, and the rest of the paper, we denote the usual supremum norm of a function $\lVert h \rVert = \lVert h \rVert_\infty = \sup_{x\in \R} |h(x)|$ {and the class of continuously differentiable functions by $\mathcal{C}^1$}.

\begin{proposition} \label{prop: QMD}
    Consider the functions $f(\xb;\mub), g(\xb;\mub,\lambdab)$, $f, g: \mathcal{X}\subset \R^k \rightarrow \R^+$ with parameters $\mub \in \R^k, \lambdab \in \R^m$. Assume that $f(\xb; \mub)$ and $f(\xb; \mub)g(\xb; \mub, \lambdab)$ are probability density functions and that $f$ and $g$ satisfy the following conditions:
    \begin{enumerate}
        \item[(i)] $g(\xb; \mub, \lambdab)$ is $\mathcal{C}^1$ with respect to $\lambdab$ {for all $\xb\in\mathcal{X}, \mub\in\R^k$};

         \item[(ii)] $||\nabla_{\lambdab} g^{1/2}(\xb; \mub, \lambdab)|_{\lambdab = \boldsymbol{0}}|| \leq C \ell(\xb;\mub)$ for some $C\in\R^+$ independent of  ${\xb},\mub, \lambdab$ and some function $\ell(\xb;\mub)$ that is in $L^2(f(\xb;\mub)d\xb)$ {for all $\xb\in\mathcal{X}, \mub\in\R^k$};

        \item[(iii)] $f(\xb; \mub), g(\xb; \mub, \lambdab) > 0$ almost everywhere (a.e.) over $\mathcal{X}$ {for all $\mub\in\R^k$ and $\lambdab\in\R^m$};
        
        \item[(iv)] $g(\xb; \mub, \boldsymbol{0}) = 1$ and $\nabla_{\mub} g(\xb; \mub, \lambdab) | _{\lambdab = \boldsymbol{0}} = \boldsymbol{0}$ {for all $\xb\in\mathcal{X}, \mub\in\R^k$};

        \item[(v)] For {all $\mub = (\mu_1, \ldots, \mu_k)'\in\R^k$} and $i \in \{1, \ldots , k\}$, $\int_{\mathcal{X}} \left( \frac{\partial}{\partial {\mu_i}} f^{1/2}(\xb; \mub) \right)^2 {\mathrm d}\xb < \infty$; 

    \end{enumerate}
    Then  $f(\xb; \mub)$ is QMD {with respect to $\mub\in\R^k$} with quadratic mean $\nabla_{\mub} f^{1/2}(\xb; \mub)$ if and only if $f(\xb; \mub) g(\xb; \mub, \lambdab)$ is QMD {with respect to $\mub\in\R^k$ and $\lambdab\in\R^m$ evaluated at $\lambdab = \boldsymbol{0}$} with quadratic mean $\begin{pmatrix}
        \nabla_{\mub} f^{1/2}(\xb; \mub) \\ f^{1/2}(\xb; \mub) \nabla_{\lambdab} g^{1/2}(\xb; \mub, \lambdab)|_{\scriptscriptstyle\lambda=\boldsymbol{0}}
    \end{pmatrix}$.
\end{proposition}

{Note that (ii) implies finiteness of the Fisher Information matrix for the $\lambda$ part, meaning $\int_{\mathcal{X}} \left( \frac{\partial}{\partial {\lambda_j}} g^{1/2}(\xb; \mub, \lambdab)|_{\lambdab = 0} \right)^2 f(\xb; \mub) {\mathrm d} \xb < \infty$.}
In order to avoid any confusion, we note that we denote by QMD both `quadratic mean differentiability' and `quadratic mean differentiable'.
Proposition~\ref{prop: QMD} is very powerful as it gives a general proof for such modulated distributions on $\R^m$ (to the best of our knowledge, this terminology goes back to \cite{JUPP2016107}).  
It retrieves in a single sweep several QMD results established in the literature for such distributions, for instance in the circular case for testing symmetry against $k$-sine-skewed alternatives in  Theorem 2.1 of \cite{ley_simple_2014}, in the spherical case for testing rotational symmetry against skew-rotationally-symmetric alternatives in  Theorem 1 of \cite{ley2017skew}, in the multivariate Euclidean setting for testing elliptical symmetry against skew-elliptical alternatives in Theorem 2.1 of \cite{babic2021optimal} (under some mild regularity assumption on their skewing function $\Pi$), to cite but these. Our result is also directly applicable for a range of modulated distribution models such  as weighted distributions (see \cite{nakhaei2025weighting} and references therein), cosine perturbation on the circle \citep{abe2011symmetric}, etc. 
Since QMD is at the core of the ULAN property, Proposition~\ref{prop: QMD} not only retrieves results from the literature but paves the way to obtain quite immediately ULAN results for modulated distributions and, hence, avoids an ad hoc proof in each case. Finally, note that in Assumption (ii) of Proposition \ref{prop: QMD} the function $\ell(\xb;\mub)$ can coincide with 1 when the support of the density is bounded, as will be the case for the family of sine-skewed distributions we describe next.

\subsection{Family of sine-skewed distributions}

As already mentioned in Section~\ref{sec: Introduction}, the family of distributions we are interested in was proposed by \cite{ameijeiras-alonso_sine-skewed_2022}.
For a $d-$dimensional angular vector $\thetab \in [-\pi, \pi)^d$, we consider sine-skewed densities of the form  
\begin{align}\label{eq: sine_skewed_density}
    \thetab\mapsto f_{\mub, \lambdab}(\thetab; \Upsilonb) := f_0(\thetab - \mub; \Upsilonb)\left(1+\sum_{j=1}^{d} \lambda_j \sin (\theta_j-\mu_j)\right) \quad \mbox{subject to}\,\, 
    \sum_{j=1}^{d} |\lambda_j| \leq 1,
\end{align}
where $\mub \in [-\pi, \pi)^d$ is the hyper-toroidal location parameter, $\lambdab \in \R^d$ the skewness parameter, and $f_0$ is any symmetric density on the $d$-dimensional torus with $\Upsilonb$ the corresponding set of non-location parameters. 
We introduce the notation $S^d:=\left\{\lambdab\in\R^d\,\mbox{such that}\,\sum_{j=1}^{d} |\lambda_j| \leq 1\right\}$ for the domain of the skewness parameter.
For the sake of readability, we will omit writing $\Upsilonb$ in the developments that follow. 
In this paper, we require $f_0$ to be a component-wise $2\pi$-periodic, unimodal, symmetric density, ensuring that the symmetry center $\mub$ is uniquely defined. 
More formally, we require $f_0 \in \mathcal{F}$, where $\mathcal{F}$ is defined as
\begin{align}\label{eq: def: F}
    \mathcal{F} := \left\{\begin{array}{ll}
& f_0(\thetab) > 0 \ \forall \thetab \in [-\pi, \pi)^d, \\
& f_0(\theta_1 + 2\pi k_1, \ldots, \theta_d + 2\pi k_d) = f_0(\theta_1, \ldots, \theta_d) \ \forall k_i\in\Z, i=1,\ldots,d,\\
f_0(\thetab): & f_0(\thetab) = f_0(-\thetab)\ \forall \thetab \in [-\pi, \pi)^d, \\
& f_0 \text{ unimodal in } \thetab \in [-\pi,\pi)^d \text{ with a mode at } \boldsymbol{0}, \\
& \int_{[-\pi,\pi)^d} f_0(\thetab) {\mathrm d} \thetab = 1.
\end{array} \right\}.
\end{align}
Most known models from the literature, such as the  Sine, Cosine, bivariate wrapped Cauchy, trivariate wrapped Cauchy with specified marginals, or the multivariate von Mises, satisfy this requirement under certain restrictions on the parameter space to ensure unimodality. Whenever $\lambdab=\pmb0$, we retrieve the original symmetric density $f_0$, otherwise the resulting density is sine-skewed. This naturally leads us to test the null hypothesis $\lambdab=\pmb0$ against $\lambdab\neq\pmb0$. Depending on whether the center of symmetry is known and on whether we assume $f_0$ to be known or not, we will have a distinct notation for the hypothesis testing problem, which we will introduce in Sections~\ref{sec:knowncenter} and~\ref{sec:unknowncenter}.

\subsection{ULAN property}\label{sec:ULAN}

As already mentioned, we are interested in the symmetry  hypothesis testing problem, that is, testing the null hypothesis $\lambdab=\zerob$ against the alternative $\lambdab\neq\zerob$. 
We derive four tests, depending on what we consider to be a nuisance parameter.
Parametric tests are derived in the cases that the symmetry center is known (no nuisance parameters) and unknown ($\mub$ is a nuisance parameter).
Similarly, semi-parametric tests are also derived, when the symmetry center is known ($f_0$ is a nuisance parameter) and unknown (both $f_0$ and $\mub$ are nuisance parameters).

We denote by $P^{(n)}_{\mub, \lambdab; f_0}$, with $f_0 \in \mathcal{F}$ as in \eqref{eq: def: F}, the joint distribution of a sequence $\{ \thetab_i \}_{i=1}^n = \{(\theta_{1i}, \ldots, \theta_{di})' \}_{i=1}^n$ of independent and identically distributed (iid) hyper-toroidal random observations with density \eqref{eq: sine_skewed_density}. Any $f_0$ then induces a toroidal location-skewness model 
$$
P^{(n)}_{f_0}:=\left\{P^{(n)}_{\mub, \lambdab; f_0}:\mub\in[-\pi,\pi)^d,\lambdab\in S^d\right\}.
$$
In order to derive tests for symmetry that are optimal in the Le Cam sense, we will establish the ULAN property in the vicinity of symmetry, meaning at $\lambdab=\pmb0$,  of the parametric model $P^{(n)}_{f_0}$. 
The ULAN property means that the distributions not only enjoy the local asymptotic normality (LAN) property, but also asymptotic linearity. 
For the ULAN property, we require the following mild assumption which is related to our general proposition from Section~\ref{sec:QMD}.

\stepcounter{assumption}

\begin{intassumption} \label{assumption_on_f0_a}
The mapping $\thetab \mapsto f_{0}^{1/2}\left( \thetab\right)$ is QMD over $\left[ -\pi, \pi \right)^d$ with quadratic mean or weak vector derivative $({f_{0}^{1/2})}\dot{}\left( \thetab\right)$ and, letting ${\pmb\psi}^{f_0}(\thetab)=-2\frac{(f_{0}^{1/2})\dot{}\left( \thetab\right)}{f_{0}^{1/2}\left( \thetab\right)}$, the Fisher information matrix for location ${\pmb I}^{f_0}_{\mub\mub}= \int_{[-\pi,\pi)^d}  {\pmb\psi}^{f_0}(\thetab)({\pmb\psi}^{f_0}(\thetab))^{\prime}  f_{0}\left( \thetab\right) {\mathrm d}\thetab$ has finite elements.
\end{intassumption}
All our subsequent developments work under this mild assumption of weak differentiability. However, in view of notational simplicity and since all existing models from the literature actually possess regular derivatives, we shall henceforth work under the slightly stronger Assumption \ref{assumption_on_f0} below.

\begin{intassumption} \label{assumption_on_f0}
The mapping $\thetab \mapsto f_{0}^{1/2}\left( \thetab\right)$ is continuously differentiable  over $\left[ -\pi, \pi \right)^d$ with continuous derivative
\begin{equation*}
    \nabla_{\mub} f^{1/2}_{0}\left( \thetab - \mub \right) 
    = \begin{pmatrix}
        \frac{\frac{\partial} {\partial\mu_1}  f_{0}\left( \thetab - \mub \right) } {2 f^{1/2}_{0}\left( \thetab - \mub \right) } \\
        \vdots \\
        \frac{\frac{\partial} {\partial\mu_d}  f_{0}\left( \thetab - \mub \right) } {2 f^{1/2}_{0}\left( \thetab - \mub \right) }
    \end{pmatrix}
    = \frac{1}{2} \begin{pmatrix}
         \phi^{f_0}_1(\thetab - \mub) \\
         \vdots \\
         \phi^{f_0}_d(\thetab - \mub)
    \end{pmatrix} f^{1/2}_{0}\left( \thetab - \mub \right),
\end{equation*}
where 
\begin{equation} \label{eq: def: phi}
    \phi^{f_0}_j(\thetab - \mub) 
    = -\frac{\frac{\partial}{\partial \theta_j}f_0(\thetab -\mub)}{f_0(\thetab - \mub)}, \textrm{ for } j \in \{ 1, \ldots, d \}.
\end{equation}
\end{intassumption}

From Theorem~12.2.2 of \cite{lehmann_quadratic_2005}, quadratic mean differentiability of $\thetab \mapsto f_{0}^{1/2}\left( \thetab\right)$ is implied by continuous differentiability.  Note that ${\psi}^{f_0}_i(\thetab)$ in Assumption \ref{assumption_on_f0_a} simplifies to $\phi^{f_0}_i(\thetab)$ under Assumption~\ref{assumption_on_f0}. Moreover,  the Fisher Information elements $I^{f_0}_{\mu_i\mu_i} = \int_{[-\pi,\pi)^d} \left( \phi^{f_0}_i(\thetab) \right)^2 f_{0}\left( \thetab\right) {\mathrm d}\thetab$ are finite for $i\in\{1,\ldots, d\}$ by the assumption of continuous differentiability over a bounded domain. As already said, all known models from the literature satisfy Assumption~\ref{assumption_on_f0}. To avoid any misunderstandings, we provide a general definition of QMD in Definition~\ref{def: QMD} of the Supplementary Material. In view of the latter, it becomes clear that Assumption~\ref{assumption_on_f0}  leads to quadratic mean differentiability of $P^{(n)}_{f_0}$ with respect to the location parameter at $\lambdab=\pmb0$. As we shall see in the proof of Proposition~\ref{prop: ULAN}, we need to establish QMD with respect to both location and skewness parameters to obtain the ULAN property. This will be a direct consequence of our general Proposition~\ref{prop: QMD}.

Let $\varthetab =(\mub^\prime,\lambdab^\prime)^\prime$ and $\varthetab_0=(\mub^\prime, \boldsymbol{0}^\prime)^\prime$. We now state the ULAN property of the family $P^{(n)}_{f_0}$ in the vicinity of symmetry, that is, at $\varthetab_0$.
\begin{proposition}\label{prop: ULAN}
Let $f_0 \in \mathcal{F}$ satisfy Assumption~\ref{assumption_on_f0}. 
Then the sine-skewed family $P^{(n)}_{f_0}$ is ULAN at $\varthetab_0$. More precisely, for any $\mu_i^{(n)}=\mu_i + O(n^{-1/2}), i=1,\ldots,d,$ and  for any bounded sequence $\taub^{(n)}:=(\taub^{(n)\prime}_{\mu}, \taub^{(n)\prime}_{\lambda})'\in\R^{2d}$ with $\taub^{(n)}_{\mu} = (\tau^{(n)}_1, \ldots, \tau^{(n)}_d)'\in\R^{d}$, $\taub^{(n)}_{\lambda} = (\tau^{(n)}_{d+1}, \ldots, \tau^{(n)}_{2d})'\in\R^{d}$ such that $\mu^{(n)}_i + n^{-1/2} \tau^{(n)}_i$, $i=1, \ldots, d,$ remains in $[-\pi, \pi)$, and $n^{-1/2}\tau^{(n)}_j \in [-1,1]$, $j=d+1, \dots, 2d$ and $\sum_{j=d+1}^{2d} |\tau^{(n)}_j|  < n^{1/2}$, we have
\begin{align} \label{eq: def: loglikelihood ULAN}
    \Lambda^{(n)} 
    & := \log \Biggl( \frac{P^{(n)}_{\mub^{(n)} + n^{-1/2} \taub^{(n)}_{\mu}, n^{-1/2} \taub^{(n)}_{\lambda}; f_0} }{P^{(n)}_{\mub^{(n)}, \boldsymbol{0}; f_0}} \Biggr) = \taub^{(n)}{'} \Delta_{f_0}^{(n)}(\mub^{(n)}) - \frac{1}{2} \taub^{(n)}{'} \Gamma_{f_0} \taub^{(n)} + o_P(1)
\end{align}
as $n\to\infty$, and the central sequence
\begin{equation} \label{eq: def: central sequence}
    \Delta_{f_0}^{(n)}(\mub) 
    := \begin{pmatrix}
        \Delta_{f_0;\mu}^{(n)}(\mub)\\
        \Delta_{\lambda}^{(n)}(\mub)
    \end{pmatrix} =\frac{1}{\sqrt{n}}\sum_{i=1}^n \Delta_{f_0; i}^{(n)}(\mub) 
    = \frac{1}{\sqrt{n}}\sum_{i=1}^n \begin{pmatrix}
        \phi^{f_0}_1(\thetab_{i} - \mub) \\
        \vdots \\
        \phi^{f_0}_d(\thetab_{i} - \mub) \\
        \sin(\theta_{1i}-\mu_1) \\
        \vdots \\
        \sin(\theta_{di}-\mu_d) 
    \end{pmatrix},
\end{equation}
for $\phi^{f_0}_j(\thetab-\mub)$ as in \eqref{eq: def: phi},
converges to $\mathcal{N}_{2d}(\zerob,\Gamma_{f_0})$ under $P^{(n)}_{\mub^{(n)}, \zerob; f_0}$ as $n\to\infty$, where the Fisher information matrix $\Gamma_{f_0}$ is given by
\begin{equation} \label{eq: def: Fisher Information}
    \Gamma_{f_0} = 
    \begin{pmatrix}
        I^{f_0}_{\mu_1\mu_1} & \hdots & I^{f_0}_{\mu_1\mu_d} & I^{f_0}_{\mu_1\lambda_1} & \hdots & I^{f_0}_{\mu_1\lambda_d} \\
        \vdots & \ddots & \vdots & \vdots & \ddots & \vdots \\
        I^{f_0}_{\mu_1\mu_d} & \hdots & I^{f_0}_{\mu_d\mu_d} & I^{f_0}_{\mu_d\lambda_1} & \hdots & I^{f_0}_{\mu_d\lambda_d} \\
        I^{f_0}_{\mu_1\lambda_1} & \hdots & I^{f_0}_{\mu_d\lambda_1} & I^{f_0}_{\lambda_1\lambda_1} & \hdots & I^{f_0}_{\lambda_1\lambda_d} \\
        \vdots & \ddots & \vdots & \vdots & \ddots & \vdots \\
        I^{f_0}_{\mu_1\lambda_d} & \hdots & I^{f_0}_{\mu_d\lambda_d} & I^{f_0}_{\lambda_1\lambda_d} & \hdots & I^{f_0}_{\lambda_d\lambda_d}
    \end{pmatrix}
\end{equation}
with
    \begin{align} \label{eq: def: I}
        & I^{f_0}_{\mu_j\mu_k} = \int_{[-\pi,\pi)^d} \phi_j^{f_0}(\boldsymbol{\theta}) \phi_k^{f_0}(\boldsymbol{\theta}) f_0(\boldsymbol{\theta}) {\mathrm d}\boldsymbol{\theta}, \qquad  I^{f_0}_{\lambda_j\lambda_k} = \int_{[-\pi,\pi)^d} \sin(\theta_j) \sin(\theta_k) f_0(\boldsymbol{\theta}) {\mathrm d}\boldsymbol{\theta}, \\
        & I^{f_0}_{\mu_j\lambda_j} = \int_{[-\pi,\pi)^d} \phi_j^{f_0}(\boldsymbol{\theta} ) \sin(\theta_j) f_0(\boldsymbol{\theta}) {\mathrm d}\boldsymbol{\theta}, \qquad I^{f_0}_{\mu_j\lambda_k} = 0, \text{ for } j\neq k, 
    \end{align}
where $j,k \in \{1,\ldots,d\}$.
\end{proposition}
The proof of the ULAN property relies on Lemma 2.3 of \cite{garel_local_1995}, which is a modification of Lemma 1 of \cite{swensen_asymptotic_1985} and is provided in Section \ref{appendix: proof_ULAN} of the Supplementary Material. It is straightforward to see that the skewness parts of the Fisher information matrix are finite by bounding $|\sin(\cdot)|$ by 1; the same holds true for the location-skewness parts thanks to our assumptions on $f_0$. In the next two sections we will derive our optimal tests for symmetry for both the case of known and unknown centers thanks to this ULAN property.

The ULAN property as outlined in Proposition~\ref{prop: ULAN} also holds for the following (unexplored in the literature) skewed family, with alternative skewing transformation
\begin{align}\label{eq: sine_skewed_density_new_proposal}
    \thetab\mapsto f_{\mub, \lambdab}(\thetab; \varthetab) := f_0(\thetab - \mub; \varthetab) \prod_{j=1}^{d} \left(1 + \lambda_j \sin (\theta_j-\mu_j)\right)
\end{align} 
subject to $|\lambda_j| \leq 1, j = 1,\ldots,d$.
This implies that all results derived in the next sections of the paper hold for \eqref{eq: sine_skewed_density_new_proposal}.
A different skewing mechanism was proposed by \cite{bekker2022} on the 2-dimensional torus.
The proposed model is of the form
\begin{equation*}
    \thetab\mapsto f_{\mub, \lambdab}(\thetab; \varthetab) := C_{m, f} f_0(\theta_1, \theta_2; \varthetab)\left(\frac{1 + \lambda_1 \sin (\theta_1) + \lambda_2 \sin (\theta_2)}{2}\right) ^m
\end{equation*}
for a normalizing constant $C_{m, f}$ and a positive integer $m$, with central sequence given by
\begin{equation*}
    \left(
        \phi^{f_0}_1(\thetab),
        \phi^{f_0}_2(\thetab),
        m\sin(\theta_{1}),
        m\sin(\theta_{2}) 
        \right)'.
\end{equation*}
Even though the ULAN property for this family is slightly different, the final test statistics that are calculated in Sections~\ref{sec:knowncenter} and \ref{sec:unknowncenter} admit the same expression as for \eqref{eq: def: central sequence}, for $d=2$, and so all the results of the paper hold for this model as well.
These observations imply that our tests are optimal against a broader family of alternatives.

\section{Optimal tests for symmetry about a known  center} \label{sec:knowncenter}

In this section, we assume that the center of symmetry  is known and  fixed to $\mub \in [-\pi, \pi)^d$.
Our goal is to derive semi-parametrically optimal tests for the following hypothesis testing problem:
\begin{equation}\label{eq: hypotheses known median}
    \mathcal{H}^{(n)}_{0; \mu}:= \bigcup_{f_0 \in\mathcal{F}} P^{(n)}_{\mub, \boldsymbol{0}; f_0}
    \quad \text{vs} \quad
    \mathcal{H}^{(n)}_{1; \mu}:= \bigcup_{\lambdab \in S^d \setminus \{\boldsymbol{0} \} } \bigcup_{f_0 \in\mathcal{F}} P^{(n)}_{\mub, \lambdab; f_0}.
\end{equation}
Since $\mub$ is known, we focus on the parts of the central sequence and Fisher Information matrix that only depend on $\lambdab$. We recall that
\begin{equation}\label{eq: def: central sequence and FI for known median}
    \Delta_{\lambda}^{(n)}(\mub) = \frac{1}{\sqrt{n}}\sum_{i=1}^n 
    \begin{pmatrix}
        \sin(\theta_{1i}-\mu_1) \\
        \vdots \\
        \sin(\theta_{di}-\mu_d)
    \end{pmatrix} 
    \quad \text{and define} \quad
    \Gamma_{f_0; \lambda} =  
    \begin{pmatrix}
        I^{f_0}_{\lambda_1\lambda_1} & \hdots & I^{f_0}_{\lambda_1\lambda_d} \\
        \vdots & \ddots & \vdots \\
        I^{f_0}_{\lambda_1\lambda_d} & \hdots & I^{f_0}_{\lambda_d\lambda_d}
    \end{pmatrix}
\end{equation}
where $I^{f_0}_{\lambda_i\lambda_j}$ for $i,j=1,\ldots,d$ is defined in \eqref{eq: def: I}.
Using the ULAN property established in Proposition~\ref{prop: ULAN}, we can establish optimal (in the Le Cam sense) parametric tests for symmetry under fixed $f_0$, that is, for the problem $\mathcal{H}^{(n)}_{0; \mu, f_0}:= P^{(n)}_{\mub, 0; f_0}$ vs $\mathcal{H}^{(n)}_{1; \mu, f_0}:= \bigcup_{\lambdab \in S^d \setminus \{\boldsymbol{0} \} } P^{(n)}_{\mub, \lambdab; f_0}$. To this end, we build the $f_0$-parametric test $\phi^{(n); \mu}_{f_0}$ that rejects $\mathcal{H}^{(n)}_{0; \mu, f_0}$ at asymptotic level $\alpha$ whenever the statistic
\begin{equation} \label{eq: parametric test specified}
    Q_{f_0}^{(n); \mub} := \left(\Delta_{\lambda}^{(n)}(\mub)\right)' \left(\Gamma_{f_0; \lambda} \right)^{-1} \Delta_{\lambda}^{(n)}(\mub)
\end{equation}
exceeds $\chi^2_{d; \alpha}$,  the $\alpha$-upper quantile of the $\chi^2$ distribution with $d$ degrees of freedom.
From the Le Cam theory, it follows that this test is locally and asymptotically maximin for testing the null $\mathcal{H}^{(n)}_{0; \mu, f_0}$ against the alternative $\mathcal{H}^{(n)}_{1; \mu, f_0}$.
However, this holds only when the true density of the data, $f_0$, is known, and does not hold when it is misspecified. So, we consider the studentized version of the test, denoted by $\phi^{\ast (n); \mu}$, that rejects the much broader null hypothesis $\mathcal{H}^{(n)}_{0; \mu}$ against the alternative $\mathcal{H}^{(n)}_{1; \mu}$ at asymptotic level $\alpha$ whenever the statistic
\begin{equation} \label{eq: semi-parametric test specified}
    Q^{\ast(n); \mub} = \left(\Delta_{\lambda}^{(n)}(\mub)\right)' \left(\hat{\Gamma}_{\lambda} \right)^{-1} \Delta_{\lambda}^{(n)}(\mub)
\end{equation}
exceeds $\chi^2_{d; \alpha}$ (an asymptotic result which we formally establish in Theorem~\ref{thm: optimality specified median}) where 
\begin{equation}\label{eq: Fisher Information specified estimate}
    \hat{\Gamma}_{\lambda} = 
    \begin{pmatrix}
        \hat{I}_{\lambda_1\lambda_1} & \hdots & \hat{I}_{\lambda_1\lambda_d} \\
        \vdots & \ddots & \vdots \\
        \hat{I}_{\lambda_1\lambda_d} & \hdots & \hat{I}_{\lambda_d\lambda_d}
    \end{pmatrix},
    \ \hat{I}_{\lambda_j\lambda_k} = \frac{1}{n}\sum_{i=1}^n \sin(\theta_{ji} - \mu_j) \sin(\theta_{ki} - \mu_k) \text{  for  } j,k=1,\ldots,d.
\end{equation}
The latter quantity is the empirical estimate of the Fisher Information matrix $\Gamma_{f_0; \lambda}$ defined in \eqref{eq: def: central sequence and FI for known median}. 
Calculating the inverse of it requires numerical methods for high dimensions. We attract the reader's attention to the interesting fact that $Q^{\ast(n); \mub}$ does \emph{not} depend on $f_0$ thanks to the estimator $\hat{\Gamma}_{\lambda}$. This is the reason why we omit $f_0$ as an index in both $Q^{\ast(n); \mub}$ and $\phi^{\ast (n); \mu}$, and implies that any $f_0$-parametric test leads to the same studentized test, a rare and very powerful result. Indeed, the test $\phi^{\ast (n); \mu}$ inherits the optimality properties from its parametric antecedents (see Theorem~\ref{thm: optimality specified median}) and, consequently, becomes a universally optimal test.
The asymptotic results concerning the test statistic \eqref{eq: semi-parametric test specified} are summarized in the following theorem, whose proof can be found in Section~\ref{appendix: proof optimality specified} of the Supplementary Material.

\begin{theorem} \label{thm: optimality specified median}
For $Q^{\ast(n); \mub}$ as defined in \eqref{eq: semi-parametric test specified}, the following hold under Assumption \ref{assumption_on_f0}:
    \begin{enumerate}
        \item[(i)] Under $\mathcal{H}^{(n)}_{0; \mu}$, we have that
        \begin{equation} \label{eq: asymptotic_result_specified}
            Q^{\ast(n); \mub}   \xrightarrow{\mathcal{D}} 
            \chi^2_d
        \end{equation}
        as $n \rightarrow \infty$, so that the test $\phi^{\ast (n); \mu}$, defined above \eqref{eq: semi-parametric test specified}, has asymptotic level $\alpha$ under the null.

        \item[(ii)] Under $P^{(n)}_{\mub, n^{-1/2} \taub^{(n)}_\lambda; f_0}$ for $f_0 \in \mathcal{F}$ and $\taub^{(n)}_\lambda$ as defined in Proposition~\ref{prop: ULAN}, it holds that
    \begin{equation}\label{eq: dist_under_alternative d dimensions}
        Q^{\ast(n); \mub} \xrightarrow{\mathcal{D}} 
        \chi^2_d(\taub'_\lambda \Gamma_{f_0; \lambda} \taub_\lambda)
    \end{equation}
     as $n \rightarrow \infty$, where $\chi^2_\nu(\kappa)$ denotes a non-central $\chi^2$ distribution with $\nu$ degrees of freedom and non-centrality parameter $\kappa$, with $\taub_\lambda = \lim_{n\rightarrow\infty} \taub^{(n)}_\lambda$.

     \item[(iii)] The test $\phi^{\ast (n); \mu}$ is universally (in $f_0$) locally and asymptotically maximin at level $\alpha$ when testing $\mathcal{H}^{(n)}_{0; \mu}$ against $\mathcal{H}^{(n)}_{1; \mu}$.
    \end{enumerate}
\end{theorem}
This theorem formally shows what we announced before, namely that we have a semi-parametric test that is valid under the entire null hypothesis $\mathcal{H}^{(n)}_{0; \mu}$ and uniformly (in $f_0$) optimal against the alternative $\mathcal{H}^{(n)}_{1; \mu}$. 
To these remarkable features we can add that our test statistic is extremely simple, as it is just based on trigonometric moments. 
Moreover, by point (ii), we are able to write down the explicit asymptotic power of the test against the local alternatives $P^{(n)}_{\mub, n^{-1/2} \boldsymbol{\tau}^{(n)}_\lambda; f_0}$ as $n \rightarrow \infty$, resulting in the expression $\Prob\Bigl( Q \geq \chi^2_{d;\alpha} \Bigr)$ for $Q \sim \chi^2_d(\boldsymbol{\tau}'_\lambda \Gamma_{f_0; \lambda} \boldsymbol{\tau}_\lambda)$.
For $\kappa = \boldsymbol{\tau}'_\lambda \Gamma_{f_0; \lambda} \boldsymbol{\tau}_\lambda$, the expression of the asymptotic power can be equivalently written as
\begin{equation*}
    \Prob\Bigl( Q \geq \chi^2_{d;\alpha} \Bigr) = 1 - e^{-\kappa/2} \sum_{j=0}^\infty \frac{\left( \kappa/2 \right)^j}{j!}\frac{\gamma \left(d/2+j, \chi^2_{d;\alpha}/2 \right)}{\Gamma \left( d/2+j \right)},
\end{equation*}
where $\gamma(s,t)$ is the lower incomplete gamma function.
As a special case, when $d=1$, we retrieve the results of \cite{ley_simple_2014}.

For any real life dataset, there will only be a limited number of observations available.
In the following theorem, working under the null hypothesis, $\mathcal{H}^{(n)}_{0; \mu}$, we derive a bound for the distance (integral probability metric) between $Q^{\ast(n); \mub}$ and its limiting chi-square distribution with $d$ degrees of freedom.
This gives information on how well the limiting distribution approximates the true distribution, in the case of finite observations.
The proof of the theorem is based on Stein's Method, and it can be found 
in Section~\ref{appendix: proof thm steins specified} of the Supplementary Material.
The notation $C^k_b(I)$ denotes the class of real-valued functions defined on $I\subset\R$ whose partial derivatives of order $k$ all exist and are bounded.

\begin{theorem} \label{thm: steins method specified}
For $i,j\in\{1,\ldots,d\}$, define  
$\gamma_{ij} = \left( \Gamma_{f_0; \lambda}^{-1} \right)_{ij}$ and $\hat{\gamma}_{ij} =  \left( \hat{\Gamma}_{\lambda}^{-1} \right)_{ij}$,
where $A_{ij}$ denotes the entry of a matrix $A$ in row $i$ and column $j$.
Then, 
for any $h \in C^6_b(\R)$, working under Assumption \ref{assumption_on_f0},
    \begin{align} \label{eq: steins method}
    \Bigl|\E & \left[ h(Q^{\ast(n); \mub}) \right] -\E \left[ h\left(\chi^2_d\right) \right] \Bigr| \nonumber \\
    & \leq 
    \frac{1}{n} \Biggl[ C_1 + 
    \lVert h^{(1)} \rVert
    \sum_{i=1}^d \sum_{j=1}^d \Biggl( \E \left[ \left( \hat{\gamma}_{ij} - \gamma_{ij} \right)^2 \right] \biggl( \ n \ \E \left[ \sin^2\left( \theta_{i1} - \mu_i \right) \sin^2\left( \theta_{j1} - \mu_j \right) \right] \nonumber \\
    & \hspace{1.5cm} + 2n(n-1) \Bigl( \E \left[ \sin\left( \theta_{i1} - \mu_i \right) \sin\left( \theta_{j2} - \mu_j \right) \right] \Bigr)^2 \nonumber\\
    & \hspace{1.5cm} + n(n-1) \E \left[ \sin^2\left( \theta_{i1} - \mu_i \right) \right] \E \left[ \sin^2\left( \theta_{j2} - \mu_j \right) \right]  \biggr) \Biggr)^{1/2} \Biggr], 
\end{align}
where $h^{(\ell)}$ denotes the $\ell^\text{th}$ derivative of $h$ and ${n} \, C_1$ is the {bound} as given in  Theorem 2.4 of \cite{gaunt_rate_2023}; more precisely, $C_1$ {is $O(1)$ and} depends on the absolute moments of up to $10^\textrm{th}$ order of $X_{lj}$, for $X_{lj} = \sum_{k=1}^{d} a_{kj}\sin(\theta_{kl} - \mu_k)$
for $l\in\{1,\ldots,n\}$ and $a_{ij} = \left( \left( \Gamma_{f_0; \lambda}^{-1} \right)^{1/2}\right)_{ij}$,
as well as on $h_4$ and $ h_6$,
with $h_m = \sum_{\ell=1}^m \begin{Bmatrix}
    m \\ \ell
\end{Bmatrix} \lVert h^{(\ell)} \rVert$, for $\begin{Bmatrix}
    m \\ \ell
\end{Bmatrix} = \frac{1}{\ell!} \sum_{l=0}^\ell \left( -1 \right)^{\ell-l} \begin{pmatrix}
    \ell \\ l
\end{pmatrix} l^m$ being the Stirling numbers of the second kind.
\end{theorem}

In order to be able to calculate the order of the bound in \eqref{eq: steins method}, we need to bound $\E \left[ \left( \hat{\gamma}_{ij} - \gamma_{ij} \right)^2 \right]$.
Using Lemma~\ref{lemma: rate of diff of gammas}, which is stated below, the order of the bound in \eqref{eq: steins method} is $O\left( \frac{1}{\sqrt{n}} \right)$.
This result not only quantifies the distance between $Q^{\ast(n); \mub}$ and its limiting chi-square distribution with $d$ degrees of freedom, but also serves as an alternative proof of Theorem~\ref{thm: optimality specified median}(i),
{as the class of functions we are considering, $C^6_b(\R)$, is rich enough to indicate convergence in distribution.}
\begin{lemma}\label{lemma: rate of diff of gammas}
    For $i,j \in \{1,\ldots,d\}$ and $\gamma_{ij}$ and $\hat{\gamma}_{ij}$ as defined in Theorem~\ref{thm: steins method specified}, it holds that 
    \begin{equation}\label{eq: rate of diff of gammas}
        \E \left[ \left( \hat{\gamma}_{ij} - \gamma_{ij} \right)^2 \right] \leq C_2/n
    \end{equation} 
    for a universal constant $C_2$, which depends on the true underlying distribution $f_0$.   
\end{lemma}
The constant $C_2$ in Lemma~\ref{lemma: rate of diff of gammas} cannot be calculated explicitly, which is a consequence of multiplying with the inverse of $\hat{\Gamma}_{\lambda}$ in the expression of $Q^{\ast(n); \mub}$.
In the proof, this can be seen from \eqref{eq: steins_method_taylor_expansion_determinant} where the bound can only be derived with a universal constant because of the last term of the equation.
The proof of the lemma is deferred to Section~\ref{appendix: proof lemma steins specified} of the Supplementary Material.
Some results of the proof might be of independent interest, in particular the formulation of the determinant in \eqref{eq:d_ab} and the derivation of the expected rate using \eqref{eq: set_c}.

{
\begin{remark}
    The bound $C_1$ can be calculated explicitly for specified densities $f_0$.
    Due to the complexity of the calculation, numerical calculation of the integrals is performed.
    It can be calculated that
    \begin{align*}
        C_1 & = \frac{h_4}{24 n} \sum_{i=1}^n \Biggl( 1085.22 + \frac{1.71 \times 10^{-6}}{n^3} \\
        & + 2.31 \times 10^{-2} \left( \frac{1.30 \times 10^{-4}}{n^2} 
        + \frac{7.62 \times 10^{-2} (n-1)}{n^2} 
        + \frac{1.22 (n-1)(n-2)}{n^3} \right) \Biggr) \\
        & + \frac{h_6}{72 n^3} \sum_{\alpha=1}^n \sum_{i=1}^n \Biggl(301.01 + \frac{5.15 \times 10^{-9}}{n^3} \\
        & + 1.26 \times 10^{-3} \left( \frac{1.30 \times 10^{-4}}{n^2} 
        + \frac{7.62 \times 10^{-2} (n-1)}{n^2} 
        + \frac{1.22 (n-1)(n-2)}{n^3} \right) \Biggr)
    \end{align*}
    for $f_0$ being the bivariate wrapped Cauchy distribution with parameters $\xi_1 = \xi_2 = 0$ and $\rho = -0.99999$.
    It is worth emphasizing that the quantity $C_1$ is $O(1)$.
    The relevant code for this calculation can be found in \url{https://doi.org/10.5281/zenodo.17224137} and \url{https://github.com/Sophia-Loizidou/Symmetry_test_on_hypertorus}.
\end{remark}
}


\section{Optimal tests for unknown symmetry center} \label{sec:unknowncenter}

In this section, we derive semi-parametrically optimal tests for the case that the symmetry center is unknown; in formal terms we consider the testing problem
\begin{equation}\label{eq: hypotheses unknown median}
    \mathcal{H}_0^{(n)}:= \bigcup_{\mub \in [-\pi, \pi)^d} \bigcup_{g_0 \in\mathcal{F}} P^{(n)}_{\mub, \zerob; g_0}
    \quad \text{vs} \quad
    \mathcal{H}_1^{(n)}:= \bigcup_{\lambdab \in (-1,1)^d \setminus \{\boldsymbol{0} \} } \bigcup_{\mub \in [-\pi, \pi)^d} \bigcup_{g_0 \in\mathcal{F}} P^{(n)}_{\mub, \lambdab; g_0}.
\end{equation}
The main change compared to the previous section is the fact that we now need to estimate the nuisance parameter $\mub$ by some estimator $\hat{\mub}^{(n)}$. We will henceforth assume that such location estimators  satisfy the following assumption.
\begin{assumption} \label{assumption: mu_estimate}
    The estimators $\hat{\mub}^{(n)} = \left( \hat{\mu}^{(n)}_1, \ldots, \hat{\mu}^{(n)}_d \right)'$ for $\mub$ are 
    \begin{enumerate}
        \item[(i)] $\sqrt{n}$-consistent, i.e. $\sqrt{n} \left( \hat{\mub}^{(n)} - \mub \right) = O_P(1)$ 
        as $n \rightarrow \infty$ under $\cup_{g_0 \in \mathcal{F}} P^{(n)}_{\mub,\zerob;g_0}$, and 

    \item[(ii)] locally asymptotic discrete, i.e. for all $i\in\{1,\ldots,d\}$, $\mu_i \in [-\pi, \pi)$, 
    and $c > 0$ there exists $M = M(c) > 0$ such that the number of possible values of $\hat{\mu}^{(n)}_i$ in intervals of the form $\{t \in [-\pi, \pi) : \sqrt{n} |t - \mu_i| \leq c\}$,
    is bounded by $M$, uniformly as $n \rightarrow \infty$.
    \end{enumerate}
\end{assumption}
Assumption \ref{assumption: mu_estimate} will allow us to use Lemma 4.4 from \cite{kreiss_adaptive_1987}, which is an integral step in obtaining all asymptotic results provided in this section. Part (i) is a typical assumption on estimators, while Part (ii) is purely a technical requirement, with little practical
implication. Indeed, for a fixed sample size, any estimator can be considered part of a
locally asymptotically discrete sequence: see, for instance, \cite{le2000asymptotics}.
At a first glance, this assumption seems to not respect the periodic nature of the data.
However, since  the estimators and location parameter de facto lie on the surface of the torus, we can simply choose the starting point of our intervals on a different place such that the periodicity is not violated.


For an estimator $\hat{\mub}^{(n)}$ that satisfies Assumption~\ref{assumption: mu_estimate}, using the ULAN property of Proposition \ref{prop: ULAN} and Lemma 4.4 from \cite{kreiss_adaptive_1987}, we get the following asymptotic linearity property, under any distribution $f_0\in\mathcal{F}$:
\begin{equation}\label{eq: asymptotic linearity}
    \begin{pmatrix}
        \Delta_{f_0; \mu}^{(n)}(\hat{\mub}^{(n)}) \\
        \Delta_{\lambda}^{(n)}(\hat{\mub}^{(n)})
    \end{pmatrix} = 
    \begin{pmatrix}
        \Delta_{f_0;\mu}^{(n)} (\mub) \\
        \Delta_{\lambda}^{(n)} (\mub) 
    \end{pmatrix} - \Gamma_{f_0} 
    \begin{pmatrix}
        \sqrt{n} (\hat{\mub}^{(n)} - \mub ) \\ \boldsymbol{0}
    \end{pmatrix} + o_P(1) \one_{2d\times 1}
\end{equation}
where $\one_{n\times m}$ is the $n\times m$ matrix of ones.
From \eqref{eq: asymptotic linearity} and the non-diagonality of the Fisher information matrix $\Gamma_{f_0}$, it is clear that the asymptotic cost of estimating  $\mub$ is not zero.
This fact needs to be taken into account when constructing the test.
We do this by calculating a new central sequence for skewness, the efficient central sequence $\Delta_{f_0;\lambda}^{(n)\ast} (\mub)$, which is orthogonal to $\Delta_{f_0;\mu}^{(n)} (\mub)$ under $f_0$.

By orthogonal projection, the efficient central sequence for skewness can be calculated as
\begin{align} \label{eq: new score function}
     \Delta_{f_0;\lambda}^{(n)\ast} (\mub)
     & = \Delta_{\lambda}^{(n)}(\mub) - \cov_{f_0}\left( \Delta_{\lambda}^{(n)}(\mub), \Delta_{f_0;\mu}^{(n)}(\mub) \right) \var_{f_0} \left( \Delta_{f_0;\mu}^{(n)} (\mub)\right)^{-1} \Delta_{f_0;\mu}^{(n)}(\mub)  
\end{align}
where
\begin{equation*}
    \cov_{f_0} \left( \Delta_{\lambda}^{(n)}(\mub), \Delta_{f_0;\mu}^{(n)}(\mub) \right) = \begin{pmatrix}
        I^{f_0}_{\mu_1\lambda_1} & \cdots & 0 \\
        \vdots & \ddots & \vdots \\
        0 & \cdots & I^{f_0}_{\mu_d\lambda_d}
    \end{pmatrix}, \ 
    \var_{f_0} \left( \Delta_{f_0;\mu}^{(n)} (\mub)\right) 
    = \begin{pmatrix}
        I^{f_0}_{\mu_1\mu_1} & \cdots & I^{f_0}_{\mu_1\mu_d} \\
        \vdots & \ddots & \vdots \\
        I^{f_0}_{\mu_1\mu_d} & \cdots & I^{f_0}_{\mu_d\mu_d}
    \end{pmatrix}
\end{equation*}
and in our notation we emphasize the fact that both $\var(\cdot)$ and $\cov(\cdot)$ are calculated with respect to the density $f_0$.
This calculation can also be viewed as finding new expressions for $\mub$.
In order to illustrate this, we consider the case $d=2$.
The exact expression of $\Delta_{f_0;\lambda}^{(n)\ast} (\mub)$ in this case becomes
\begin{equation} \label{eq: new central sequence parametric}
    \Delta_{f_0;\lambda}^{(n)\ast} (\mub)
    = \frac{1}{\sqrt{n}}\sum_{i=1}^n
    \begin{pmatrix}
        \sin(\theta_{1i}-\mu_1) - \frac{I^{f_0}_{\mu_1\lambda_1}I^{f_0}_{\mu_2\mu_2} \phi^{f_0}_1(\thetab_{i} - \mub) -I^{f_0}_{\mu_1\lambda_1}I^{f_0}_{\mu_1\mu_2} \phi^{f_0}_2(\thetab_{i} - \mub)}{d^{f_0}_{\mu_1\mu_2}}
        \\
        \sin(\theta_{2i}-\mu_2) - \frac{-I^{f_0}_{\mu_2\lambda_2} I^{f_0}_{\mu_1\mu_2} \phi^{f_0}_1(\thetab_{i} - \mub)  +
        I^{f_0}_{\mu_2\lambda_2} I^{f_0}_{\mu_1\mu_1} \phi^{f_0}_2 (\thetab_{i} - \mub) }{d^{f_0}_{\mu_1\mu_2}} 
    \end{pmatrix},
\end{equation}
where $d^{f_0}_{\mu_1\mu_2} = I^{f_0}_{\mu_1\mu_1} I^{f_0}_{\mu_2\mu_2} - \left( I^{f_0}_{\mu_1\mu_2} \right)^2$. 
The expression in \eqref{eq: new central sequence parametric} is the score function of observations $\{\theta_{1i}, \theta_{2i}\}_{i=1}^n$ with density 
\begin{equation*}
    f_0 \Bigl( \theta_1 - \mu_1(\lambda_1, \lambda_2), \theta_2 - \mu_2(\lambda_1, \lambda_2)\Bigr)(1+\lambda_1 \sin(\theta_1-\mu_1(\lambda_1, \lambda_2)) + \lambda_2 \sin(\theta_2 - \mu_2(\lambda_1, \lambda_2))),
\end{equation*}
where $\mu_1(\lambda_1, \lambda_2) = (I^{f_0}_{\mu_1\lambda_1} I^{f_0}_{\mu_2\mu_2} \lambda_1 - I^{f_0}_{\mu_2\lambda_2} I^{f_0}_{\mu_1\mu_2} \lambda_2)/d_{\mu_1\mu_2}$ and \linebreak $\mu_2(\lambda_1, \lambda_2) = (-I^{f_0}_{\mu_1\lambda_1} I^{f_0}_{\mu_1\mu_2} \lambda_1 + I^{f_0}_{\mu_2\lambda_2}I^{f_0}_{\mu_1\mu_1} \lambda_2) / d_{\mu_1\mu_2}$.

Going back to the derivation of the parametric test for symmetry, we have
\begin{align*}
    & \var_{f_0} \left( \Delta_{f_0; \lambda}^{(n)\ast}(\mub) \right)  = \ \var_{f_0} \left( \Delta_{\lambda}^{(n)}(\mub)  \right) \\
    & \hspace{1.5cm} - \cov_{f_0} \left( \Delta_{\lambda}^{(n)}(\mub), \Delta_{f_0; \mu}^{(n)}(\mub)  \right) \var_{f_0} \left( \Delta_{f_0;\mu}^{(n)}(\mub) \right)^{-1} \cov_{f_0} \left( \Delta_{\lambda}^{(n)}(\mub), \Delta_{f_0; \mu}^{(n)}(\mub) \right)
\end{align*}
and, replacing the unknown $\mub$ with an estimator satisfying Assumption~\ref{assumption: mu_estimate}, the $f_0$-parametric test $\phi^{(n)}_{f_0}$ for symmetry rejects $\mathcal{H}^{(n)}_{0; f_0}:= \bigcup_{\mub \in [-\pi, \pi)^d} P^{(n)}_{\mub, \zerob; f_0}$ at asymptotic level $\alpha$ whenever the statistic
\begin{equation} \label{eq: parametric test unspecified}
    Q_{f_0}^{(n)} := \left(\Delta_{f_0; \lambda}^{(n)\ast}(\hat{\mub}^{(n)})  \right)' \left(\var_{f_0} \left( \Delta_{f_0; \lambda}^{(n)\ast}(\hat{\mub}^{(n)}) \right) \right)^{-1} \Delta_{f_0; \lambda}^{(n)\ast}(\hat{\mub}^{(n)})
\end{equation}
exceeds $\chi^2_{d; \alpha}$.
This asymptotic result can be proven using Lemma 4.4 from \cite{kreiss_adaptive_1987} and the fact that, 
under $\mathcal{H}^{(n)}_0$,
$\Delta_{f_0; \lambda}^{(n)\ast}(\mub) \xrightarrow{\mathcal{D}} \mathcal{N}_d \left( \boldsymbol{0}, \var_{f_0} \left( \Delta_{f_0; \lambda}^{(n)\ast}(\mub) \right) \right)
$ as $n\rightarrow\infty$.

In order to derive the parametric test in \eqref{eq: parametric test unspecified}, the efficient central sequence was calculated under the assumption that the true underlying distribution of $\{\thetab_i\}_{i=1}^n$ is $f_0$.
Now, we want to derive a more generally valid  test, and consider to that end any density $g_0 \in \mathcal{F}$ that satisfies the same assumptions as $f_0$ and for the rest of this section we assume that $g_0$ is the true underlying distribution of $\{\thetab_i\}_{i=1}^n$. 
The orthogonalization as was done in \eqref{eq: new score function} needs to be adjusted accordingly, which is done as follows:
\begin{align}\label{eq: param test unspecified mu f,g}
    \Tilde{\Delta}_{f_0;g_0;\lambda}^{(n)} (\mub)
     & = \Delta_{\lambda}^{(n)}(\mub) - C^{g_0}_{\mu; \lambda} \left( C^{f_0; g_0}_{\mu; \mu} \right) ^{-1} \Delta_{f_0; \mu}^{(n)}(\mub),
\end{align}
where 
\begin{equation} \label{eq: def: cov matrix f0 g0}
    C^{f_0; g_0}_{\mu; \mu}
    := \cov_{g_0} \left(\Delta_{f_0; \mu}^{(n)}(\mub), \Delta_{g_0; \mu}^{(n)}(\mub) \right) 
    = \begin{pmatrix}
        I^{f_0;g_0}_{\mu_1\mu_1} & \cdots & I^{f_0;g_0}_{\mu_1\mu_d} \\
        \vdots & \ddots & \vdots \\
        I^{f_0;g_0}_{\mu_d\mu_1} & \cdots & I^{f_0;g_0}_{\mu_d\mu_d}
    \end{pmatrix}
\end{equation}
for $I^{f_0;g_0}_{\mu_j\mu_k} = \int_{[-\pi,\pi)^d} \phi_j^{f_0}(\thetab - \mub) \phi_k^{g_0}(\thetab - \mub) g_0(\thetab - \mub) {\mathrm d} \thetab$, $j,k\in\{1,\ldots,d\}$, and 
\begin{equation} \label{eq: def: cov matrix mu lambda}
    C^{g_0}_{\mu; \lambda}
    := \cov_{g_0} \left( \Delta_{\lambda}^{(n)}(\mub), \Delta_{g_0; \mu}^{(n)}(\mub) \right) 
    = \begin{pmatrix}
        I^{g_0}_{\mu_1\lambda_1} & \cdots & 0 \\
        \vdots & \ddots & \vdots \\
        0 & \cdots & I^{g_0}_{\mu_d\lambda_d}
    \end{pmatrix}
\end{equation}
for $I^{g_0}_{\mu_j\lambda_j} = \int_{[-\pi,\pi)^d} \sin(\theta_j - \mu_j) \phi_j^{g_0}(\thetab - \mub) g_0(\thetab - \mub) {\mathrm d} \thetab$, $j\in\{1,\ldots,d\}$.
Note that \eqref{eq: def: cov matrix f0 g0} is a symmetric matrix, something that can easily be shown by integration by parts. In addition, \eqref{eq: def: cov matrix mu lambda} is a diagonal matrix since, as given in \eqref{eq: def: I}, $I^{f_0}_{\mu_j\lambda_k} = 0$ for $j\neq k$.

Since $g_0$ is unknown, we need a semi-parametric version of the central sequence, meaning in particular that we need to estimate the aforementioned covariances. Integrating by parts and using the periodicity of $f_0$ and $g_0$, we obtain
\begin{align}
    I^{g_0}_{\mu_j\lambda_j} & \ = \int_{[-\pi,\pi)^d} \sin(\theta_j-\mu_j)\phi_j^{g_0}(\thetab - \mub) g_0(\thetab - \mub) {\mathrm d} \thetab \nonumber \\
    & \ = \int_{[-\pi,\pi)^d} \cos(\theta_j - \mu_j) g_0 (\thetab - \mub) {\mathrm d} \thetab \ = \ \E_{g_0}\left[ \cos(\theta_j - \mu_j) \right] \label{eq:Imulambda}
\end{align}
for $j\in\{1,\ldots,d\}$ and
\begin{align*}
    I^{f_0;g_0}_{\mu_j\mu_k} 
    & = \int_{[-\pi,\pi)^d} \phi_j^{f_0} (\thetab - \mub) \phi_k^{g_0} (\thetab - \mub) g_0 (\thetab - \mub) {\mathrm d} \thetab \\
    & = - \int_{[-\pi,\pi)^d} \phi_j^{f_0} (\thetab - \mub) \frac{\partial}{\partial \theta_k} g_0 (\thetab - \mub) {\mathrm d} \thetab \\
    & = - \int_{[-\pi,\pi)^{d-1}} \phi_j^{f_0} (\thetab - \mub) g_0 (\thetab - \mub) \Big\vert_{-\pi}^\pi {\mathrm d} \thetab_{(k)} + \int_{[-\pi,\pi)^d} \frac{\partial}{\partial \theta_k} \phi_j^{f_0} (\thetab - \mub) g_0 (\thetab - \mub) {\mathrm d} \thetab\\
    & = \E_{g_0}\left[ \frac{\partial}{\partial\theta_k} \phi_j^{f_0} (\thetab - \mub) \right],
\end{align*}
for any combination of $j,k\in\{1,\ldots,d\}$, where ${\mathrm d} \thetab_{(k)}$ denotes integrating with respect to $\theta_1,\ldots,\theta_{k-1},$ $\theta_{k+1},\ldots,\theta_d$.
In order to derive consistent estimators of the unknown quantities $I^{g_0}_{\mu_j\lambda_j}$ and $I^{f_0;g_0}_{\mu_j\mu_k}$, a further assumption is needed on $f_0$.

\begin{assumption} \label{assumption_on_f0_unspecified_median}
    The mapping $\thetab \mapsto f_0(\thetab - \mub)$ is $\mathcal{C}^2$ a.e. on $\left[-\pi, \pi \right)^d$.
\end{assumption}
Combining Assumption~\ref{assumption_on_f0_unspecified_median} with the fact that $ f_0(\thetab - \mub) > 0$ leads to $\phi_i^{f_0}(\thetab - \mub)$ being differentiable a.e. over $[-\pi, \pi)^d$.

\begin{proposition} \label{prop: estimate quantities}
    Suppose that $f_0, g_0 \in \mathcal{F}$ and that Assumptions \ref{assumption_on_f0}, \ref{assumption: mu_estimate} and \ref{assumption_on_f0_unspecified_median} hold. Then, for $j,k \in \{ 1,\ldots,d \}$, it holds that
    \begin{align}
    & I^{g_0}_{\mu_j\lambda_j} - \hat{I}_{\mu_j\lambda_j} = o_P(1), 
    \text{  for  }
    \hat{I}_{\mu_j\lambda_j} = \frac{1}{n}\sum_{i=1}^n \cos(\theta_{ji} - \hat{\mu}^{(n)}_j), \label{eq: prop estimate quantities eq1} \\
    & I^{f_0;g_0}_{\mu_j\mu_k} - \hat{I}^{f_0}_{\mu_j\mu_k} = o_P(1), 
    \text{  for  }
    \hat{I}^{f_0}_{\mu_j\mu_k} =  \frac{1}{n}\sum_{i=1}^n \frac{\partial}{\partial\theta_k}\phi_j^{f_0}(\thetab_{i} - \hat{\mub}^{(n)}) \label{eq: prop estimate quantities eq2}
\end{align}
as $n\rightarrow\infty$ under $P^{(n)}_{\mub, \boldsymbol{0}; g_0}$.
\end{proposition}
The proof of the proposition can be found in Section \ref{appendix: proof estimate quantities} of the Supplementary Material.
Using these estimators, the new efficient semi-parametric central sequence can be written as 
\begin{align}\label{eq: semi-param test unspecified mu f,g}
    \Tilde{\Delta}_{f_0; \lambda}^{(n)\ast} (\hat{\mub}^{(n)})
    = \Delta_{\lambda}^{(n)}(\hat{\mub}^{(n)}) - \widehat{C}_{\mu; \lambda}
    \left( \widehat{C}^{f_0}_{\mu; \mu} \right)^{-1} \Delta_{f_0; \mu}^{(n)}(\hat{\mub}^{(n)})
\end{align}
where 
\begin{equation} \label{eq: def: cov_estimated}
    \widehat{C}^{f_0}_{\mu; \mu} 
    = \begin{pmatrix}
        \hat{I}^{f_0}_{\mu_1\mu_1} & \cdots & \hat{I}^{f_0}_{\mu_1\mu_d} \\
        \vdots & \ddots & \vdots \\
        \hat{I}^{f_0}_{\mu_1\mu_d} & \cdots & \hat{I}^{f_0}_{\mu_d\mu_d}
    \end{pmatrix}
    \quad \text{and} \quad
    \widehat{C}_{\mu; \lambda}
    = \begin{pmatrix}
        \hat{I}_{\mu_1\lambda_1} & \cdots & 0 \\
        \vdots & \ddots & \vdots \\
        0 & \cdots & \hat{I}_{\mu_d\lambda_d}
    \end{pmatrix}
\end{equation}
for $\hat{I}_{\mu_j\lambda_j}$, $j\in\{1,\ldots,d\},$ and $\hat{I}^{f_0}_{\mu_j\mu_k}$, $j,k\in\{1,\ldots,d\},$ as defined in \eqref{eq: prop estimate quantities eq1} and \eqref{eq: prop estimate quantities eq2}, respectively.
Using the fact that $\E_{g_0}\left[ \Tilde{\Delta}_{f_0;g_0;\lambda}^{(n)} (\mub)\right] = \boldsymbol{0}$, the variance of the efficient central sequence is given by
\begin{equation} \label{eq: variance}
    V_{g_0}^{f_0} \left( \mub \right)  
    = \var_{g_0} \left( \Tilde{\Delta}_{f_0;g_0;\lambda}^{(n)} (\mub) \right) 
    = \E_{g_0}\left[ \Tilde{\Delta}_{f_0;g_0;\lambda} (\mub) \Tilde{\Delta}_{f_0;g_0;\lambda} (\mub)' \right],
\end{equation}
where
\begin{equation*}
    \Tilde{\Delta}_{f_0;g_0;\lambda} (\mub)
     = \Delta_{\lambda}(\mub) - C^{g_0}_{\mu; \lambda} \left( C^{f_0; g_0}_{\mu; \mu} \right) ^{-1} \Delta_{f_0; \mu}(\mub), 
\end{equation*}
for
\begin{equation*}
     \Delta_{f_0; \mu}(\mub)
     = \begin{pmatrix}
        \phi^{f_0}_1(\thetab - \mub) \\
        \vdots \\
        \phi^{f_0}_d(\thetab - \mub)
        \end{pmatrix} \text{ and }
        \Delta_{\lambda}(\mub) = 
        \begin{pmatrix}
        \sin(\theta_{1}-\mu_1) \\
        \vdots \\
        \sin(\theta_{d}-\mu_d) 
    \end{pmatrix},
\end{equation*}
where $\thetab \overset{\mathrm{d}}{=} \thetab_i$.
This can be estimated using 
\begin{equation} \label{eq: variance estimate}
    \hat{V}_{f_0}\left(\hat{\mub}^{(n)} \right) 
    = \frac{1}{n} \sum_{i=1}^n \tilde{\Delta}^{(n)\ast}_{f_0; \lambda; i} \left(\hat{\mub}^{(n)} \right) \tilde{\Delta}^{(n) \ast}_{f_0; \lambda; i} \left( \hat{\mub}^{(n)} \right)' 
\end{equation}
for $\Tilde{\Delta}_{f_0; \lambda; i}^{(n)\ast} (\hat{\mub}^{(n)})$ defined via
\begin{equation*}
    \Tilde{\Delta}_{f_0; \lambda}^{(n)\ast} (\hat{\mub}^{(n)})
    = \frac{1}{\sqrt{n}} \sum_{i=1}^n \Tilde{\Delta}_{f_0; \lambda; i}^{(n)\ast} (\hat{\mub}^{(n)}).
\end{equation*}
In the following proposition, we show that the semi-parametric central sequence and the estimate of its variance converge, in probability, to the parametric ones.
Its proof can be found in Section \ref{appendix: proof unspecified_median_esimates} of the Supplementary Material.
\begin{proposition} \label{prop: unspecified_median_esimates}
    Suppose that $f_0, g_0 \in \mathcal{F}$ and Assumptions \ref{assumption_on_f0},  \ref{assumption: mu_estimate} and \ref{assumption_on_f0_unspecified_median} hold. Then, as $n\rightarrow\infty$ under $P^{(n)}_{\mub, \boldsymbol{0}; g_0}$, it holds that 
    \begin{enumerate}
        \item[(i)] $\Tilde{\Delta}_{f_0; \lambda}^{(n)\ast} (\hat{\mub}^{(n)}) - \Tilde{\Delta}_{f_0;g_0;\lambda}^{(n)} (\mub) = o_P(1) \one_{d \times 1}$ and 

        \item[(ii)] $\hat{V}_{f_0}\left(\hat{\mub}^{(n)} \right)  - V^{f_0}_{g_0}\left(\mub \right) = o_P(1)\one_{d \times d}$
    \end{enumerate} 
    where $\one_{n\times m}$ is the $n\times m$ matrix of ones.
\end{proposition}

It is worth noting that in the proof of Proposition~\ref{prop: unspecified_median_esimates}, we prove that
\begin{equation*}
    \Delta_{f_0; \mu}^{(n)} (\mub^{(n)})-\Delta_{f_0; \mu}^{(n)} (\mub)
    = - C^{f_0; g_0}_{\mu; \mu}\taub_\mu + o_P(1) \one_{d\times1}
\end{equation*}
as $n\rightarrow \infty$, which is a result that provides the asymptotic linearity of the central sequence when the expectations are evaluated under $g_0$.
This is not a straightforward result and is sometimes assumed to hold in situations where Le Cam theory is used to build optimal tests.

Using Propositions \ref{prop: estimate quantities} and \ref{prop: unspecified_median_esimates}, we propose the locally and asymptotically optimal test for symmetry $\phi^{\ast (n)}_{f_0}$  that rejects $\mathcal{H}^{(n)}_{0}$ at asymptotic level $\alpha$ whenever the statistic
\begin{equation} \label{eq: semi-parametric test unspecified}
    Q_{f_0}^{\ast (n)} := \left(\Tilde{\Delta}_{f_0; \lambda}^{(n)\ast} (\hat{\mub}^{(n)}) \right)' \left(\hat{V}_{f_0} (\hat{\mub}^{(n)} ) \right)^{-1} \Tilde{\Delta}_{f_0; \lambda}^{(n)\ast} (\hat{\mub}^{(n)})
\end{equation}
exceeds $\chi^2_{d; \alpha}$.
The following theorem states the asymptotic properties of $Q_{f_0}^{\ast (n)}$.
Its proof can be found in Section \ref{appendix: proof optimality unspecified median} of the Supplementary Material.
\begin{theorem} \label{thm: optimality unspecified median}
For $Q_{f_0}^{\ast(n)}$ as defined in \eqref{eq: semi-parametric test unspecified} and working under Assumptions \ref{assumption_on_f0}, \ref{assumption: mu_estimate} and \ref{assumption_on_f0_unspecified_median}, the following hold
    \begin{enumerate}
        \item[(i)] Under $\mathcal{H}^{(n)}_{0}$, 
        \begin{equation}
            Q_{f_0}^{\ast(n)}   \xrightarrow{\mathcal{D}} 
            \chi^2_d
        \end{equation}
        as $n \rightarrow \infty$ so that the test $\phi^{\ast (n)}_{f_0}$, defined above \eqref{eq: semi-parametric test unspecified}, has asymptotic level $\alpha$ under the same hypothesis.

        \item[(ii)] Under $\cup_{\mub \ \in [-\pi,\pi)^d}P^{(n)}_{\mub, n^{-1/2} \boldsymbol{\tau}^{(n)}_\lambda; g_0}$ with $g_0 \in \mathcal{F}$,
\begin{equation}\label{eq: dist_under_alternative unspecified}
    Q_{f_0}^{\ast (n)} \xrightarrow{\mathcal{D}} 
    \chi^2_d \Bigl( \boldsymbol{\tau}'_\lambda C_{g_0}^{f_0}(\mub) V_{g_0}^{f_0}(\mub)^{-1} C_{g_0}^{f_0}(\mub) \boldsymbol{\tau}_\lambda \Bigr)
\end{equation}
as $n \rightarrow \infty$ with $\boldsymbol{\tau}_\lambda = \lim_{n\rightarrow\infty} \boldsymbol{\tau}^{(n)}_\lambda$, $V_{g_0}^{f_0}(\mub)$ is defined in \eqref{eq: variance}
and
\begin{equation*}
    C_{g_0}^{f_0}(\mub) = \cov_{g_0}\left( \Tilde{\Delta}^{(n)}_{f_0;g_0;\lambda} (\mub), \Delta^{(n)}_{\lambda} (\mub) \right).
\end{equation*}

     \item[(iii)] The test $\phi^{\ast (n)}_{f_0}$ is locally and asymptotically maximin at level $\alpha$ when testing $\mathcal{H}^{(n)}_0$ against $\cup_{\mu \ \in [-\pi,\pi)^d}P^{(n)}_{\mub, n^{-1/2} \boldsymbol{\tau}^{(n)}_\lambda; f_0}$.
    \end{enumerate}
\end{theorem}

By point (ii), we can write down the explicit asymptotic power of the test against the local alternatives $\cup_{\mu \ \in [-\pi,\pi)^d} P^{(n)}_{\mub, n^{-1/2} \boldsymbol{\tau}^{(n)}_\lambda; f_0}$ as $n \rightarrow \infty$ under the form $\Prob\Bigl( Q \geq \chi^2_{d;\alpha} \Bigr)$ for $Q \sim \chi^2_d \Bigl( \boldsymbol{\tau}'_\lambda C_{g_0}^{f_0}(\mub) V_{g_0}^{f_0}(\mub)^{-1} C_{g_0}^{f_0}(\mub) \boldsymbol{\tau}_\lambda \Bigr)$.
As a special case, when $d=1$, we retrieve the results of \cite{ameijeiras-alonso_optimal_2021}.

From our construction, it is clear that we obtain a different test for each distribution $f_0$ it is based on, and that each test is valid under the entire null hypothesis $\mathcal{H}^{(n)}_{0}$. A finite-sample comparison of the performances of these distinct tests is given in the following section. Practitioners have the advantage to choose which form of the test they prefer, based on their own preferences such as simplicity of the density $f_0$, optimality under a certain distribution, etc. We however draw the reader's attention to the fact that singularity issues arise when using the Cosine distribution as $f_0$.
In the one-dimensional case, the von Mises distribution has singular Fisher Information matrix, as noted in \cite{ameijeiras-alonso_optimal_2021}. 
So, our observation is expected, since the Cosine distribution is a submodel of the bivariate von Mises distribution.
However, as our model is robust to the assumption of the underlying distribution, $f_0$, the tests can be built with other distributions, such as the bivariate wrapped Cauchy in the two-dimensional case.
Thus, we do not investigate this issue further.
Interestingly, the test built using the Sine model is invariant to the choice of the value of the parameters.

As with the case of a known symmetry center, we are interested in deriving a bound for the distance (integral probability metric) between $Q_{f_0}^{\ast(n)}$ and its limiting distribution; such a bound quantifies the quality of the approximation in case of finite observations.
The proof of Theorem~\ref{theorem: stein's method unspecified} is based on Stein's Method, and it can be found in Section \ref{sec: proof steins method unspecified} of the Supplementary Material.

\begin{theorem}\label{theorem: stein's method unspecified}
For $i,j\in\{1,\ldots,d\}$, define  $\gamma_{ij} =  \left( V_{g_0}^{f_0} \left( \mub \right)^{-1} \right)_{ij}$ and $\hat{\gamma}_{ij} =  \left(\hat{V}_{f_0}\left(\mub \right)^{-1} \right)_{ij}$.
    Then, for $h \in C^6_b(\R)$ and working under Assumptions \ref{assumption_on_f0}, \ref{assumption: mu_estimate}, \ref{assumption_on_f0_unspecified_median}, and \ref{assumption: MSE},
    \begin{align} \label{eq: steins method unspecified}
        \Bigl|\E_{g_0} \left[ h\left(Q_{f_0}^{\ast (n)} \right) \right] & - \E_{g_0} \left[ h\left(\chi^2_d\right) \right] \Bigr| \nonumber\\
        \leq
        \frac{C_3}{n} 
        + \lVert h^{(1)} \rVert & \left( \E_{g_0} \left[ \Big\lvert\Big\lvert \left( { \left. \nabla_{\mub} Q_{f_0}^{\ast (n)} (\mub) \right|_{\mub = \mub^\ast}}  \right)' \Big\rvert\Big\rvert_2^2 \right] \E_{g_0} \left[ \Big\lVert \hat{\mub}^{(n)} - \mub \Big\rVert_2^2 \right] \right)^{1/2} \nonumber \\
        + \frac{1}{n} \lVert h^{(1)} \rVert \sum_{i=1}^d & \sum_{j=1}^d \left\{ 
        \left( \E_{g_0} \left[ \left( \hat{\gamma}_{ij} \right)^2 \right] \right)^{1/2}
        \left( \E_{g_0} \left[ \left( \sum_{k=1}^n \sum_{\ell=1}^n \left( \Tilde{\Delta}_{ik}^{\ast} \Tilde{\Delta}_{j\ell}^{\ast} - \Tilde{\Delta}_{ik} \Tilde{\Delta}_{j\ell} \right) \right)^2 \right] \right)^{1/2} \right. \nonumber \\
        & \left. + \ \left( \E_{g_0} \left[ \left( \hat{\gamma}_{ij} - \gamma_{ij} \right)^2 \right] \right)^{1/2} 
        \left( \E_{g_0} \left[ \left( \sum_{k=1}^n \sum_{\ell=1}^n \Tilde{\Delta}_{ik} \Tilde{\Delta}_{j\ell} \right)^2 \right] \right)^{1/2} \right\}, 
    \end{align}
{for some $\mub^\ast$ between $\hat{\mub}$ and $\mub$ and where
\begin{equation}\label{eq: test_statistic_unspecified_evaluated}
    Q_{f_0}^{\ast (n)} (\mub) := \left(\Tilde{\Delta}_{f_0; \lambda}^{(n)\ast} (\mub)  \right)' \left(\hat{V}_{f_0} (\mub) \right)^{-1} \Tilde{\Delta}_{f_0; \lambda}^{(n)\ast} (\mub)
\end{equation}
is the test statistic evaluated using the true value of $\mub$.} {$C_3$ is $O(1)$ and} ${n \,} C_3$ is the same {bound} as in Theorem~\ref{thm: steins method specified},  for $X_{lj} = \begin{pmatrix}
        a_{j1} & \cdots & a_{jd}
    \end{pmatrix}\Tilde{\Delta}_{f_0;g_0;\lambda;l}^{(n)} (\mub)$, 
    $l\in\{1,\ldots,n\}$, 
    $a_{ij} = \left( \left( V_{g_0}^{f_0} \left( \mub \right)^{-1} \right)^{1/2}\right)_{ij}$, $\Tilde{\Delta}_{f_0;g_0;\lambda;l}^{(n)} (\mub)$ is $\Tilde{\Delta}_{f_0;g_0;\lambda}^{(n)} (\mub)$ evaluated at the $l^\text{th}$ observation, 
    and $V_{g_0}^{f_0} \left( \mub \right)$ is as in \eqref{eq: variance}. 
    $\Tilde{\Delta}_{ik}$ and $\Tilde{\Delta}_{ik}^{\ast}$ denote the $i^\text{th}$ component of the vectors $\Tilde{\Delta}_{f_0;g_0;\lambda}^{(n)} (\mub)$ and $\Tilde{\Delta}_{f_0; \lambda}^{(n)\ast} (\mub)$, respectively, evaluated at the $k^\text{th}$ observation, as defined in \eqref{eq: semi_parametric_central_seq_ij}.
    \end{theorem}
In order to be able to calculate the order of the bound in \eqref{eq: steins method unspecified}, we need a further assumption on the estimator of the location parameter, given in Assumption~\ref{assumption: MSE} below.
\begin{assumption}\label{assumption: MSE}
    It holds that the mean squared error (MSE) of $\hat\mu^{(n)}_j$ as an estimator of $\mu_j$ converges to 0 with rate $n^{-1}$, meaning $\E_{g_0} \left[ (\hat{\mu}^{(n)}_j - \mu_j)^2 \right] = O \left( n^{-1} \right)$. 
\end{assumption}
The assumption basically requires $L^2$-control of  $\sqrt{n} \left( \hat{\mub}^{(n)} - \mub \right)$. A sufficient condition to obtain such a control is that this quantity admits an asymptotically linear expansion, which means that $\forall j \in \{1,\ldots,d\}$, we have the representation
\begin{equation}
\label{eq:asymptotic_linearity}
\sqrt n\left(\hat\mu_j^{(n)}-\mu_j\right)=\frac1{\sqrt n}\sum_{i=1}^n\psi_j(\theta_i)+R_{n,j},
\end{equation}
where $\E_{g_0} [ \psi_j(\theta_1) ]=0$, $\E_{g_0} [ \psi_j^2(\theta_1) ] < \infty$,  and $\E_{g_0}(R_{n,j}^2)=o(1)$. 
To see that the linear representation in \eqref{eq:asymptotic_linearity} is sufficient for Assumption \ref{assumption: MSE} to be satisfied, notice that
\begin{equation}
\E_{g_0}\left[n\left(\hat\mu_j^{(n)}-\mu_j\right)^2\right] 
\leq 2\E_{g_0}\left[\left(\frac1{\sqrt n} \sum_{i=1}^n \psi_j(\theta_i) \right)^2 \right] + 2\E_{g_0} \left[ R_{n,j}^2 \right] 
=  2\E_{g_0} \left[\psi_j^2(\theta_1)\right] + o(1) 
= O(1),
\end{equation}
since $\E_{g_0} [ \psi_j^2(\theta_1) ]$ is bounded. Therefore, $\E_{g_0} \left[ n\left(\hat\mu_j^{(n)}-\mu_j\right)^2\right] = O(1)$, which is equivalent to the statement of the assumption.

We now highlight that regular estimators, such as $M$-estimators, can be (under standard regularity conditions strengthened to provide $L^2$-control) represented as in \eqref{eq:asymptotic_linearity} upon appropriate choice of the function $\psi_j(\cdot)$. For ease of presentation and understanding, we focus on the maximum likelihood estimator (MLE). 
One can show that under standard regularity conditions (similar to those required for its asymptotic normality)  \eqref{eq:asymptotic_linearity} holds, where, in this case, the function $\psi(\cdot) = (\psi_1(\cdot),\ldots,\psi_d(\cdot))^\prime$ is the score function adjusted appropriately through the inverse of the expected Fisher information. 
The relevant proposition and its proof can be found in Section \ref{appendix: proof msal mle}, where we essentially show that a regular MLE satisfies Assumption \ref{assumption: MSE} provided the usual Taylor expansion of the score has an $L^2$ negligible remainder. 
This requires that the observed information matrix converges to the expected Fisher information with sufficient \(L^p\), $p>1$, control and that the normalized score has a bounded moment of order larger than two.

Lemma~\ref{lemma: rate of diff of gammas unspecified} below provides upper bounds that are, as well, useful for the discussion on the order of the bound in \eqref{eq: steins method unspecified}. The proof is in Section \ref{appendix: proof lemma rate of diff of gammas unspecified} of the online supplement.

\begin{lemma}\label{lemma: rate of diff of gammas unspecified}
For $i,j \in \{1,\ldots,d\}$, $k\in\{1,\ldots,n\}$ and $\Tilde{\Delta}_{ik}$, $\Tilde{\Delta}_{ik}^{\ast}$, $\gamma_{ij}$ and $\hat{\gamma}_{ij}$ as defined in Theorem~\ref{theorem: stein's method unspecified}, it holds that 
    \begin{enumerate}
    \item[(i)] $\E_{g_0} \left[ \Big\lVert \left( { \left. \nabla_{\mub} Q_{f_0}^{\ast (n)} (\mub) \right|_{\mub = \mub^\ast}} \right)' \Big\rVert_2^2 \right] \leq C_4$;
        \item[(ii)] $\E_{g_0} \left[ \left( \sum_{k=1}^n \sum_{\ell=1}^n \Tilde{\Delta}_{ik} \Tilde{\Delta}_{j\ell} \right)^2 \right] \leq n^2C_5$;
        \item[(iii)] $\E_{g_0} \left[ \left( \sum_{k=1}^n \sum_{\ell=1}^n \left( \Tilde{\Delta}_{ik}^{\ast} \Tilde{\Delta}_{j\ell}^{\ast} - \Tilde{\Delta}_{ik} \Tilde{\Delta}_{j\ell} \right) \right)^2 \right] \leq nC_6$;
        \item[(iv)] $\E_{g_0} \left[ \left( \hat{\gamma}_{ij} - \gamma_{ij} \right)^2 \right] \leq C_7/n$;
        \item[(v)] $\E_{g_0} \left[ \left( \hat{\gamma}_{ij} \right)^2 \right] \leq C_8$,
    \end{enumerate} 
    for a large enough constant $C_4$ as given in \eqref{eq: bound for grad Q} and universal constants $C_5, C_6, C_7, C_8$, which depend on $f_0$ and $g_0$.   
\end{lemma}
\noindent Similarly to Lemma~\ref{lemma: rate of diff of gammas}, the constants $C_5, C_6, C_7, C_8$ cannot be calculated explicitly due to multiplying with inverses of matrices in the expressions of the efficient central sequence and the test statistic.

Under Assumption~\ref{assumption: MSE}, $\E_{g_0} \left[ \Big\lVert \hat{\mub}^{(n)} - \mub \Big\rVert_2^2 \right] = O \left( \frac{1}{n} \right)$ and based on the results from Lemma~\ref{lemma: rate of diff of gammas unspecified}, straightforward calculations yield that the order of the bound in \eqref{eq: steins method unspecified} is $O\left( \frac{1}{\sqrt{n}} \right)$.
{Similarly to Theorem~\ref{thm: steins method specified}, the result quantifies the distance between $Q_{f_0}^{\ast (n)}$ and its limiting chi-square distribution with $d$ degrees of freedom and serves as an alternative proof of Theorem~\ref{thm: optimality unspecified median}(i),
as it indicates convergence in distribution.}

\section{Simulation results} \label{sec: simulations}
In this section we present the simulation results for our new tests, $\phi^{\ast (n); \mu}$ for the specified symmetry center case, and $\phi^{\ast (n)}_{f_0}$ for the unspecified symmetry center case.
We will consider various distinct settings, and for each setting we conduct a simulation study with 1000 replications.
The Monte Carlo estimates of the probability of rejection of the respective null hypotheses ($\mathcal{H}^{(n)}_{0; \mu}$ and $\mathcal{H}^{(n)}_{0}$) at the $5\%$ level of significance are presented.
For $d=2$, the distributions used to generate data are the Sine, Cosine and bivariate wrapped Cauchy distributions, denoted by $S_{\kappa_1;\kappa_2;\rho}$, $C_{\kappa_1;\kappa_2;\rho}$, and $BWC_{\xi_1;\xi_2;\Tilde{\rho}}$, respectively, where $\kappa_1, \kappa_2 \geq 0$, $\xi_1,\xi_2\in[0,1)$ are concentration parameters and $\rho\in\R$, $\Tilde{\rho}\in(-1,1)$ are dependence parameters.
For the Sine and Cosine models the parameters are chosen to satisfy the conditions of Theorems {2 and 3} of \cite{mardia_protein_2007} such that the distributions are unimodal.
For $d=3$, the trivariate wrapped Cauchy (TWC) copula \citep{TWCM} is used with marginal distributions chosen to be the wrapped Cauchy distribution. 
This distribution is denoted by $TWC_{\rho_{12};\rho_{13};\rho_{23}}^ {\beta_1; \beta_2; \beta_3}$, where $\rho_{12}, \rho_{13}, \rho_{23}\in\R\setminus\{0\}$ are the copula parameters satisfying Equation (8) in \cite{TWCC} and $\beta_1, \beta_2, \beta_3>0$ are the parameters for the marginal distributions.
For any dimension $d$, data can be generated by assuming independence between the marginals or using the multivariate non-negative trigonometric sums model (MNNTS) introduced in \cite{fernandez-duran_modeling_2014}.
For the independent model, the wrapped Cauchy distribution is used and the model is denoted by $I_{\beta_1; \ldots; \beta_d}$ where $\beta_1,\ldots,\beta_d > 0 $ are the marginal concentration parameters.
Data of dimension up to $d=20$ are generated using this model.
The MNNTS model is a general toroidal model
for which, in order to ensure that the data are generated from a symmetric, unimodal distribution for any $d$, we set $M_1=\cdots=M_d=1$ (as defined in Equation (7) of \cite{fernandez-duran_modeling_2014}) and the parameters are chosen to be equal real numbers, and more specifically, equal to $1/(4\pi)^{d/2}$.
To simplify the notation, we simply denote this by MNNTS$_d$.
Due to the high computational complexity of generating data from MNNTS in higher dimensions, the model is used for $d \leq 6$.

Considering firstly the symmetry specified case, the simulation results for our test $\phi^{\ast (n); \mu}$ are indicated in Tables~\ref{tab: simulations_2_known}, \ref{tab: simulations_3_known} and \ref{tab: simulations_456_known}  for dimensions $d=2, d=3$ and $d=5,10,20$, respectively.
Further simulation results are provided in Section \ref{sec: supp_simulations} of the Supplementary Material.
More specifically, Tables \ref{tab: simulations_2_known_supp}, \ref{tab: simulations_3_known_supp} and \ref{tab: simulations_456_known_supp} contain results for $d=2, d=3$ and $d=4,6$, respectively.
As expected, the Type I error of the test is around $5\%$ for $\lambdab = \boldsymbol{0}$ and for $\lambdab \neq \boldsymbol{0}$, the power increases as either the sample size or $||\lambdab||$ increase.
Similar conclusions hold for higher dimensions, which we however do not show here for the sake of presentation. We remark that from dimension 10 on, there are some empty entries in Table~\ref{tab: simulations_456_known} due to the fact that the sum of the components of the skewness parameter exceeds 1, which is not permitted for the distribution to exist.

Turning our attention now to the tests $\phi^{\ast (n)}_{f_0}$ when the symmetry center is not specified, we need to provide some notational explanations. The data are generated from the distribution $g_0$, which can be found on the top of each table, and the test is built using the distribution $f_0$ which is presented on the first column of each table.
Due to the singularity issues arising for the Cosine model, it is not used in the simulations for the unspecified symmetry center case. 
{
Whenever possible, we provide results for the test built using as $f_0$ the family of $g_0$, with both true parameter values and other values.
In the case of data generated from the Sine distribution, we only report the results obtained using one set of parameters, as the test is invariant to that choice.}
Table~\ref{tab: simulations_2_unknown_sine} contains results for $d=2$, while Tables~\ref{tab: simulations_3_unknown}, \ref{tab: simulations_3_unknown_twcc} concern dimension $d=3$.
Further simulation results, including higher dimensions, can be found in Section \ref{sec: supp_simulations} of the Supplementary Material. 
More specifically, simulations for $d=2$ can be found in Tables \ref{tab: simulations_2_unknown_ind}, \ref{tab: simulations_2_unknown_bwc} and \ref{tab: simulations_2_unknown_mnnts} for different distributions $g_0$. 
Tables \ref{tab: simulations_3_unknown_iid} and \ref{tab: simulations_3_unknown_supp} concern $d=3$ while Tables \ref{tab: simulations_456_unknown_mnnts} and \ref{tab: simulations_456_unknown_ind} contain results for $d=4,5,6$ and $d=4,5,6,10,20$, respectively.

Similarly to the symmetry-specified situation, the tests satisfy the significance level and their power increases with the degree of skewness and sample size. 
However, here a larger sample size is required to see clearly that the power converges to 1.
This is expected as the median direction is not known and needs to be calculated.
It can be observed that building tests from some distributions results in slower convergence of the power, see for example Tables~\ref{tab: simulations_2_unknown_sine} and Tables \ref{tab: simulations_2_unknown_bwc}, \ref{tab: simulations_2_unknown_mnnts} in the Supplementary Material.
The simulation results show that the test has higher power when it is built using the distribution $g_0$ from which data are generated, which confirms Theorem~\ref{thm: optimality unspecified median}, see for example Table~\ref{tab: simulations_3_unknown_twcc} and Table \ref{tab: simulations_2_unknown_bwc} in the Supplementary Material.
{It can also be noted that the power of the test increases in most cases as the dimension increases, see Tables \ref{tab: simulations_456_known}, \ref{tab: simulations_456_known_supp}, 
\ref{tab: simulations_456_unknown_mnnts}, \ref{tab: simulations_456_unknown_ind}.
This can be explained by the fact that we set $\lambda_1 = \ldots = \lambda_d$, meaning that increasing the dimension also increases the overall skewness of the data.}

After these extensive and general observations, we now investigate certain settings in more detail, using plots that are provided in Section~\ref{sec: supporting_plots} in the Supplementary Material. 
Figure~\ref{fig: sim_twcc} shows plots of the Monte Carlo estimates of the probability of rejection of three different null hypotheses under specified symmetry center as the sample size increases. 
The blue line represents the power of the test for testing $\mathcal{H}^{(n)}_{0; \mu}: \lambdab = \boldsymbol{0}$, while
the black, red and green lines represent the power of the test for testing for symmetry on each of the directions separately, meaning $\mathcal{H}^{(n)}_{0; \mu;i}: \lambda_i=0$ for $i=1,2,3$. 
The data are generated from the TWC, with parameters $\rho_{12}=\rho_{23}=1, \rho_{13}=0.25$ and wrapped Cauchy marginal distributions with parameters $\beta_1=0.1, \beta_2=0.2, \beta_3=0.3$.
The three components are not independent, which explains why in the left plot there is some power of the test corresponding to $\lambda_3$, even though $\lambda_3 = 0$.
In the first two plots, it is clear that the power of the test when testing all three components at the same time is higher than the power of the other tests.
In the last plot, $\lambda_1$ is the only non-zero skewness parameter and thus the power of the test only for $\lambda_1$ is the most powerful.
However, the power of the test for all three components is only slightly less powerful.
These plots provide evidence in favor of using a test for symmetry on all the dimensions of the data at hand, instead of testing for symmetry for each component separately.

Figure~\ref{fig: theor_power_3d} is a plot of the theoretical power of the test, as given in Theorems~\ref{thm: optimality specified median} and \ref{thm: optimality unspecified median}, in light blue color, and the simulated power, as obtained by Monte Carlo simulations, in red color.
For these plots we consider only $d=2$ and denote $\taub_\lambda = (\tau_3, \tau_4)'$.
For the simulations, 1.000 repetitions were used for each value of $\lambdab = \taub/\sqrt{n}$, and the sample size for each replication was $n=10.000$.
In Figure~\ref{fig: theor_power_3d}(a), the results for the test for the specified symmetry center are plotted, with data being generated from the BWC distribution with parameters $\xi_1 = 0.1, \xi_2=0.5$ and $\Tilde{\rho} = 0.4$.
We observe that the theoretical and estimated powers are almost identical.
In Figures~\ref{fig: theor_power_3d}(b)-(d), the test for unspecified symmetry center is considered.
In (b), the true and assumed densities $f_0$ and $g_0$ are chosen to be the same while, for (c) and (d), they are different.
For (b), the distribution is chosen as the BWC with parameters $\xi_1 = 0.9, \xi_2=0.9$ and $\Tilde{\rho} = 0.4$.
For small values of the skewness parameter the simulated power is lower than the theoretical one, as expected.
For (c), the data are generated from the BWC with $\xi_1 = 0.3, \xi_2=0.3$ and $\rho = 0.4$ while the test is built using the same distribution with $\xi_1 = 0.8, \xi_2=0.8$ and $\Tilde{\rho} = 0.6$.
Finally, for (d), the data are generated from the Sine model with $\kappa_1 = 0.1, \kappa_2=0.1$ and $\rho = 0.1$ and the test is built using the Sine model with $\kappa_1 = 0.2, \kappa_2=0.2$ and $\rho = 0.4$.
The simulated powers of the test in (c) and (d) are also lower than the theoretical ones.
This can be improved by considering a larger $n$, which we did not attempt due to the computational complexity of the simulations. 
It is important to note that the high computational complexity refers to the resources needed to generate data from the underlying distributions, and not to the calculation of the test statistic, which can always be evaluated within seconds.
We note that the difference between theoretical and practical power is the smallest when $f_0=g_0$.

\begin{table}[tbp]\centering
\caption{Percentage of rejections for $\phi^{\ast (n); \mu}$ when $d=2$. The data are generated using the model mentioned, $n$ represents the sample size and $\lambdab$ the value of the skewness parameter.\label{tab: simulations_2_known}}
\begin{tabular}{@{}rrrrrrrr@{}} 
\toprule
& \multicolumn{1}{c}{$\lambdab$} & {$(0,0)$} & {$(0.1,0)$} &
{$(0.1,0.1)$} & {$(0.2,0.1)$} & {$(0.2,0.2)$}\\ 
{Model} & \multicolumn{1}{c}{$n$} & \\ \midrule
& 200 & 0.058 & 0.118 & 0.225 & 0.488 & 0.717 \\
$I_{0.1;0.1}$ & 500 & 0.056 & 0.253 & 0.501 & 0.893 & 0.985 \\
& 1000 & 0.042 & 0.492 & 0.804 & 1.000 & 1.000 \\
\\
& 200 & 0.051 & 0.114 & 0.202 & 0.448 & 0.720 \\
$S_{1; 1; 0.1}$ & 500 & 0.043 & 0.280 & 0.472 & 0.886 & 0.977 \\
& 1000 & 0.045 & 0.484 & 0.808 & 0.993 & 1.000 \\
\\
& 200 & 0.047 & 0.118 & 0.194 & 0.456 & 0.710 \\
$C_{1; 1; 0.1}$ & 500 & 0.046 & 0.269 & 0.459 & 0.870 & 0.975 \\
& 1000 & 0.047 & 0.456 & 0.808 & 0.996 & 1.000 \\
\\
& 200 & 0.048 & 0.121 & 0.242 & 0.569 & 0.780 \\
$BWC_{0.1; 0.5; 0.3}$ & 500 & 0.045 & 0.292 & 0.551 & 0.923 & 0.992 \\
& 1000 & 0.047 & 0.494 & 0.857 & 1.000 & 1.000 \\
\\
& 200 & 0.044 & 0.144 & 0.224 & 0.512 & 0.730 \\
$MNNTS_2$ & 500 & 0.051 & 0.268 & 0.521 & 0.895 & 0.984
\\
& 1000 & 0.048 & 0.502 & 0.829 & 0.996 & 1.000 \\
\bottomrule
\end{tabular}
\end{table}

\begin{table}[tbp]\centering
\caption{Percentage of rejections for $\phi^{\ast (n); \mu}$ when $d=3$. The data are generated using the model mentioned, $n$ represents the sample size and $\lambdab$ the value of the skewness parameter.\label{tab: simulations_3_known}}
\begin{tabular}{@{}rrrrrrrr@{}} 
\toprule
& \multicolumn{1}{c}{$\lambdab$} & {$(0,0,0)$} & {$(0.1,0,0)$} &
{$(0.2,0.1,0)$} & {$(0.2,0.2,0.2)$} \\ 
{Model} & \multicolumn{1}{c}{$n$} & \\ \midrule
& 200 & 0.041 & 0.109 & 0.414 & 0.821\\
$I_{0.1; 0.1; 0.1}$ & 500 & 0.054 & 0.238 & 0.854 & 0.999\\
& 1000 & 0.063 & 0.451 & 0.994 & 1.000 \\
\\
& 200 & 0.066 & 0.116 & 0.440 & 0.834 \\
$TWC_{5; 2; 0.1}^{0.1; 0.1; 0.1}$ & 500 & 0.048 & 0.229 & 0.852 & 0.998 \\
& 1000 & 0.047 & 0.448 & 0.992 & 1.000 \\
\\
& 200 & 0.070 & 0.116 & 0.440 & 0.861 \\
$MNNTS_3$ & 500 & 0.046 & 0.220 & 0.864 & 1.000 \\
& 1000 & 0.061 & 0.446 & 0.995 & 1.000 \\
\bottomrule
\end{tabular}
\end{table}

\begin{table}[tbp]\centering
\caption{Percentage of rejections for $\phi^{\ast (n); \mu}$ when $d=5,10,20$. The data are generated using the model mentioned, $n$ represents the sample size and $\lambdab$ the value of the skewness parameter. Note that the missing entries correspond to non-allowed settings where the sum of the components of $\lambdab$ would exceed 1.\label{tab: simulations_456_known}}
\begin{tabular}{@{}rrrrrrrr@{}} 
\toprule
& & \multicolumn{1}{c}{$\lambdab$} & {$(0,\ldots,0)$} & {$(0.05,\ldots,0.05)$} &
{$(0.1,\ldots,0.1)$} & {$(0.15,\ldots,0.15)$} \\ 
{Model} & $d$ & \multicolumn{1}{c}{$n$} & \\ \midrule
& & 200 & 0.041 & 0.134 & 0.350 & 0.705 \\
$I_{0.1; \ldots; 0.1}$ & $5$ & 500 & 0.049 & 0.206 & 0.773 & 0.998 \\
& & 1000 & 0.049 & 0.442 & 0.979 & 1.000 \\
\\
& & 200 & 0.047 & 0.080 & 0.212 & 0.516 \\
$I_{0.6; \ldots; 0.6}$ & $5$ & 500 & 0.051 & 0.131 & 0.534 & 0.925 \\
& & 1000 & 0.040 & 0.278 & 0.884 & 0.997 \\
\\
& & 200 & 0.049 & 0.118 & 0.346 & 0.709 \\
$MNNTS_5$ & 5 & 500 & 0.050 & 0.222 & 0.780 & 0.992 \\
& & 1000 & 0.066 & 0.447 & 0.982 & 1.000  \\
\\
& & 200 & 0.051 & 0.147 & 0.484 & - \\
$I_{0.1; \ldots; 0.1}$ & 10 & 500 & 0.047 & 0.321 & 0.937 & - \\
& & 1000 & 0.052 & 0.640 & 1.000 & - \\
\\
& & 200 & 0.040 & 0.089 & 0.316 & - \\
$I_{0.6; \ldots; 0.6}$ & 10 & 500 & 0.044 & 0.197 & 0.773 & - \\
& & 1000 & 0.050 & 0.440 & 0.991 & - \\
\\
& & 200 & 0.049 & 0.162 & - & - \\
$I_{0.1; \ldots; 0.1}$ & 20 & 500 & 0.047 & 0.515 & - & - \\
& & 1000 & 0.055 & 0.866 & - & - \\
& & 1500 & 0.041 & 0.983 & - & - \\
\\
& & 200 & 0.043 & 0.088 & - & - \\
$I_{0.6; \ldots; 0.6}$ & 20 & 500 & 0.046 & 0.284 & - & - \\
& & 1000 & 0.055 & 0.623 & - & - \\
& & 1500 & 0.050 & 0.859 & - & - \\
\bottomrule
\end{tabular}
\end{table}

\begin{table}[tbp]\centering
\caption{Percentage of rejections for $\phi^{\ast (n)}_{f_0}$ when $d=2$. The data are generated using the model $g_0$ and the test statistic is evaluated using the model $f_0$, $n$ represents the sample size and $\lambdab$ the value of the skewness parameter.\label{tab: simulations_2_unknown_sine}}
\begin{tabular}{@{}rrrrrrrr@{}} 
\toprule
\multicolumn{2}{c}{$g_0 = S_{1; 1; 0.7}$} \\
\midrule
& \multicolumn{1}{c}{$\lambdab$} & {$(0,0)$} & {$(0.1,0)$} &
{$(0.1,0.1)$} & {$(0.2,0.1)$} & {$(0.2,0.2)$}\\ 
$f_0$ & \multicolumn{1}{c}{$n$} & \\
\midrule
& 500 & 0.058  & 0.069 & 0.147 & 0.290 & 0.387 \\
$S_{0.5; 0.5; 0.1}$ & 1000 & 0.050  & 0.108 & 0.255 & 0.494 & 0.673 \\
& 5000 & 0.059  & 0.402 & 0.865 & 0.996 & 1.000 \\
\midrule
& \multicolumn{1}{c}{$\lambdab$} & {$(0,0)$} & {$(0.2,0)$} &
{$(0.4,0)$} & {$(0.6,0)$} & {$(0.8,0)$}\\ 
\midrule
& 500 & 0.067 & 0.068 & 0.055 & 0.099 & 0.409 \\
$I_{0.1;0.1}$ & 1000 & 0.045 &  0.061  & 0.069  & 0.139 &  0.690 \\
& 5000 & 0.062  & 0.115  & 0.097 &  0.522  & 1.000 \\
\\
& 500 & 0.049 &  0.043  & 0.061  & 0.123  & 0.254 \\
$BWC_{0.5; 0.5; 0.3}$ & 1000 & 0.048 &  0.069 &  0.118  & 0.217  & 0.454 \\
& 5000 & 0.054 &  0.127 &  0.371 &  0.823  & 0.998 \\
\bottomrule
\end{tabular}
\end{table}

\begin{table}[tbp]\centering
\caption{Percentage of rejections for $\phi^{\ast (n)}_{f_0}$ when $d=3$. The data are generated using the model $g_0$ and the test statistic is evaluated using the model $f_0$, $n$ represents the sample size and $\lambdab$ the value of the skewness parameter.\label{tab: simulations_3_unknown}}
\begin{tabular}{@{}rrrrrrrr@{}} 
\toprule
\multicolumn{2}{c}{$g_0 = MNNTS_3$}\\
\midrule
& \multicolumn{1}{c}{$\lambdab$} & {$(0,0,0)$} & {$(0.1,0,0)$} &
{$(0.2,0.1,0)$} & {$(0.2,0.2,0.2)$} \\ 
$f_0$ & \multicolumn{1}{c}{$n$} & \\ 
\midrule
& 200 & 0.078 & 0.095 & 0.142 & 0.210 \\
$I_{0.1; 0.1; 0.1}$ & 500 & 0.052 & 0.089 & 0.202 & 0.397 \\
& 1000 & 0.058 & 0.109 & 0.324 & 0.646 \\
\\
& 200 & 0.056 & 0.077 & 0.213 & 0.644 \\
$TWC_{1; 1.2; 0.5}^{0.1; 0.2; 0.3}$ & 500 & 0.054 & 0.097 & 0.265 & 0.967 \\
& 1000 & 0.059 & 0.088 & 0.446 & 1.000 \\
\bottomrule
\end{tabular}
\end{table}

\begin{table}[tbp]\centering
\caption{Percentage of rejections for $\phi^{\ast (n)}_{f_0}$ when $d=3$. The data are generated using the model $g_0$ and the test statistic is evaluated using the model $f_0$,  $n$ represents the sample size and $\lambdab$ the value of the skewness parameter.\label{tab: simulations_3_unknown_twcc}}
\begin{tabular}{@{}rrrrrrrr@{}} 
\toprule
\multicolumn{2}{c}{$g_0 = TWC_{5; 2; 0.1}^{0.1; 0.1; 0.1}$}\\
\midrule
& \multicolumn{1}{c}{$\lambdab$} & {$(0,0,0)$} & {$(0.1,0,0)$} &
{$(0.1,0.1,0)$} & {$(0.1,0.1,0.1)$} \\ 
$f_0$ & \multicolumn{1}{c}{$n$} & \\ 
\midrule
& 500 & 0.042 & 0.232 & 0.440 & 0.596 \\
$I_{0.1; 0.1; 0.1}$ & 1000 & 0.052 & 0.436 & 0.742 & 0.906 \\
& 5000 & 0.056 & 0.982 & 1.000 & 1.000 \\
\\
& 500 & 0.062 & 0.226 & 0.405 & 0.530 \\
$I_{0.1; 0.2; 0.3}$ & 1000 & 0.064 & 0.424 & 0.720 & 0.860 \\
& 5000 & 0.061 & 0.990 & 1.000 & 1.000 \\
\\
& 500 & 0.035 & 0.201 & 0.384 & 0.585 \\
$TWC_{5; 2; 0.1}^{0.1; 0.1; 0.1}$ & 1000 & 0.052 & 0.389 & 0.746 & 0.896 \\
& 5000 & 0.047 & 0.984 & 1.000 & 1.000 \\
\\
& 500 & 0.044 & 0.171 & 0.248 & 0.499 \\
$TWC_{1; 1.2; 0.5}^{0.1; 0.2; 0.3}$ & 1000 & 0.051 & 0.228 & 0.429 & 0.785 \\
& 5000 & 0.040 & 0.461 & 0.899 & 1.000 \\
\bottomrule
\end{tabular}
\end{table}

\section{Application: protein data} \label{sec: real data}

In this section, our new tests are applied to a real-world dataset.
The chosen dataset includes data from the protein folding problem, which involves predicting the three-dimensional structure of a protein using the sequence of amino acids that make up the protein.
This is an important and challenging problem with implications to vaccine and medicine development.
Deep learning models like AlphaFold2, whose lead scientists won the Nobel Prize in Chemistry in 2024, have drastically improved accuracy, solving protein structures that were unknown for decades. 
However, the problem is not yet completely solved.
The open questions include dynamics, mutants, accuracy \citep{lopez-sagaseta_severe_2025, chakravarty_alphafold_2024}.
A natural extension of protein folding is RNA folding, for which even the most advanced protein prediction algorithms fall short \citep{kwon_rna_2025}.

Many papers in the literature model datasets relating to the protein folding problem using distributions on the torus, see for example \cite{TWCM}, \cite{mardia_multivariate_2008} and \cite{mardia_mixtures_2012}.
\cite{Mardia13} reviews statistical advances in some major active areas of protein structural bioinformatics, including structure prediction.
Most efforts have employed symmetric distributions.
Based on visual inspection of such datasets, biologists believe that asymmetric distributions should be used \citep{Mardia13}.
Our aim here is to provide a statistical confirmation of this fact.

Proteins are made up of amino acids, which in turn are made up of the backbone and the sidechains.
In this analysis, we only focus on the backbone but one could also consider the angles formed in the sidechains.
The backbone consists of two flexible chemical bonds, NH-C$\alpha$ and C$\alpha$-CO, where C$\alpha$ denotes the central Carbon atom.
The angle that can be found at the beginning of each residue is the NH-C$\alpha$ torsion angle, denoted by $\phi$ and at the end is the C$\alpha$-CO torsion angle, denoted by $\psi$.
The peptide bond, CO-NH, is a torsion angle that cannot rotate freely, due to the physiochemical properties and is denoted by $\omega$.
These angles need to be studied as specific combinations of them allow the favourable hydrogen bonding patterns, while others can `result in clashes within the backbone or between adjacent sidechains' \citep{jacobsen_introduction_2023}.

For the present data analysis, we consider position 55 at 2000 randomly selected times in the molecular dynamic trajectory of the SARS-CoV-2 spike domain from \cite{genna_sars-cov-2_2020}.
The position occurs in $\alpha$-helix throughout the trajectory. DPPS \cite{kabsch_dictionary_1983} is used to compute the secondary structure and \cite{cock_biopython_2009} to verify the chains. 
Since we only use data from an $\alpha$-helix, the data used have a unique mode, which can also be observed by the plot of the angles in Figure~\ref{fig: rose_plots}, which is provided in Section~\ref{sec: supporting_plots} in the Supplementary Material. 
It is a known fact in biology (see \cite{tooze_introduction_1998}) that the $\alpha$-helix occurs when the consecutive ($\phi$, $\psi$) angle pairs are around $(-1.05 \text{ rads}, -0.87 \text{ rads})$ (or equivalently $(-60^{\circ}, -50^{\circ})$), for the shape of the helix to come up. 
The angle $\omega$ can only take the values $0$ and $\pi$, but usually the values are recorded with noise. In the case of our data, from the plot of Figure~\ref{fig: rose_plots}, the theoretical value is $\omega = \pi$ rads. Thus the theoretical symmetry center for our specified-center test  will be $(-1.05, -0.87, \pi)$. 
In Figure~\ref{fig: rose_plots}, the theoretical value of the mean is plotted in red and the estimated one in blue, for each angle.

As expected and confirming biologists' observations, the test for specified symmetry center $Q_{f_0}^{\ast(n); \mub}$ rejects the null hypothesis of symmetry $\mathcal{H}^{(n)}_{0;\mu}$ at all commonly used levels of significance.
We next applied the test for unspecified symmetry center, $Q_{f_0}^{\ast (n)}$ with  the TWCC as $f_0$, which  also rejects the null hypothesis of symmetry $\mathcal{H}^{(n)}_{0}$ at all commonly used levels of significance.
This supports the biologists' belief that the data is not symmetric, and thus must be modeled using skewed distributions.
The test can also be performed for only $(\phi, \psi)$, a dihedral pair which is commonly modeled in the literature (see for example \cite{mardia_mixtures_2012}, \cite{kato_mobius_2015} and \cite{singh_probabilistic_2002}), and the conclusion remains the same.

\section{Conclusion} \label{sec: conclusion}
In this paper we have proposed tests for symmetry for data on a hyper-torus, using the Le Cam approach for building optimal tests.
Two different tests are proposed, depending on whether the symmetry center is specified or needs to be calculated.
In both cases, the distribution of the test statistic under the null hypothesis of symmetry is the chi-squared distribution, with $d$ degrees of freedom.
The asymptotic distribution under local alternatives is a non-central chi-square distribution with $d$ degrees of freedom, with different centrality parameter for the two tests.
For both cases, we also provide, using Stein's Method, bounds on the distributional distance between the distribution of the test statistic and its limiting distribution.

The finite sample performance of the proposed tests is also investigated using simulation studies. 
Data from different distributions were generated, in different dimensions $d$, and, in the case of the test corresponding to the case of the unspecified median, different distributions were used to build the tests.
Our tests' power was also exhibited in a real world scenario, using data from the protein structure prediction problem.

{A problem of interest for future work is the high-dimensional setting when the dimension $d$ becomes large, potentially larger than the sample size $n$. In such a case, the asymptotic results of Theorems~\ref{thm: optimality specified median} and~\ref{thm: optimality unspecified median} need to get revisited under the light of the double limit of $n$ and $d$ going to infinity, with the tough task of determining asymptotic regimes of $d$ as a function of $n$ allowing to find a valid asymptotic distribution. This requires rewriting the test statistic (typically by subtracting the mean $d$ and dividing by the standard deviation $\sqrt{2d}$) and using special Central Limit Theorems such as martingale CLTs used in \cite{MR3474832}. The test statistic in~\eqref{eq: semi-parametric test unspecified} has the additional challenge that the location parameter estimate is present in the test statistic. }

\textbf{Acknowledgments:}
{The authors are grateful to the Associate Editor and two anonymous reviewers for helpful comments and suggestions that led to an improvement of the present paper. The authors would also}  like to thank Thomas Hamelryck and Ola R\o nning for the protein data and Ivan Nourdin for constructive conversations. \\

\textbf{Funding:}
The third author was funded by the Luxembourg National Research Fund (FNR), grant reference PRIDE/21/16747448/MATHCODA.

\bibliographystyle{abbrv}
\bibliography{references}

\clearpage

\begingroup
\renewcommand{\thepage}{S\arabic{page}} 
\renewcommand{\thesection}{S\arabic{section}} 
\renewcommand{\thefigure}{S\arabic{figure}}
\renewcommand{\thetable}{S\arabic{table}}
\renewcommand{\theequation}{S\arabic{equation}}
\renewcommand{\thetheorem}{S\arabic{theorem}}
\renewcommand{\thelemma}{S\arabic{lemma}}
\renewcommand{\theassumption}{S\arabic{assumption}}

\setcounter{page}{1}                    
\setcounter{section}{0}
\setcounter{equation}{0}
\setcounter{figure}{0}
\setcounter{table}{0}

\begin{center}
  \LARGE \textbf{Supplement to ``Construction of optimal tests for symmetry on the torus and their quantitative error bounds''} 
\end{center}

\section{Supporting plots} \label{sec: supporting_plots}

In this section we provide the figures that are referenced in the main paper.
More specifically, Figures~\ref{fig: sim_twcc} and \ref{fig: theor_power_3d} concern Section~\ref{sec: simulations}, while Figure~\ref{fig: rose_plots} concerns Section~\ref{sec: real data}.
Comments on these plots are provided in their respective sections.

\begin{figure}[htbp]
\centering
\begin{minipage}{5cm}
    \centering
\includegraphics[scale=0.55]{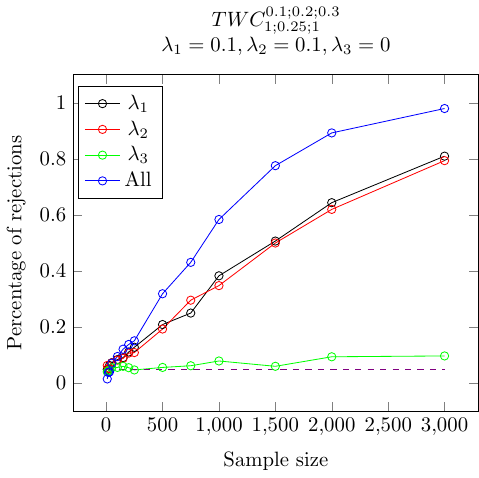}
\end{minipage}%
\begin{minipage}{5cm}
    \centering
\includegraphics[scale=0.55]{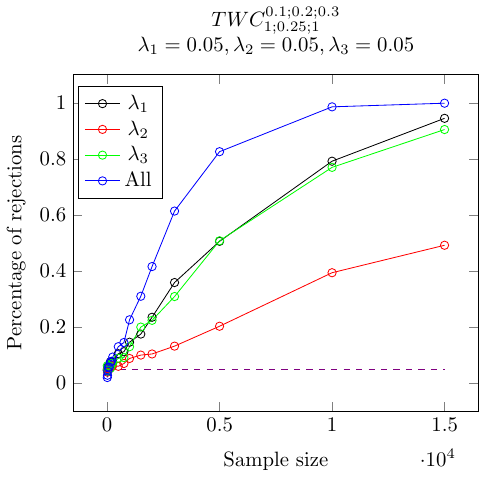}
\end{minipage}%
\begin{minipage}{5cm}
    \centering
\includegraphics[scale=0.55]{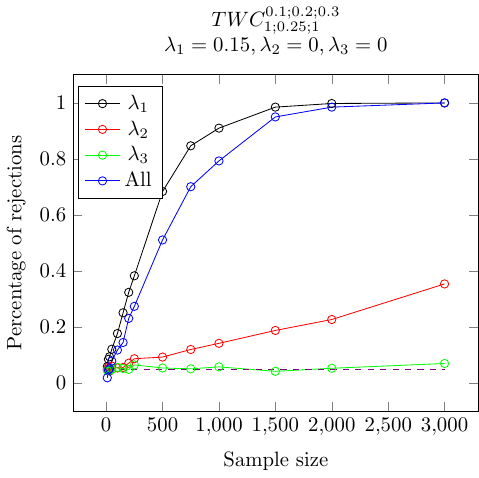}
\end{minipage}

\caption{\label{fig: sim_twcc} Results of Monte Carlo estimates of percentage of rejection of different null hypotheses of specified center symmetry, with increasing sample size and $f_0$ being the TWC with parameters $\rho_{12}=\rho_{23}=1, \rho_{13}=0.25$ and wrapped Cauchy marginal distributions with parameters $\beta_1=0.1, \beta_2=0.2, \beta_3=0.3$. The black line corresponds to the test $\mathcal{H}^{(n)}_{0; \mu;1}:\lambda_1=0 \text{ vs } \mathcal{H}^{(n)}_{1; \mu;1}:\lambda_1\neq0$, the red and green lines correspond to the analogous tests for $\lambda_2$ and $\lambda_3$ and the blue line corresponds to the test $\mathcal{H}^{(n)}_{0; \mu}:\boldsymbol{\lambda} = 0 \text{ vs } \mathcal{H}^{(n)}_{1; \mu}:\boldsymbol{\lambda} \neq 0$. The purple dotted line indicates 0.05. }
\end{figure}

\begin{figure}[htpb]
	\begin{tabular}{@{}cccc@{}}
		\centering 
		\includegraphics[scale = 0.4]{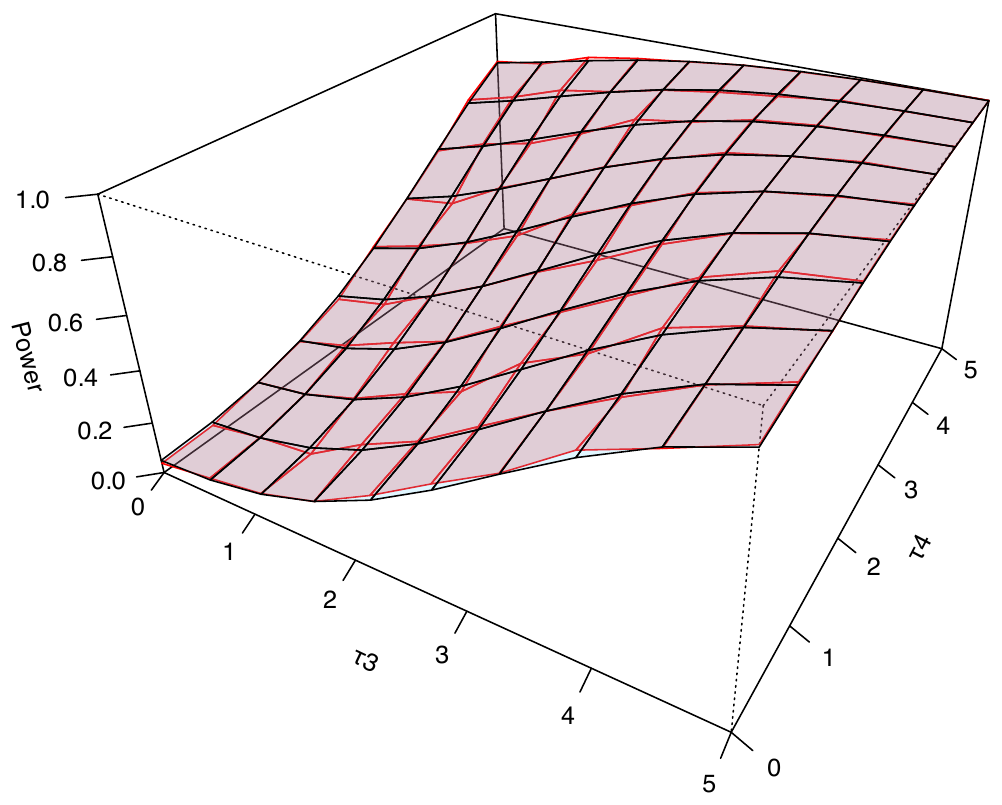}&
   \includegraphics[scale = 0.45]{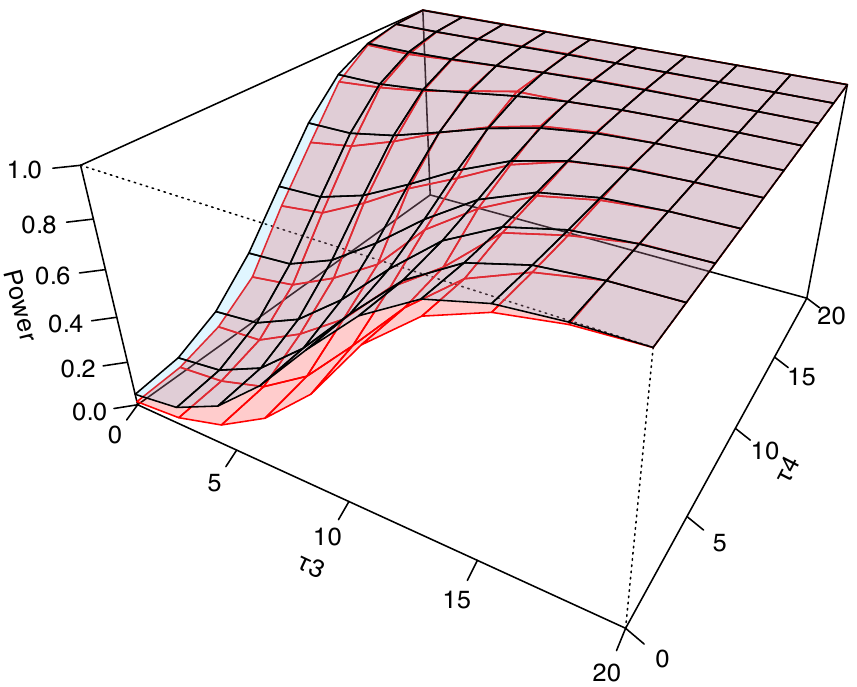} \\
		\hspace{0.3cm} (a) & \hspace{0.45cm}(b) \vspace{0.2cm}\\ 
		\includegraphics[scale = 0.45]{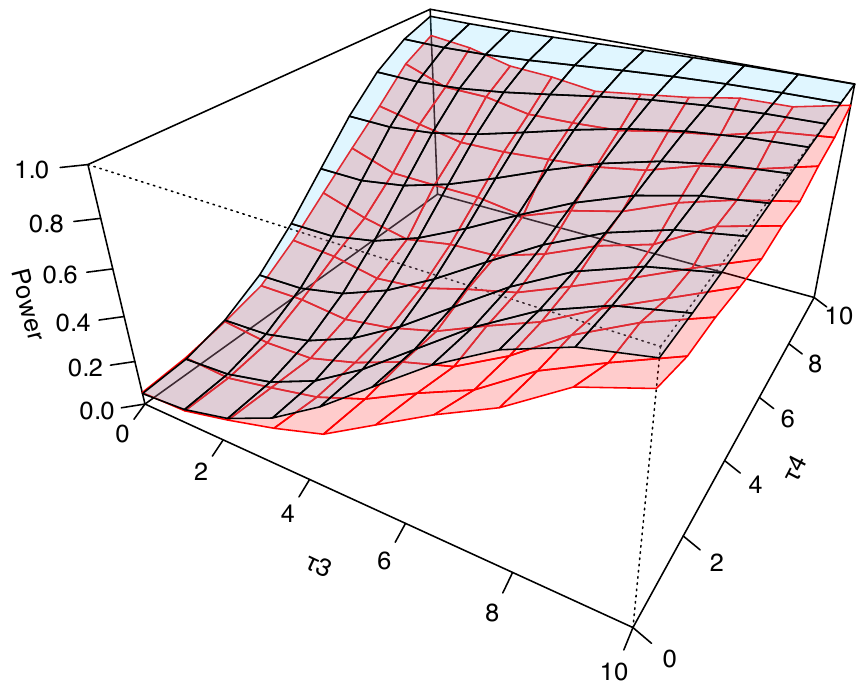} &
		\includegraphics[scale = 0.45]{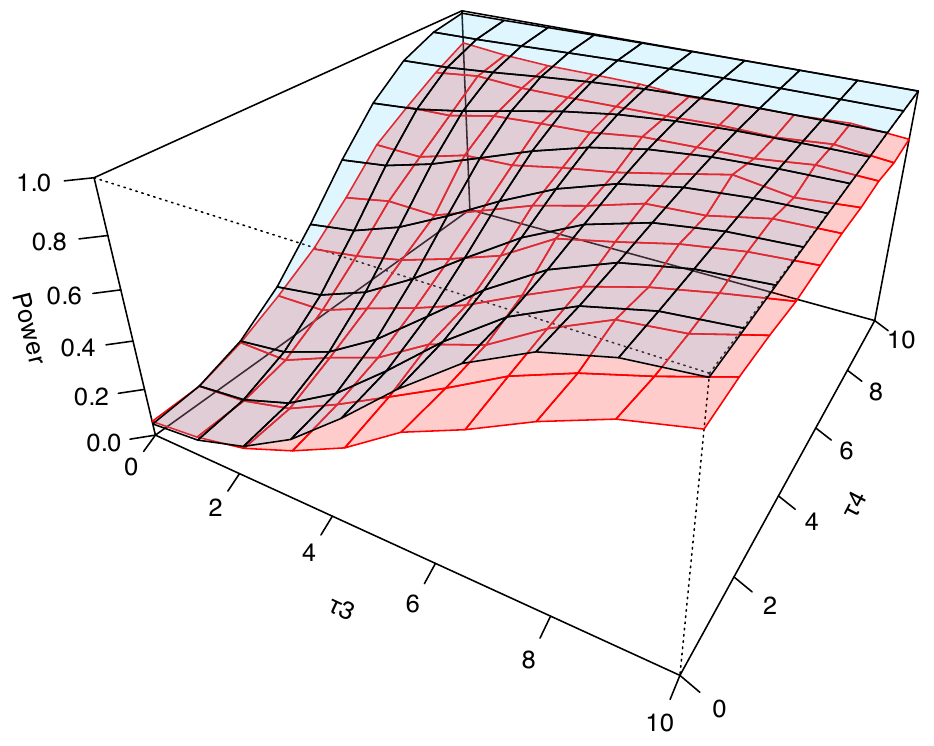} \\	
		\hspace{0.45cm}(c) & \hspace{0.45cm}(d) 
\end{tabular}
\caption{\label{fig: theor_power_3d} Plots of the theoretical (light blue) and simulated (red) power of the test for (a) $\phi^{\ast (n); \mu}$ with $f_0 = BWC_{0.1, 0.5, 0.4}$, (b) $\phi^{\ast (n)}_{f_0}$ with $f_0 = g_0 = BWC_{0.9, 0.9, 0.4}$, (c) $\phi^{\ast (n)}_{f_0}$ with $f_0 = BWC_{0.8, 0.8, 0.6}$, $g_0 = BWC_{0.3, 0.3, 0.4}$ and (d) $\phi^{\ast (n)}_{f_0}$ with $f_0 = S_{0.2, 0.2, 0.4}$, $g_0 = S_{0.1, 0.1, 0.1}$.}
\end{figure}

\begin{figure}[htbp]
    \centering 
     \includegraphics[bb= 40 140 1000 450, clip=true, scale = 0.6]{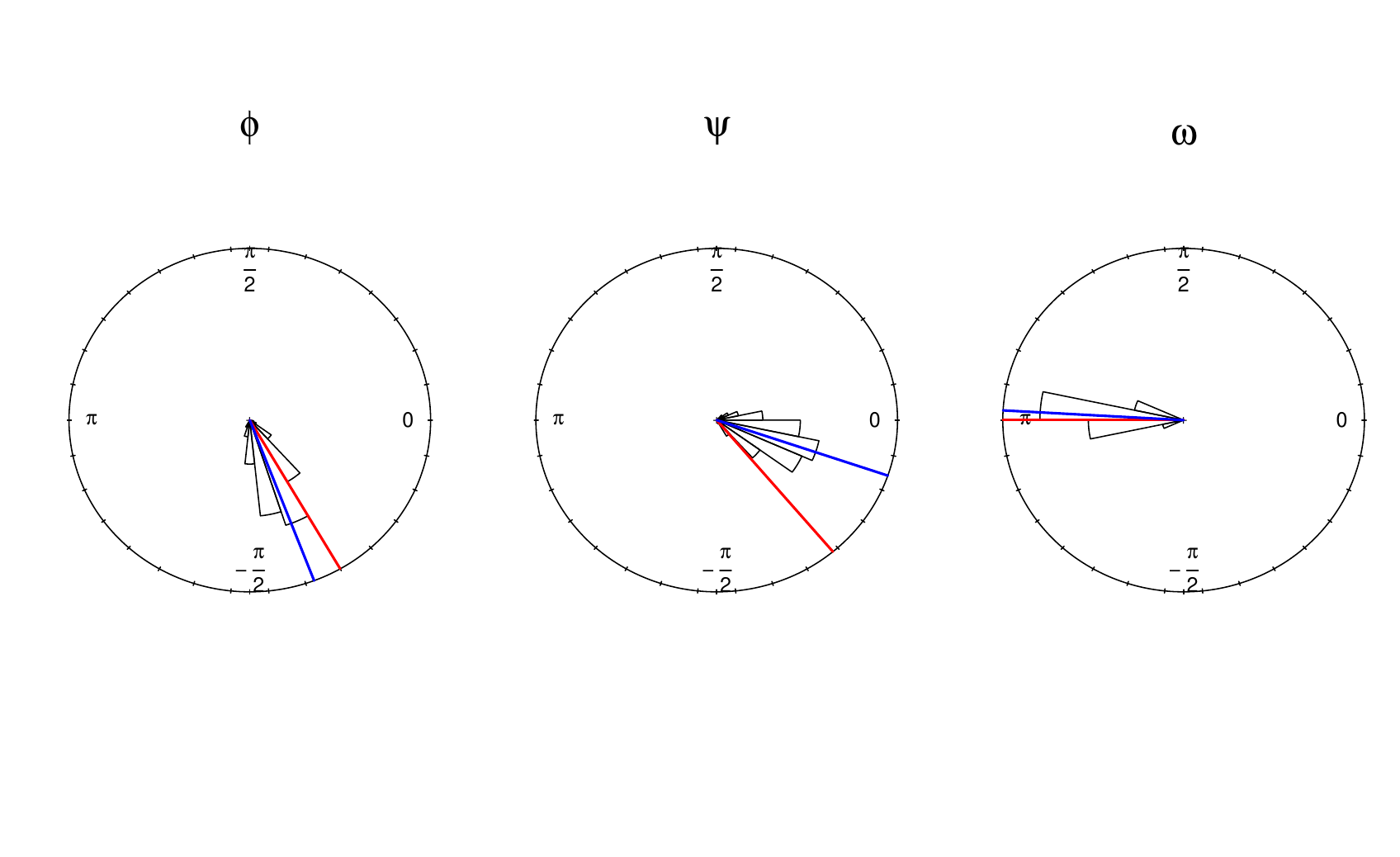}
\caption[]{Rose plots of the protein data. The theoretical mean $(-1.05, -0.87, \pi)$ is plotted in red and the estimated one $(-1.21, -0.33, 3.08)$ is plotted in blue. \label{fig: rose_plots}}
\end{figure}

\section{Proofs}\label{sec: proof_supplement}

\subsection{Proof of Proposition \ref{prop: QMD}} \label{appendix: proof prop QMD}

Before the proof, we state for the sake of readability the definition of quadratic mean differentiable functions (see Definition 12.2.1 from \cite{lehmann_quadratic_2005}), which is needed to prove the proposition.
\begin{definition} \label{def: QMD}
    The model $\left( P_\theta: \theta\in\Theta \right)$ is said to be \textit{differentiable in quadratic mean} if at an inner point $\theta_0\in\Theta$  there exists a measurable function $\dot{\ell}_{\theta_0}: \mathcal{X}\rightarrow \R^k$ such that 
    \begin{equation} \label{eq: definition QMD}
        \int_\mathcal{X} \left\{ \sqrt{p_{\theta_0+h}} - \sqrt{p_{\theta_0}} - \frac{1}{2}h'\dot{\ell}_{\theta_0}\sqrt{p_{\theta_0}} \right\}^2 {\mathrm d} \mu = o\left(||h||^2 \right) \quad \text{as } ||h||\rightarrow 0,
    \end{equation}
    where $p_{\theta}$ is the density of $P_\theta$ with respect to some measure $\mu$ and $\frac{1}{2}\dot{\ell}_{\theta_0}\sqrt{p_{\theta_0}}$ is the quadratic mean derivative at $\theta_0$.
\end{definition}

\begin{proof}[Proof of Proposition \ref{prop: QMD}]
``$\Leftarrow$" is immediate by taking $\hb = (h_1^1, \ldots, h_1^k, 0, \ldots, 0)'$ in \eqref{eq: definition QMD}. \\

\noindent ``$\Rightarrow$" Assume that $\mub\mapsto f(\xb; \mub)$ is QMD. 
Using \eqref{eq: definition QMD} and defining
\begin{align*}
    A := \int_{\mathcal{X}} \Biggl(  
f^{1/2}(\xb; \mub_0 + \hb_1) & g^{1/2}(\xb; \mub_0 + \hb_1, \hb_2) - f^{1/2}(\xb; \mub_0) \\
- & \hb'\left. \begin{pmatrix}
        \nabla_{\mub} f^{1/2}(\xb; \mub) \\ f^{1/2}(\xb; \mub) \nabla_{\lambdab} g^{1/2}(\xb; \mub, \lambdab)
    \end{pmatrix}
    \right|_{
 \substack{\scriptscriptstyle{\mub=\mub_0} \\ \scriptscriptstyle{\lambdab = \boldsymbol{0}}}} \Biggr)^2 {\mathrm d} \xb 
\end{align*}
for $\hb = (\hb_1', \hb_2')'$, $\hb_1 \in\R^k$ and $\hb_2 \in\R^m$, we need to prove that $A = o(||\hb|| ^2)$. 
Condition (iii) of the statement of the proposition ensures that the partial derivatives of the square roots of the functions are well defined.
We can bound $A$  by $3(A_1 + A_2 + A_3)$, where
\begin{align*}
     A_1 : & = \int_{\mathcal{X}} \Biggl(  
f^{1/2}(\xb; \mub_0 + \hb_1) - f^{1/2}(\xb; \mub_0) - \hb_1'\left.
        \nabla_{\mub} f^{1/2}(\xb; \mub)
    \right|_{
{\scriptscriptstyle{\mub=\mub_0}}} \Biggr)^2 {\mathrm d} \xb, \\
 A_2 :& = \int_{\mathcal{X}} \Biggl(  
f^{1/2}(\xb; \mub_0 + \hb_1)g^{1/2}(\xb;\mub_0 + \hb_1, \hb_2) 
- f^{1/2}(\xb; \mub_0 + \hb_1) \\
& \hspace{3cm} - \hb_2'\left.
        f^{1/2}(\xb; \mub_0 + \hb_1) \nabla_{\lambdab} g^{1/2}(\xb; \mub, \lambdab)
    \right|_{
 \substack{\scriptscriptstyle{\mub = \mub_0 + \hb_1} \\ \scriptscriptstyle{\lambdab = \boldsymbol{0}}}} \Biggr)^2 {\mathrm d} \xb, \\
A_3 : & = \int_{\mathcal{X}} \Biggl( \hb_2' \biggl( \left.
        f^{1/2}(\xb; \mub_0 + \hb_1) \nabla_{\lambdab} g^{1/2}(\xb;\mub,\lambdab)
    \right|_{
 \substack{\scriptscriptstyle{\mub = \mub_0 + \hb_1} \\ \scriptscriptstyle{\lambdab = \boldsymbol{0} }}} \\
 & \hspace{3cm} 
- f^{1/2}(\xb; \mub_0) \left. \nabla_{\lambdab} g^{1/2}(\xb;\mub,\lambdab)
    \right|_{
 \substack{\scriptscriptstyle{\mub = \mub_0} \\ \scriptscriptstyle{\lambdab = \boldsymbol{0} }}} \biggr) \Biggr)^2 {\mathrm d} \xb. 
\end{align*}
Showing that $A_1 = o(\parallel \hb_1 \parallel^2)$ and $A_2 = o(\parallel \hb_2 \parallel^2) = A_3$ implies that $A = o(||\hb|| ^2)$. 
Note that conditions (ii) and (v) in the statement of the proposition are needed for $A_1$, $A_2$ and $A_3$ to be finite.
The result concerning $A_1$ follows immediately from Definition \ref{def: QMD} and the fact that $\mub\mapsto f(\xb; \mub)$ is QMD. For the part involving $A_2$, we rewrite it  as
\begin{equation*}
    A_2 = \int_{\mathcal{X}} f(\xb; \mub_0 + \hb_1) \Biggl( g^{1/2}(\xb; \mub_0 + \hb_1, \hb_2) - 1 - \hb_2'\left.
        \nabla_{\lambdab} g^{1/2}(\xb; \mub, \lambdab)
    \right|_{
    \substack{\scriptscriptstyle{\mub = \mub_0 + \hb_1} \\ \scriptscriptstyle{\lambdab = \boldsymbol{0} }}} \Biggr)^2 {\mathrm d} \xb. 
\end{equation*}
From assumption (i), $g$ is $\mathcal{C}^1$ in $\lambdab$, hence by the Mean Value Theorem it holds that 
\begin{align*}
    \hb_2' \nabla_{\lambdab} \left.g^{1/2}(\xb; \mub, \lambdab) \right|&_{ \substack{\scriptscriptstyle{\mub = \mub_0 + \hb_1} \\ \scriptscriptstyle{\lambdab = \lambdab^{\ast}_{\hb_2}}}} \\
    = g^{1/2} & (\xb; \mub_0 + \hb_1, \hb_2) - g^{1/2}(\xb; \mub_0 + \hb_1, \boldsymbol{0})
    = g^{1/2}(\xb; \mub_0 + \hb_1, \hb_2) - 1
\end{align*}
for some $\lambdab^{\ast}_{\hb_2}$ which lies on the line connecting $\boldsymbol{0}$ and $\hb_2$, where the last equality follows from assumption (iv). 
This entails that
\begin{align*}
    A_2 & = \int_{\mathcal{X}} f(\xb; \mub_0 + \hb_1)\Biggl(  
 \hb_2' \left(\nabla_{\lambdab} \left.g^{1/2}(\xb; \mub, \lambdab) \right|_{ \substack{\scriptscriptstyle{\mub = \mub_0 + \hb_1} \\ \scriptscriptstyle{\lambdab = \lambdab^{\ast}_{\hb_2}}}} - \left.
        \nabla_{\lambdab} g^{1/2}(\xb; \mub, \lambdab)
    \right|_{
    \substack{\scriptscriptstyle{\mub = \mub_0 + \hb_1} \\ \scriptscriptstyle{\lambdab = \boldsymbol{0} }}} \right)\Biggr)^2 {\mathrm d} \xb \\
& \leq ||\hb_2||^2\int_{\mathcal{X}} f(\xb; \mub_0 + \hb_1)\Biggl|\Biggl|  
 \nabla_{\lambdab} \left.g^{1/2}(\xb; \mub, \lambdab) \right|_{ \substack{\scriptscriptstyle{\mub = \mub_0 + \hb_1} \\ \scriptscriptstyle{\lambdab = \lambdab^{\ast}_{\hb_2}}}} - \left.
        \nabla_{\lambdab} g^{1/2}(\xb; \mub, \lambdab)
    \right|_{
    \substack{\scriptscriptstyle{\mub = \mub_0 + \hb_1} \\ \scriptscriptstyle{\lambdab = \boldsymbol{0} }}} \Biggr|\Biggr|^2 {\mathrm d}\xb, 
\end{align*}
where the last inequality follows by Cauchy-Schwarz. The triangle inequality combined with assumption (ii) yields 
\begin{equation*}
    \Biggl|\Biggl| \nabla_{\lambdab} \left.g^{1/2}(\xb; \mub, \lambdab)\right|_{ \substack{\scriptscriptstyle{\mub = \mub_0 + \hb_1} \\ \scriptscriptstyle{\lambdab = \lambdab^{\ast}_{\hb_2}}}} - \left.
        \nabla_{\lambdab} g^{1/2}(\xb; \mub, \lambdab)
    \right|_{
 \substack{\scriptscriptstyle{\mub = \mub_0 + \hb_1} \\ \scriptscriptstyle{\lambdab = \boldsymbol{0} }}} \Biggr|\Biggr| \leq 2 C \ell(\xb; \mub_0 + \hb_1)
\end{equation*}
for $C \in \R^{+}$  independent of $\lambda^{\ast}_{\hb_2}$ and hence of $\hb_2$. The Dominated Convergence Theorem allows us to conclude  that $A_2 = o(\parallel \hb_2\parallel^2)$.\\

\noindent Finally for $A_3$, we also have by the triangle inequality and assumption (ii) that 
\begin{align*}
    &\Biggl|\Biggl| \left.
        f^{1/2}(\xb; \mub_0 + \hb_1) \nabla_{\lambdab} g^{1/2}(\xb; \mub, \lambdab)
    \right|_{
 \substack{\scriptscriptstyle{\mub = \mub_0 + \hb_1} \\ \scriptscriptstyle{\lambdab = \boldsymbol{0} }}}
- f^{1/2}(\xb; \mub_0) \left. \nabla_{\lambdab} g^{1/2}(\xb; \mub, \lambdab)
    \right|_{
 \substack{\scriptscriptstyle{\mub = \mub_0} \\ \scriptscriptstyle{\lambdab = \boldsymbol{0} }}}\Biggr|\Biggr| \\
&\leq f^{1/2}(\xb; \mub_0 + \hb_1) \Biggl|\Biggl| \left. \nabla_{\lambdab} g^{1/2}(\xb; \mub, \lambdab)
    \right|_{
 \substack{\scriptscriptstyle{\mub = \mub_0 + \hb_1} \\ \scriptscriptstyle{\lambdab = \boldsymbol{0} }}} \Biggr|\Biggr|
+ f^{1/2}(\xb; \mub_0) \Biggl|\Biggl| \left. \nabla_{\lambdab} g^{1/2}(\xb; \mub, \lambdab)
    \right|_{
 \substack{\scriptscriptstyle{\mub = \mub_0} \\ \scriptscriptstyle{\lambdab = \boldsymbol{0} }}}\Biggr|\Biggr|\\
&\leq C \left(f^{1/2}(\xb;\mub_0 + \hb_1)\ell(\xb; \mub_0+\hb_1)
+ f^{1/2}(\xb; \mub_0)\ell(\xb; \mub_0)\right)
\end{align*}
which is an integrable function not depending on $\hb_2$. The Dominated Convergence Theorem thus applies and gives $A_3 = o(\parallel \hb_2 \parallel^2)$.
\end{proof}

\subsection{Proof of Proposition \ref{prop: ULAN}}\label{appendix: proof_ULAN}
For the sake of a better understanding, we recall here Lemma 2.3 of \cite{garel_local_1995}, on which the proof is based.
\begin{lemma} \label{lemma_LAN}
Denote by $P^{(n)}_1$ and $P^{(n)}_0$ two sequences of probability measures on measurable spaces $\left( \mathcal{X}^{(n)}, \mathcal{A}^{(n)} \right)$.
For all $n$, let $\mathcal{A}_t^{(n)} \subset \mathcal{A}_{t+1}^{(n)}$ be a filtration such that $\mathcal{A}_n^{(n)} = \mathcal{A}^{(n)}$, and denote by $P^{(n)}_{1,t}$ and $P^{(n)}_{0,t}$ the restrictions to $\mathcal{A}_t^{(n)}$ of $P^{(n)}_1$ and $P^{(n)}_0$, respectively.
Assuming that $P^{(n)}_{1,t}$ is absolutely continuous (on $\mathcal{A}_t^{(n)}$) with respect to $P^{(n)}_{0,t}$, let $\alpha_0^{(n)} = 1, \alpha_t^{(n)} = dP^{(n)}_{1,t} / dP^{(n)}_{0,t}$ and $\xi_t^{(n)} = \left( \alpha_t^{(n)} / \alpha_{t-1}^{(n)} \right)^{1/2}-1$. 
    Assume that the random variables $\zeta_t^{(n)}$ satisfy the following conditions (all convergences are in $P^{(n)}_0$-probability, as $n\rightarrow\infty$; expectations are also taken with respect to $P^{(n)}_0$):
    \begin{enumerate}
        \item[(i)] $\E \left[ \sum_{t=1}^n \left( \zeta_t^{(n)} - \xi_t^{(n)} \right)^2 \right] \rightarrow 0$;
        
        \item[(ii)] $\sup_n \E \left[ \sum_{t=1}^n \left( \zeta_t^{(n)} \right)^2 \right] < \infty$;
        
        \item[(iii)] $\max_{1\leq t\leq n} \left| \zeta_t^{(n)} \right| \rightarrow 0$;
        
        \item[(iv)] $\sum_{t=1}^n \left( \zeta_t^{(n)} \right)^2 - \frac{\left( \uptau^{(n)} \right)^2}{4} \rightarrow 0 $ for some non-random sequence $\left( \uptau^{(n)}, n \in \N \right)$ such that \linebreak $\sup_n\left( \uptau^{(n)} \right)^2 < \infty$;
        
        \item[(v)] $\sum_{t=1}^n \E \left\{ \left( \zeta_t^{(n)} \right)^2 I \left[ \left| \zeta_t^{(n)} \right| > \frac{1}{2} \right] \Bigg| \mathcal{A}^{(n)}_{t-1} \right\} \rightarrow 0$;

        \item[(vi)] $\E \left[ \zeta_t^{(n)} | \mathcal{A}^{(n)}_{t-1} \right] = 0$;

        \item[(vii)] $\sum_{t=1}^n \E \left[ \left( \xi_t^{(n)} \right)^2 + 2\xi_t^{(n)} | \mathcal{A}^{(n)}_{t-1}\right] \rightarrow 0$.
    \end{enumerate}
    Then, under $P_0^{(n)}$, as $n \rightarrow \infty$,
    \begin{equation*}
        \Lambda^{(n)} 
        = \log \left( dP_1^{(n)} / P_0^{(n)} \right) 
        = 2 \sum_{t=1}^n \zeta_t^{(n)} - \left( \uptau^{(n)} \right)^2 / 2 + o_P(1),
    \end{equation*}
    and the distribution of $\left[ \Lambda^{(n)} + \left( \uptau^{(n)} \right)^2 / 2 \right] / \uptau^{(n)}$ is asymptotically standard normal.
\end{lemma}

Before proceeding to the proof of Proposition \ref{prop: ULAN}, we match the  notation from this lemma with that in our paper:
\begin{align*}
    & \alpha_t^{(n)} 
    = \frac{dP^{(n)}_{1,t}} {dP^{(n)}_{0,t}}
    = \prod_{i=1}^t \frac{f_{\mub^{(n)} + n^{-1/2} \taub^{(n)}_{\mu}, n^{-1/2} \taub^{(n)}_{\lambda}}(\thetab_{i})}{f_{\mub^{(n)}, \boldsymbol{0}}(\thetab_{i})} \\
    & \xi_t^{(n)}
    = \left( \frac{f_{\mub^{(n)} + n^{-1/2} \taub^{(n)}_{\mu}, n^{-1/2} \taub^{(n)}_{\lambda}}(\thetab_{t})}{f_{\mub^{(n)}, \boldsymbol{0}}(\thetab_{t})} \right)^{1/2} - 1,  \\
    & \zeta_t^{(n)}
    = \frac{1}{2\sqrt{n}} \taub^{(n)\prime} \Delta_{f_0; t}^{(n)}(\mub^{(n)})
    = \frac{1}{2\sqrt{n}} \Biggl( \sum_{j=1}^d \tau_j^{(n)} \phi^{f_0}_j(\thetab_{t} - \mub^{(n)})
    + \sum_{j=d+1}^{2d} \tau_j^{(n)} \sin(\theta_{tj}-\mu_j^{(n)}) \Biggr), \\
    & \left( \uptau^{(n)} \right)^2 
    = \taub^{(n)\prime}
    \Gamma_{f_0} 
    \taub^{(n)}. 
\end{align*}
Note that $\sup_n \left( \uptau^{(n)} \right) < \infty$, so the assumption in (iv) is satisfied. Now we can prove Proposition \ref{prop: ULAN}.

\begin{proof}[Proof of Proposition \ref{prop: ULAN}]
\noindent Proving that Lemma \ref{lemma_LAN} holds in our setting for family $P^{(n)}_{f_0}$ implies that Proposition \ref{prop: ULAN} holds. 
Thus, we need to check that the conditions of Lemma \ref{lemma_LAN} hold.\\

\textbf{(i)} Straightforward calculations yield
\begin{align} \label{eq: condition_1}
    \E & \left[ \sum_{t=1}^n \left( \zeta_t^{(n)} - \xi_t^{(n)} \right)^2 \right] 
    = n \ \E \left[\left( \frac{\taub^{(n)\prime}} {2\sqrt{n}} \Delta_{f_0}^{(n)}(\mub^{(n)}) - \frac{f^{1/2}_{\mub^{(n)} + n^{-1/2} \taub^{(n)}_{\mu}, n^{-1/2} \taub^{(n)}_{\lambda}}(\thetab)} {f^{1/2}_{\mub^{(n)}, \boldsymbol{0}}(\thetab)} + 1 \right)^2 \right] \nonumber\\
    & = n \int_{[-\pi, \pi)^d} \Biggl( \frac{\taub^{(n)\prime}} {2\sqrt{n}} \nabla \log \left( f_{\mub^{(n)}, \boldsymbol{0}}(\thetab) \right) - \frac{f^{1/2}_{\mub^{(n)} + n^{-1/2} \taub^{(n)}_{\mu}, n^{-1/2} \taub^{(n)}_{\lambda}}(\thetab)}{f^{1/2}_{\mub^{(n)}, \boldsymbol{0}}(\thetab)} + 1 \Biggr)^2 f_{\mub^{(n)}, \boldsymbol{0}}(\thetab) \ {\mathrm d} \thetab \nonumber\\
    & = n \int_{[-\pi, \pi)^d} \left( f^{1/2}_{\mub^{(n)} + n^{-1/2} \taub^{(n)}_{\mu}, n^{-1/2} \taub^{(n)}_{\lambda}}(\thetab) - f^{1/2}_{\mub^{(n)}, \boldsymbol{0}}(\thetab) - \frac{\taub^{(n)\prime}}{2\sqrt{n}} \frac{\nabla f_{\mub^{(n)}, \boldsymbol{0}}(\thetab)} {f^{1/2}_{\mub^{(n)}, \boldsymbol{0}}(\thetab)} \right)^2 {\mathrm d} \thetab \nonumber \\
    & = n \int_{[-\pi, \pi)^d} \left( f^{1/2}_{\mub^{(n)} + n^{-1/2} \taub^{(n)}_{\mu}, n^{-1/2} \taub^{(n)}_{\lambda}} (\thetab) - f^{1/2}_{\mub^{(n)}, \boldsymbol{0}}(\thetab) - \frac{\taub^{(n)\prime}} {\sqrt{n}} \nabla f^{1/2}_{\mub^{(n)}, \boldsymbol{0}}(\thetab) \right)^2 {\mathrm d} \thetab.
\end{align}
Note that in these calculations the gradient is taken with respect to $(\mub',\lambdab')'$ and then evaluated at $(\mub^{(n)\prime}, \boldsymbol{0}')'$, but for the sake of readability we used a simplified notation.
By Proposition \ref{prop: QMD} and Assumption \ref{assumption_on_f0}, $\varthetab\mapsto f_{\varthetab}(\thetab)$ is QMD  at $(\mub^{(n)\prime}, \boldsymbol{0}')'$, which entails that
$\eqref{eq: condition_1} = o(n^{-1})$. 

\textbf{(ii)} Reminding that the Fisher information matrix is the variance of the central sequence, we get
\begin{equation*}
    \E \left[ \sum_{t=1}^n \left( \zeta_t^{(n)} \right)^2 \right]
    = n \int_{[-\pi, \pi)^d} \left( \zeta_t^{(n)} \right)^2 f_{\mub^{(n)}, \boldsymbol{0}}(\thetab) {\mathrm d} \thetab
    = \frac{1}{4} \left( \uptau^{(n)} \right)^2
\end{equation*}
and, by assumption, $\sup_n\left( \uptau^{(n)} \right)^2 < \infty$. 

\textbf{(iii)} By definition, \begin{align*}
    \left| \zeta_t^{(n)} \right| 
    &= \frac{1}{2\sqrt{n}} \Biggl| \sum_{j=1}^d \tau_j^{(n)} \phi^{f_0}_j(\thetab_{t} - \mub^{(n)})
    + \sum_{j=d+1}^{2d} \tau_j^{(n)} \sin(\theta_{tj}-\mu_j^{(n)}) \Biggr|\\
    &\leq \frac{1}{2\sqrt{n}} \Biggl| \sum_{j=1}^d \tau_j^{(n)} \phi^{f_0}_j(\thetab_{t} - \mub^{(n)})\Biggr|
    + \frac{1}{2\sqrt{n}} \Biggr|\sum_{j=d+1}^{2d} \tau_j^{(n)} \sin(\theta_{tj}-\mu_j^{(n)}) \Biggr|.
\end{align*}
Since $\taub^{(n)}$ is bounded and $\sin(x) \in \left[ -1,1 \right]$, the second summand goes to zero as $n\to\infty$. For any $\thetab_{t}, \mub \in \left[-\pi, \pi \right)^d$, since $f_0(\cdot) > 0$ and continuous, $\phi^{f_0}_j(\thetab_{t} - \mub^{(n)})$ is finite for $j=1, \ldots, d$, implying that the first term also converges to 0. 
Consequently  $\left| \zeta_t^{(n)} \right| \rightarrow 0$ as $n\to\infty$.

\textbf{(iv)}
Bearing in mind that the variance of the central sequence is the Fisher information matrix $\Gamma_{f_0}$, we get by the Law of Large Numbers that
\begin{align*}
    \sum_{t=1}^n \left( \zeta_t^{(n)} \right)^2 
    =\frac{1}{4n} \taub^{(n)\prime}\sum_{t=1}^n
   \Delta_{f_0; t}^{(n)}(\mub^{(n)})\left(\Delta_{f_0; t}^{(n)}(\mub^{(n)})\right)'
    \taub^{(n)} \\
    = \frac{1}{4}\taub^{(n)\prime}
    \Gamma_{f_0}
     \taub^{(n)}  + o_P(1)
    = \frac{\left( \uptau^{(n)} \right)^2}{4} + o_P(1)
\end{align*}
as $n\to\infty$.

\textbf{(v)} 
Using the same argument as in \textbf{(iii)}, $\left| \zeta_t^{(n)} \right| \rightarrow 0$ as $n\rightarrow\infty$, so there exists $M\in\N$ such that, $\forall n > M, \left| \zeta_t^{(n)} \right| \leq \frac{1}{2}$. Consequently 
\begin{align*}
    \lim_{n\rightarrow\infty} \sum_{t=1}^n \E & \left\{ \left( \zeta_t^{(n)} \right)^2 I \left[ \left| \zeta_t^{(n)} \right| > \frac{1}{2} \right] \Bigg| \mathcal{A}^{(n)}_{t-1} \right\} \\
    & = \lim_{n\rightarrow\infty} \sum_{t=1}^n \E \left\{ \left( \zeta_t^{(n)} \right)^2 I \left[ \left| \zeta_t^{(n)} \right| > \frac{1}{2} \right] \right\} \\
    & = \lim_{n\rightarrow\infty} \sum_{t=1}^n \E \left\{ \left( \frac{1}{2\sqrt{n}} \taub^{(n)} \Delta_{f_0; t}^{(n)}(\mub^{(n)}) \right)^2 I \left[ \left| \zeta_t^{(n)} \right| > \frac{1}{2} \right] \right\} \\
    & = \lim_{n\rightarrow\infty} \sum_{t=1}^n \frac{1}{4n} \E \left\{ \left( \taub^{(n)} \Delta_{f_0;t}^{(n)}(\mub^{(n)}) \right)^2 I \left[ \left| \zeta_t^{(n)} \right| > \frac{1}{2} \right] \right\} \\
    & = \lim_{n\rightarrow\infty} \frac{1}{4} \E \left\{ \left( \taub^{(n)} \Delta_{f_0;1}(\mub) \right)^2 I \left[ \left| \zeta_t^{(n)} \right| > \frac{1}{2} \right] \right\} \\
    & = \frac{1}{4} \E \left\{ \lim_{n \rightarrow \infty} \left( \taub^{(n)} \Delta_{f_0}(\mub) \right)^2 I \left[ \left| \zeta_t^{(n)} \right| > \frac{1}{2} \right] \right\} = 0,
\end{align*}
where in the third last line the dependence on $n$ in $\Delta_{f_0;1}^{(n)}(\mub\n)$ only comes from $\mub\n$ which can be replaced by $\mub$ by a simple change of variables and where in the last line the Dominated Convergence Theorem allows us to enter the limit inside the integral since $\left( \taub^{(n)} \Delta_{f_0;1}(\mub) \right)^2 I \left[ \left| \zeta_t^{(n)} \right| > \frac{1}{2} \right]\leq \sup_n \left| \taub^{(n)} \right|^2\Delta_{f_0;1}(\mub)$ which is integrable.

\textbf{(vi)} By independence we readily have $
    \E \left[ \zeta_t^{(n)} | \mathcal{A}^{(n)}_{t-1} \right] 
    = \E \left[ \zeta_t^{(n)} \right]
    = 0.$
    
\textbf{(vii)} We have
\begin{align*}
    \sum_{t=1}^n \E \left[ \left( \xi_t^{(n)} \right)^2 + 2\xi_t^{(n)} | \mathcal{A}^{(n)}_{t-1}\right]
    & = \sum_{t=1}^n \E \left[ \frac{f_{\mub^{(n)} + n^{-1/2} \taub^{(n)}_{\mu}, n^{-1/2} \taub^{(n)}_{\lambda}}(\thetab_{t})}{f_{\mub^{(n)}, \boldsymbol{0}}(\thetab_{t})} - 1 \right] \\
    & = n \int_{[-\pi, \pi)^d} \left( f_{\mub^{(n)} + n^{-1/2} \taub^{(n)}_{\mu}, n^{-1/2} \taub^{(n)}_{\lambda}}(\thetab) - f_{\mub^{(n)}, \zerob}(\thetab) \right) {\mathrm d} \thetab \\
    & = 0.
\end{align*}
\end{proof}

\subsection{Proof of Theorem \ref{thm: optimality specified median}}\label{appendix: proof optimality specified}

\begin{proof}
\textbf{(i)} Using the Central Limit Theorem (CLT), $\Delta_{\lambda}^{(n)}(\mub) \xrightarrow{\mathcal{D}} \mathcal{N}_d\left( \boldsymbol{0}, \Gamma_{f_0; \lambda} \right)$ as $n\rightarrow\infty$.
So, for $Z \sim \mathcal{N}_d\left( \boldsymbol{0}, I_d \right)$, with $I_d$ being the $d\times d$ identity matrix, and using properties of $o_P(\cdot)$, we have that
\begin{align}
    Q^{\ast(n); \mub} 
    & = \left(\Delta_{\lambda}^{(n)}(\mub)\right)' \left(\hat{\Gamma}_{\lambda} \right)^{-1} \Delta_{\lambda}^{(n)}(\mub) \nonumber \\
    & = \left(\Delta_{\lambda}^{(n)}(\mub)\right)' \left(\Gamma_{f_0; \lambda} \right)^{-1} \Delta_{\lambda}^{(n)}(\mub) + o_P(1) \label{contig}\\
    & = \left(\Delta_{\lambda}^{(n)}(\mub)\right)' \left(\Gamma_{f_0; \lambda} \right)^{-1/2} \left(\Gamma_{f_0; \lambda} \right)^{-1/2} \Delta_{\lambda}^{(n)}(\mub) + o_P(1) \nonumber\\
    & = Z' Z + o_P(1) \nonumber
\end{align}
as $n\rightarrow\infty$, which implies the desired convergence.

\textbf{(ii)} Consider the quantity
\begin{align*}
    \Lambda'_d :& = \log \frac{dP^{(n)}_{\mub, n^{-1/2} \taub^{(n)}_\lambda; f_0}}{dP^{(n)}_{\mub,\boldsymbol{0};f_0}} 
    = \left(\taub^{(n)}_\lambda \right)'
    \Delta_{\lambda}^{(n)}(\mub) - \frac{1}{2}
    \left(\taub^{(n)}_\lambda \right)'
    \Gamma_{f_0; \lambda}
    \taub^{(n)}_\lambda 
    + o_P(1)
\end{align*}
as $n\rightarrow\infty$.
Under $P^{(n)}_{\mub,\boldsymbol{0};f_0}$, we, firstly, deduce from Proposition \ref{prop: ULAN} via a simple projection that 
$
    \Lambda'_d \xrightarrow{\mathcal{D}} \mathcal{N} \Bigl(- \frac{1}{2}
    \taub'_\lambda
    \Gamma_{f_0; \lambda}
    \taub_\lambda , \ 
    \taub'_\lambda
    \Gamma_{f_0; \lambda}
    \taub_\lambda \Bigr)
$ 
and, secondly, we readily obtain that
    $\text{Cov}\Bigl(\Delta_{\lambda}^{(n)}(\mub), \Lambda'_d \Bigr)
    \xrightarrow{n \rightarrow\infty} 
    \Gamma_{f_0; \lambda}
    \taub_\lambda$
as $n\rightarrow\infty$, and so 
\begin{equation*}
    \begin{pmatrix}
        \Delta_{\lambda}^{(n)}(\mub) \\ \Lambda'_d
    \end{pmatrix}
        \xrightarrow{\mathcal{D}} \mathcal{N}_{d+1} \Biggl( 
        \begin{pmatrix}
            \boldsymbol{0}_d \\ - \frac{1}{2}
    \taub'_\lambda
    \Gamma_{f_0; \lambda}
    \taub_\lambda
        \end{pmatrix}, \quad 
    \begin{pmatrix}
        \Gamma_{f_0; \lambda} & \Gamma_{f_0; \lambda}
    \taub_\lambda \\
    \taub'_\lambda\Gamma_{f_0; \lambda}
     & \taub'_\lambda \Gamma_{f_0; \lambda}
    \taub_\lambda
    \end{pmatrix}
        \Biggr)
\end{equation*}
as $n\rightarrow\infty$. 
Now, since $P^{(n)}_{\mub, 0; f_0}$  and $P^{(n)}_{\mub, n^{-1/2} \taub^{(n)}_\lambda; f_0}$ are mutually contiguous, using the \textit{Third Le Cam Lemma}, which can be found in \cite{verdebout_modern_2017} (Proposition 5.2.2) and holds thanks to the ULAN property, we obtain
\begin{equation*}
    \Delta_{\lambda}^{(n)}(\mub)
    \xrightarrow{\mathcal{D}} \mathcal{N}_{d}
    \left( 
    \Gamma_{f_0; \lambda}
    \taub_\lambda , 
    \Gamma_{f_0; \lambda}
    \right)
\end{equation*}
under $P^{(n)}_{\mub, n^{-1/2} \taub^{(n)}_\lambda; f_0}$ as $n \rightarrow \infty$. Also by contiguity, \eqref{contig} holds, and by setting $W = \hat{\Gamma}_{\lambda}^{-1/2} \Delta_{\lambda}^{(n)}(\mub)$, we have
$W \xrightarrow{\mathcal{D}} \mathcal{N}_d \left( \Gamma_{f_0; \lambda}^{1/2}
    \taub_\lambda, 
    I_d \right)$ as $n \rightarrow \infty$. This  entails that, as $n \rightarrow \infty$, 
\begin{equation*}
    Q^{\ast(n);\mub}
    = W' W 
    \xrightarrow{\mathcal{D}} 
    \chi^2_d(\kappa),
\end{equation*}
where $\kappa = (\Gamma_{f_0; \lambda}^{1/2}
    \taub_\lambda)' \Gamma_{f_0; \lambda}^{1/2}
    \taub_\lambda = \taub'_\lambda \Gamma_{f_0; \lambda} \taub_\lambda$.

\textbf{(iii)} In (i) we proved that $Q^{\ast(n);\mub} = Q^{(n);\mub}_{f_0} + o_P(1)$ asymptotically, so the result follows from the optimality of the $f_0$-parametric test $\phi^{(n); \mu}_{f_0}$. The $f_0$-universal optimality is obtained directly thanks to the fact that $Q^{\ast(n);\mub}$ does not depend on $f_0$.
\end{proof}

\subsection{Proof of Theorem \ref{thm: steins method specified}} \label{appendix: proof thm steins specified}
\begin{proof}
Using the triangle inequality,
\begin{align}
    \left| \E \left[ h(Q^{\ast(n); \mub}) \right] - \E \left[ h\left(\chi^2_d\right) \right] \right| 
    \ \leq \ & \left|\E \left[ h(Q_{f_0}^{(n); \mub}) \right] -\E \left[ h\left(\chi^2_d\right) \right] \right| \label{eq: steins method parametric specified} \\
    & + \left|\E \left[ h(Q^{\ast(n); \mub}) \right] -\E \left[ h(Q_{f_0}^{(n); \mub}) \right] \right|. \label{eq: difference of statistics specified}
\end{align}
The quantity in \eqref{eq: steins method parametric specified} can be bounded using Theorem 2.4 of \cite{gaunt_rate_2023} by taking in their result the even function $g(\boldsymbol{W}) = \boldsymbol{W}' \boldsymbol{W}$, where $\boldsymbol{W} = (W_1, \ldots, W_d)'$ and $W_j = \frac{1}{\sqrt{n}} \sum_{l=1}^n X_{lj}$ for $j=1,\ldots,d$, with $X_{lj}$ as defined in the statement of Theorem \ref{thm: steins method specified}.
For the constants used in \cite{gaunt_rate_2023}, it can be derived that in our case $A = 2^{3}$, $B_1 = \cdots = B_d = 2^{6}$ and ${r_1, \ldots,}r_d = 6$.
The bound obtained is 
\begin{equation}\label{eq: difference of statistics specified bound}
    \left|\E \left[ h(Q^{\ast(n); \mub}) \right] -\E \left[ h(Q_{f_0}^{(n); \mub}) \right] \right| \leq C_1/n
\end{equation}
with $C_1$ as in the statement of Theorem \ref{thm: steins method specified}. 

The expression in \eqref{eq: difference of statistics specified} requires more work. 
For ease of presentation, for the rest of the proof we denote $Q^{\ast(n); \mub}$ by $Q^{\ast}$ and $Q_{f_0}^{(n); \mub}$ by $Q$.
Firstly, from a first-order Taylor expansion, we have that $h(Q^{\ast}) = h(Q) + \left( Q^{\ast} - Q \right) h^{(1)} (\tilde{Q})$, where $\tilde{Q}$ is between $Q^{\ast}$ and $Q$. Therefore, \eqref{eq: difference of statistics specified} can be written as 
\begin{equation*}
    \left|\E \left[ h(Q^{\ast}) \right] -\E \left[ h(Q) \right] \right|
    = \left|\E \left[ \left( Q^{\ast} - Q \right) h^{(1)} (\tilde{Q}) \right] \right|
    \leq \lVert h^{(1)} \rVert \E \left| Q^{\ast} - Q \right|.
\end{equation*}
So, our aim is to find an upper bound for $\E \left| Q^{\ast} - Q \right|$. 
For $\Gamma_{f_0; \lambda}$ and $\hat{\Gamma}_{\lambda}$ as defined in \eqref{eq: def: central sequence and FI for known median} and \eqref{eq: Fisher Information specified estimate}, respectively, denote by
\begin{equation*}
    \hat{\Gamma}_{\lambda}^{-1}  = 
    \begin{pmatrix}
        \hat{\gamma}_{11} & \hdots & \hat{\gamma}_{1d} \\
        \vdots & \ddots & \vdots \\
        \hat{\gamma}_{d1} & \hdots & \hat{\gamma}_{dd}
    \end{pmatrix}
    \quad \text{and} \quad
    \Gamma_{f_0; \lambda}^{-1}  = 
    \begin{pmatrix}
        {\gamma}_{11} & \hdots & {\gamma}_{1d} \\
        \vdots & \ddots & \vdots \\
        {\gamma}_{d1} & \hdots & {\gamma}_{dd}
    \end{pmatrix}
\end{equation*}
as in the statement of the theorem.
Then
\begin{equation*}
    Q^{\ast} = \frac{1}{n} \sum_{i=1}^d \sum_{j=1}^d \hat{\gamma}_{ij} \sum_{k=1}^n \sum_{\ell=1}^n \sin (\theta_{ik} - \mu_i)\sin (\theta_{j\ell} - \mu_j)
\end{equation*}
and, similarly,
\begin{equation*}
    Q = \frac{1}{n} \sum_{i=1}^d \sum_{j=1}^d \gamma_{ij} \sum_{k=1}^n \sum_{\ell=1}^n \sin (\theta_{ik} - \mu_i)\sin (\theta_{j\ell} - \mu_j).
\end{equation*}
Using the triangle inequality and the Cauchy-Schwarz inequality leads to
\begin{align}
    & \E \left| Q^{\ast} - Q \right| 
\leq \frac{1}{n} \sum_{i=1}^d \sum_{j=1}^d \E \left| \left(\hat{\gamma}_{ij} - \gamma_{ij}\right)\sum_{k=1}^n \sum_{\ell=1}^n \sin (\theta_{ik} - \mu_i)\sin (\theta_{j\ell} - \mu_j) \right| \nonumber\\
    & \leq \frac{1}{n} \sum_{i=1}^d \sum_{j=1}^d \left(\E \left[ \left(\hat{\gamma}_{ij} - \gamma_{ij}\right)^2 \right]  \E \left[ \left( \sum_{k=1}^n \sum_{\ell=1}^n \sin (\theta_{ik} - \mu_i)\sin (\theta_{j\ell} - \mu_j) \right) ^2 \right] \right)^{1/2} \label{eq: stein d cauchy schwarz}.
\end{align}
A long but straightforward calculation yields
\begin{align}
\E & \left[ \left( \sum_{k=1}^n \sum_{\ell=1}^n \sin (\theta_{ik} - \mu_i)\sin (\theta_{j\ell} - \mu_j) \right) ^2 \right] \nonumber\\
= & \ n \ \E \left[ \sin^2\left( \theta_{i1} - \mu_i \right) \sin^2\left( \theta_{j1} - \mu_j \right) \right] + 2n(n-1) \Bigl( \E \left[ \sin\left( \theta_{i1} - \mu_i \right) \sin\left( \theta_{j2} - \mu_j \right) \right] \Bigr)^2 \nonumber\\
& \hspace{1cm} + n(n-1) \E \left[ \sin^2\left( \theta_{i1} - \mu_i \right) \right] \E \left[ \sin^2\left( \theta_{j2} - \mu_j \right) \right] \label{eq: steins lemma second term}.
\end{align}
The statement of the theorem follows by combining \eqref{eq: difference of statistics specified bound}, \eqref{eq: stein d cauchy schwarz} and \eqref{eq: steins lemma second term}.
\end{proof}

\subsection{Proof of Lemma \ref{lemma: rate of diff of gammas}} \label{appendix: proof lemma steins specified}

\begin{proof}
For $i,j\in\{1,\ldots,d\}$ it holds that $\hat{\gamma}_{ij} = \hat{\gamma}_{ji} = (-1)^{i+j}{\hat{d}_{ij}^{\lambda_d}} /{\hat{d}^{\lambda_d}}$, ${\gamma}_{ij} = {\gamma}_{ji} = (-1)^{i+j}{d_{ij}^{\lambda_d}}/{d^{\lambda_d}}$ for
\begin{align}
    \hat{d}^{\lambda_d} = \det\left( 
    \hat{\Gamma}_{\lambda}
    \right), \qquad & \qquad \hat{d}_{ij}^{\lambda_d} = \det \left( 
    \left( \hat{\Gamma}_{\lambda} \right)_{(ij)}
    \right), \label{eq: det est removing}\\
    {d}^{\lambda_d} = \det\left( 
    \Gamma_{f_0; \lambda}
    \right), \qquad & \qquad {d}_{ij}^{\lambda_d} = \det \left( 
    \left( \Gamma_{f_0; \lambda} \right)_{(ij)}
    \right), \label{eq: det removing}
\end{align}
where $A_{(ij)}$ denotes the matrix $A$ after removing the $i^{th}$ row and $j^{th}$ column.
Notice that we can write
\begin{equation*}
    \frac{\hat{d}^{\lambda_d}_{ij}}{\hat{d}^{\lambda_d}} 
    = \frac{\hat{d}^{\lambda_d}_{ij} - d^{\lambda_d}_{ij}}{\hat{d}^{\lambda_d}} + \frac{d^{\lambda_d}_{ij}}{\hat{d}^{\lambda_d}} 
\end{equation*}
and applying the Taylor expansion 
to $1/{\hat{d}^{\lambda_d}}$ about ${d^{\lambda_d}}$ along with straightforward manipulations gives us
\begin{align} \label{eq: steins_method_taylor_expansion_determinant}
    \frac{\hat{d}^{\lambda_d}_{ij}}{\hat{d}^{\lambda_d}} 
    = \frac{d^{\lambda_d}_{ij}}{d^{\lambda_d}} 
    + \frac{\hat{d}^{\lambda_d}_{ij} - d^{\lambda_d}_{ij}}{d^{\lambda_d}} 
    - \frac{d^{\lambda_d}_{ij} \left( \hat{d}^{\lambda_d} - d^{\lambda_d} \right)}{\left(d^{\lambda_d}\right)^2} 
    - \frac{\left( \hat{d}^{\lambda_d}_{ij} - d^{\lambda_d}_{ij} \right) \left( \hat{d}^{\lambda_d} - d^{\lambda_d} \right)}{\left(d^{\lambda_d} \right)^2}
    + o\left( \left| \hat{d}^{\lambda_d} - d^{\lambda_d} \right| \right).
\end{align}
Therefore, using also that $(\sum_{i=1}^{4}\alpha_i)^2 \leq 4\sum_{i=1}^{4}\alpha_i^2$, for any $\alpha_i \in \mathbb{R}$, yields
\begin{align}\label{mid_step_equation}
\nonumber \E \left[ \left( \hat{\gamma}_{ij} - \gamma_{ij} \right)^2 \right] 
    = & \ \E \left[\left( \frac{\hat{d}^{\lambda_d}_{ij}}{\hat{d}^{\lambda_d}} - \frac{d^{\lambda_d}_{ij}}{d^{\lambda_d}} \right)^2 \right] \\
    \leq & \ 4 \Biggl( \frac{1}{\left(d^{\lambda_d}\right)^2} \E \left[\left(\hat{d}^{\lambda_d}_{ij} - d^{\lambda_d}_{ij} \right)^2\right] 
    + \frac{\left(d^{\lambda_d}_{ij} \right) ^2}{\left(d^{\lambda_d}\right)^4} \E \left[ \left( \hat{d}^{\lambda_d} - d^{\lambda_d} \right)^2 \right] \\
    & + \frac{1}{\left(d^{\lambda_d}\right)^4} \E \left[ \left( \hat{d}^{\lambda_d}_{ij} - d^{\lambda_d}_{ij} \right)^2 \left( \hat{d}^{\lambda_d} - d^{\lambda_d} \right)^2 \right] + o \left( \E \left[ \left( \hat{d}^{\lambda_d} - d^{\lambda_d} \right)^2 \right] \right) \Biggr). \nonumber
\end{align}
We will focus on proving that $\E \left[\left(\hat{d}^{\lambda_d}_{ij} - d^{\lambda_d}_{ij} \right)^2\right] = O \left(\frac{1}{n} \right)$, as the result follows from there. Indeed, apart from showing the order of the first term of the bound in \eqref{mid_step_equation}, such a result yields two more outcomes. Firstly, it also implies that the second term is $O\left(\frac{1}{n}\right)$ because $\hat{d}^{\lambda_d}$ and $d^{\lambda_d}$ can be expressed as $\hat{d}^{\lambda_{d+1}}_{d+1,d+1}$ and $ d^{\lambda_{d+1}}_{d+1,d+1}$, respectively. Secondly, showing that $\E \left[\left(\hat{d}^{\lambda_d}_{ij} - d^{\lambda_d}_{ij} \right)^2\right] = O \left(\frac{1}{n} \right)$ means that the last two terms of the bound in \eqref{mid_step_equation} can be upper-bounded by an $O\left( \frac{1}{n} \right)$ term, which then proves
\eqref{eq: rate of diff of gammas}.

\noindent The focus is put on the order of the determinant, and not on its exact expression, because the last term of \eqref{mid_step_equation} already prevents us from calculating an explicit constant for the bound of \eqref{eq: rate of diff of gammas}.
The determinant of a $d\times d$ matrix is given by $d!$ summands, each being a product of $d$ elements of the matrix.
For $\alpha, \beta \in \{1,\ldots,d\}$
we can write \eqref{eq: det removing} as 
\begin{equation}
\label{eq:d_ab}
    d^{\lambda_d}_{\alpha \beta} = \sum_{k = 1}^{(d-1)!} \prod_{i=1}^{d-1} (-1)^{k} J_{ik}
\end{equation}
where $J_{ik} = I^{f_0}_{\lambda_{h_k(i)} \lambda_{\ell_k(i)}}$, with $h_k(i)$ and $\ell_k(i)$ being (possibly different) linear functions of $i$ that take values in $\{1,\ldots,d\}$. 
These functions are of the form $\{c\pm i\}\mod d + 1$ for $c\in \{1,\ldots,d\}$.
Similarly, \eqref{eq: det est removing} can be written as 
\begin{align*}
    \hat{d}^{\lambda_d}_{\alpha \beta} 
    = & \sum_{k=1}^{(d-1)!} \prod_{i = 1}^{d-1} \frac{1}{n} \sum_{j=1}^n (-1)^{k} \hat{J}_{ik;j}
    = \ \frac{1}{n^{d-1}} \sum_{k = 1}^{(d-1)!}
    \sum\limits_{ \substack{ \scriptscriptstyle{j_\eta=1} \\ \scriptscriptstyle{\eta \in B} }}^n \prod_{i \in B} (-1)^{k} \hat{J}_{ik;j_\eta}
\end{align*}
where $\hat{J}_{ik;j} = \sin(\theta_{h_k(i) j} - \mu_{h_k(i)}) \sin(\theta_{\ell_k(i) j} - \mu_{\ell_k(i)})$ and $B = \{1,\ldots,d-1\}$.
It holds that $\E[\hat{J}_{ik;j}] = J_{ik}$ for $j=1,\ldots,n$.
These general expressions for the determinant allow us to calculate the necessary orders without using the explicit expressions of $h_k(i)$ and $\ell_k(i)$.
Note that the sum over $\eta$ and $j_\eta$ is a product of $(d-1)$ sums of $n$ summands each.
Define $B[-i]$ to be the set $B$ without the $i^\text{th}$ element, then,
\begin{align}
    & \E \left[\left(\hat{d}^{\lambda_d}_{\alpha\beta} - d^{\lambda_d}_{\alpha\beta} \right)^2\right] \nonumber\\
    & = \ \E \left[ \left( \frac{1}{n^{d-1}} \sum_{k = 1}^{(d-1)!} \sum\limits_{ \substack{ \scriptscriptstyle{j_\eta=1} \\ \scriptscriptstyle{\eta \in B}}}^n \left( \prod_{i\in B} (-1)^k \hat{J}_{ik;j_\eta} 
    - \frac{d^{\lambda_d}_{\alpha\beta}}{(d-1)!} \right) \right)^2 \right] \nonumber\\
    & = \ \frac{1}{n^{2d-2}} \sum_{k_1,k_2 = 1}^{(d-1)!} 
    \sum\limits_{ \substack{ \scriptscriptstyle{j_{\eta_1},l_{\eta_2}=1} \\ \scriptscriptstyle{\eta_1,\eta_2 \in B}}}^n \E \Biggl[ \left( \prod_{i\in B} (-1)^{k_1} \hat{J}_{ik_1;j_{\eta_1}} 
    - \frac{d^{\lambda_d}_{\alpha\beta}}{(d-1)!} \right) \nonumber\\
    & \hspace{6cm} \times \left( \prod_{i\in B} (-1)^{k_2} \hat{J}_{ik_2;l_{\eta_2}} 
    - \frac{d^{\lambda_d}_{\alpha\beta}}{(d-1)!} \right) \Biggr] \nonumber\\
    & = \ \frac{1}{n^{2d-2}} \sum_{k_1,k_2 = 1}^{(d-1)!} 
    \sum\limits_{ \substack{\scriptscriptstyle{j_{\eta_1},l_{\eta_2}=1} \\ 
    \scriptscriptstyle{\eta_1, \eta_2 \in B} \\ 
    \scriptscriptstyle{l_1 \neq l_{2}}}}^n
    \E \Biggl[ \left( \prod_{i\in B} (-1)^{k_1} \hat{J}_{ik_1;j_{\eta_1}} 
    - \frac{d^{\lambda_d}_{\alpha\beta}}{(d-1)!} \right) \nonumber\\
    & \hspace{6cm} \times
    \left( \prod_{i\in B} (-1)^{k_2} \hat{J}_{ik_2;l_{\eta_2}} 
    - \frac{d^{\lambda_d}_{\alpha\beta}}{(d-1)!} \right) \Biggr] \nonumber \\
    & \hspace{0.2cm} + \frac{1}{n^{2d-2}} \sum_{k_1,k_2 = 1}^{(d-1)!}
    \sum\limits_{\substack{ \scriptscriptstyle{j_{\eta_1},l_{\eta_2}=1} \\ \scriptscriptstyle{\eta_1 \in B} \\ \scriptscriptstyle{\eta_2 \in B[-1]} \\ \scriptscriptstyle{l_{1} = l_{2}}}}^n
    \E \Biggl[ \left( \prod_{i\in B} (-1)^{k_1} \hat{J}_{ik_1;j_{\eta_1}} 
    - \frac{d^{\lambda_d}_{\alpha\beta}}{(d-1)!} \right) \nonumber\\
    & \hspace{6cm} \times
    \left( \prod_{i\in B} (-1)^{k_2} \hat{J}_{ik_2;l_{\eta_2}} 
    - \frac{d^{\lambda_d}_{\alpha\beta}}{(d-1)!} \right) \Biggr] \label{eq: expectation_order_1n_term} 
\end{align}
The second term in \eqref{eq: expectation_order_1n_term} is the product of $2d-3$ sums from $1$ up to $n$. 
So, this term is $O\left(\frac{1}{n}\right)$. 
Similarly, for all other terms that involve indices that are equal to each other, we see that they can contain at most $2d-3$ sums from $1$ up to $n$ and so are at most $O(\frac{1}{n})$. 
It remains to consider the case where all indices are not equal to each other. To this end, defining 
\begin{equation} \label{eq: set_c}
    C := \left\{(j_{\eta_1}, l_{\eta_2}): \eta_1, \eta_2 \in B, 
    \hspace{-1cm}\begin{array}{c}
       j_{1} \in \{1,\ldots,n\}, \\
       j_{2} \in \{1,\ldots,n\} \setminus \left\{j_{1}\right\}, \\
       \vdots \\
       j_{d-1} \in \{1,\ldots,n\} \setminus \left\{j_{1}, \ldots, j_{d-2}\right\}  \\
        l_{1} \in \{1,\ldots,n\} \setminus \left\{j_{1}, \ldots, j_{d-1} \right\}, \\ 
        \vdots \\
        l_{d-1} \in \{1,\ldots,n\} \setminus \left\{j_{1}, \ldots, j_{d-1}, l_{1}, \ldots, l_{d-2}\right\} 
    \end{array}
    \right\},
\end{equation}
we have that
\begin{align} \label{eq: bounding multivariate det}
    & \E \left[\left(\hat{d}^{\lambda_d}_{\alpha\beta} - d^{\lambda_d}_{\alpha\beta} \right)^2\right] \nonumber \\
    & = \ \frac{1}{n^{2d-2}} \sum_{k_1,k_2 = 1}^{(d-1)!} \sum_{j_{\eta_1}, l_{\eta_2} \in C} \ \E \Biggl[ \left( \prod_{i\in B} (-1)^{k_1} \hat{J}_{ik_1;j_{\eta_1}}
    - \frac{d^{\lambda_d}_{\alpha\beta}}{(d-1)!} \right) \nonumber \\ 
    & \hspace{1cm} \times \left( \prod_{i\in B} (-1)^{k_2} \hat{J}_{ik_2;l_{\eta_2}} 
    - \frac{d^{\lambda_d}_{\alpha\beta}}{(d-1)!} \right) \Biggr] 
    + O\left( \frac{1}{n} \right) \nonumber \\
    & = \ \frac{1}{n^{2d-2}} \sum_{k_1,k_2 = 1}^{(d-1)!} \sum_{j_{\eta_1}, l_{\eta_2} \in C} \left( \prod_{i\in B} (-1)^{k_1} \E \left[\hat{J}_{ik_1;j_{\eta_1}}\right] 
    - \frac{d^{\lambda_d}_{\alpha\beta}}{(d-1)!} \right) \nonumber \\ 
    & \hspace{4.5cm} \times \left( \prod_{i\in B} (-1)^{k_2} \E \left[ \hat{J}_{ik_2;l_{\eta_2}} \right] 
    - \frac{d^{\lambda_d}_{\alpha\beta}}{(d-1)!} \right) + O\left( \frac{1}{n} \right) \nonumber \\
    & = \ \frac{n(n-1) \ldots (n-2d+3)}{n^{2d-2}} \left( \sum_{k = 1}^{(d-1)!} \prod_{i\in B} (-1)^k J_{ik} 
    - d^{\lambda_d}_{\alpha\beta} \right)^2 + O\left( \frac{1}{n} \right) \nonumber \\
    & = O\left( \frac{1}{n} \right),
\end{align}
where the last equality is a result of the definition of $d^{\lambda_d}_{\alpha\beta}$ as in \eqref{eq:d_ab}. Thus \eqref{eq: rate of diff of gammas} holds.
\end{proof}

\subsection{Proof of Proposition \ref{prop: estimate quantities}} \label{appendix: proof estimate quantities}
\begin{proof}
As in Proposition \ref{prop: ULAN}, we set ${\mub}^{(n)} = \mub + n^{-1/2}\taub^{(n)}_\mu$ for a bounded sequence $\taub^{(n)}_\mu$ such that $\mub^{(n)}$ remains in $[-\pi, \pi)^d$. 
We first prove that \eqref{eq: prop estimate quantities eq1} holds.
For estimator $\hat{\mub}^{(n)}$ that satisfies Assumption \ref{assumption: mu_estimate}, using Lemma 4.4 from \cite{kreiss_adaptive_1987}, it suffices to show that 
\begin{equation*}
    \E_{g_0}\left[ \cos(\theta_j - \mu_j) \right] - \frac{1}{n}\sum_{i=1}^n \cos(\theta_{ji} - \mu^{(n)}_j) = o_P(1)
\end{equation*}
as $n\rightarrow\infty$ under $P^{(n)}_{\mub, \boldsymbol{0}; g_0}$.
By the law of large numbers, 
\begin{equation*}
    \E_{g_0}\left[ \cos(\theta_j - \mu_j) \right] - \frac{1}{n}\sum_{i=1}^n \cos(\theta_{ji} - \mu_j) = o_P(1)
\end{equation*}
as $n\rightarrow\infty$ under $P^{(n)}_{\mub, \boldsymbol{0}; g_0}$, so it remains to show that
\begin{equation*}
    S_n = \frac{1}{n}\sum_{i=1}^n \left( \cos(\theta_{ji} - \mu^{(n)}_j) -\cos(\theta_{ji} - \mu_j) \right) = o_P(1)
\end{equation*}
as $n\rightarrow\infty$ under $P^{(n)}_{\mub, \boldsymbol{0}; g_0}$.
By the definition of $\mu^{(n)}_j$, $\lim_{n\rightarrow\infty}\mu^{(n)}_j=\mu_j$ and so $\lim_{n\rightarrow\infty}\cos(\theta_{ji} - \mu^{(n)}_j)=\cos(\theta_{ji} - \mu_j)$.
Since the $\theta_{ji}$ are iid, using the triangle inequality
\begin{align*}
    \E_{g_0} \left[ \left| S_n \right| \right] 
    & \leq \frac{1}{n}\sum_{i=1}^n \E_{g_0} \left[ \left| \cos(\theta_{ji} - \mu^{(n)}_j) -\cos(\theta_{ji} - \mu_j) \right| \right] \\
    & = \E_{g_0} \left[ \left| \cos(\theta_{j} - \mu^{(n)}_j) -\cos(\theta_{ji} - \mu_j) \right| \right] 
\end{align*}
Now, since $\left| \cos(\theta_{j} - \mu^{(n)}_j) -\cos(\theta_{ji} - \mu_j) \right| \leq 2$, we can use dominated convergence to conclude that 
\begin{align*}
    \lim_n \E_{g_0} \left[ \left| S_n \right| \right] 
    \leq \E_{g_0} \left[ \lim_n \left| \cos(\theta_{j} - \mu^{(n)}_j) -\cos(\theta_{ji} - \mu_j) \right| \right] = 0 \\
    \Rightarrow \lim_{n\rightarrow\infty} \E_{g_0} \left[ \left| S_n \right| \right] = 0.
\end{align*}
Thus, we have convergence in the first moment, which implies convergence in probability.\\

\noindent Now, we prove \eqref{eq: prop estimate quantities eq2}. The proof follows along the same lines as for \eqref{eq: prop estimate quantities eq1} and relies on showing that 
\begin{equation*}
    \E_{g_0}\left[ \frac{\partial}{\partial\theta_k}\phi_j^{f_0} (\thetab - \mub)  \right] - \frac{1}{n}\sum_{i=1}^n \frac{\partial}{\partial\theta_k}\phi_j^{f_0}(\thetab_{i} - \mub^{(n)}) = o_P(1)
\end{equation*}
as $n\rightarrow\infty$ under $P^{(n)}_{\mub, \boldsymbol{0}; g_0}$. By the law of large numbers and the triangular inequality, this is obtained if

\begin{align*}
    \E \left[ \left| \frac{\partial}{\partial\theta_k}\left( \frac{\frac{\partial}{\partial\theta_j} f_0(\thetab - \mub)}{f_0(\thetab - \mub)}\right) -\frac{\partial}{\partial\theta_k}\left( \frac{\frac{\partial}{\partial\theta_j} f_0(\thetab - \mub^{(n)})}{f_0(\thetab - \mub^{(n)})}\right) \right| \right] 
\end{align*}
converges to zero under $P^{(n)}_{\mub, \boldsymbol{0}; g_0}$. 
By Assumption \ref{assumption_on_f0_unspecified_median}, $\sup_n \frac{\partial}{\partial\theta_1}\phi_1^{f_0}(\thetab - \mub^{(n)})$ is bounded on the bounded interval $[-\pi, \pi)$. The result follows by applying dominated convergence. 
\end{proof}

\subsection{Proof of Proposition \ref{prop: unspecified_median_esimates}}\label{appendix: proof unspecified_median_esimates}
\begin{proof}
\noindent  For this proof, we need to define the following notation
    \begin{equation*}
        \Delta_{\lambda; i} \left( \mub \right) = 
        \begin{pmatrix}
            \sin(\theta_{1i} - \mu_1) \\ 
            \vdots \\
            \sin(\theta_{di} - \mu_d)
        \end{pmatrix} \quad \text{ and } \quad 
        \Delta_{f_0; \mu; i} \left( \mub \right)= 
        \begin{pmatrix}
            \phi_1^{f_0}(\thetab_{i} - \mub) \\ 
            \vdots \\
            \phi_d^{f_0}(\thetab_{i} - \mub) 
        \end{pmatrix}.
    \end{equation*} Let us now prove the two statements.
    
    (i) The result follows if we show that 
        \begin{equation}\label{eq: efficient_central_seq_res1}
        \Tilde{\Delta}_{f_0; \lambda}^{(n)\ast} (\hat{\mub}^{(n)}) - \Tilde{\Delta}_{f_0;g_0; \lambda}^{(n)} (\hat{\mub}^{(n)}) = o_P(1) \one_{d\times1}
        \end{equation}
        and 
        \begin{equation}\label{eq: efficient_central_seq_res2}
        \Tilde{\Delta}_{f_0;g_0; \lambda}^{(n)} (\hat{\mub}^{(n)}) - \Tilde{\Delta}_{f_0;g_0; \lambda}^{(n)} (\mub) = o_P(1) \one_{d\times1}
        \end{equation}
        under $P^{(n)}_{\mub, \boldsymbol{0}; g_0}$ as $n\rightarrow\infty$.
        Using Assumption \ref{assumption: mu_estimate} and Lemma 4.4 from \cite{kreiss_adaptive_1987}, \eqref{eq: efficient_central_seq_res2} follows from 
        \begin{equation}\label{eq: efficient_central_seq_res3}
            \Tilde{\Delta}_{f_0;g_0; \lambda}^{(n)} (\mub^{(n)}) - \Tilde{\Delta}_{f_0;g_0; \lambda}^{(n)} (\mub) = o_P(1) \one_{d\times1}
        \end{equation}
        under $P^{(n)}_{\mub, \boldsymbol{0}; g_0}$ as $n\rightarrow\infty$ where, as defined in Proposition \ref{prop: ULAN}, $\mub^{(n)} = \mub + n^{-1/2} \taub^{(n)}_\mu$. We will prove \eqref{eq: efficient_central_seq_res3} in what follows. From asymptotic linearity and the ULAN property, we know that
        \begin{equation}\label{eq: asymptotic_linearity_under_g0_lambda}    
            \Delta_{\lambda}^{(n)} (\mub^{(n)} ) - \Delta_{\lambda}^{(n)} (\mub )=-C^{g_0}_{\mu; \lambda}\taub_\mu + o_P(1)\one_{d\times1}
        \end{equation}
        under $P^{(n)}_{\mub, \boldsymbol{0}; g_0}$ as $n \rightarrow \infty$.
        Applying Taylor series expansion, the Law of Large Numbers and Slutsky's Lemma, we find that
        for some $\mub^\ast$ between $\mub^{(n)}$ and $\mub+n^{-1/2}\taub^{(n)}_\mu$,
        \begin{align} \label{eq: asymptotic_linearity_under_g0_mu}
           \Delta_{f_0; \mu}^{(n)} (\mub^{(n)})-\Delta_{f_0; \mu}^{(n)} (\mub)&= \frac{1}{\sqrt{n}}\sum_{i=1}^n\left(\Delta_{f_0;\mu;i}(\mub+n^{-1/2}\taub^{(n)}_\mu)-\Delta_{f_0;\mu;i}(\mub)\right) \nonumber \\
           &=\frac{1}{\sqrt{n}}\sum_{i=1}^n \left( \nabla_{\mub} \Delta_{f_0;\mu;i}( \mub) n^{-1/2}\taub^{(n)}_\mu + \left. \nabla^2_{\mub} \Delta_{f_0;\mu;i} (\mub) \right|_{\mub = \mub^\ast}  \left( n^{-1/2}\taub^{(n)}_\mu \right)^2 \right) \nonumber \\
           & = \E_{g_0}[\nabla_{\mub} \Delta_{f_0;\mu;1}(\mub)] \taub_\mu + o_P(1) \one_{d\times1}
           = - C^{f_0; g_0}_{\mu; \mu}\taub_\mu + o_P(1) \one_{d\times1}
        \end{align}
        under $P^{(n)}_{\mub, \boldsymbol{0}; g_0}$ as $n \rightarrow \infty$. 
        Using \eqref{eq: asymptotic_linearity_under_g0_lambda} and \eqref{eq: asymptotic_linearity_under_g0_mu}, we  have
        \begin{align*}
            \Tilde{\Delta}_{f_0;g_0; \lambda}^{(n)} ({\mub}^{(n)}) 
            = \ & \Delta_{\lambda}^{(n)} ({\mub}^{(n)} ) - 
            C^{g_0}_{\mu; \lambda}
            \left( C^{f_0; g_0}_{\mu; \mu} \right)^{-1}
            \Delta_{f_0; \mu}^{(n)} ({\mub}^{(n)}) \\
            = \ & \Delta_{\lambda}^{(n)}(\mub) - 
            C^{g_0}_{\mu; \lambda}
            \taub_{\mu}  - 
            C^{g_0}_{\mu; \lambda}
            \left( C^{f_0; g_0}_{\mu; \mu} \right)^{-1} 
            \left[ \Delta_{f_0; \mu}^{(n)}(\mub) - 
            C^{f_0; g_0}_{\mu; \mu}
                \taub_{\mu} \right]+ o_P(1) \one_{d\times1}\\
            = \ & \Delta_{\lambda}^{(n)}(\mub) - 
            C^{g_0}_{\mu; \lambda}
            \left( C^{f_0; g_0}_{\mu; \mu} \right)^{-1}  \Delta_{f_0; \mu}^{(n)}(\mub) + o_P(1) \one_{d\times1}\\
            = \ & \Tilde{\Delta}_{f_0;g_0; \lambda}^{(n)} (\mub) + o_P(1) \one_{d\times1}
        \end{align*}
    under $P^{(n)}_{\mub, \boldsymbol{0}; g_0}$ as $n\rightarrow\infty$. 
    It remains to prove \eqref{eq: efficient_central_seq_res1}.
    It holds that
    \begin{align*}
        \Tilde{\Delta}_{f_0; \lambda}^{(n)\ast} & (\hat{\mub}^{(n)}) - \Tilde{\Delta}_{f_0;g_0; \lambda}^{(n)} (\hat{\mub}^{(n)} ) 
        = 
        - \left[\widehat{C}_{\mu; \lambda}
        \left( \widehat{C}^{f_0}_{\mu; \mu} \right)^{-1} -
        C^{g_0}_{\mu; \lambda}
            \left( C^{f_0; g_0}_{\mu; \mu} \right)^{-1}   \right] \Delta_{f_0; \mu}^{(n)} (\hat{\mub}^{(n)} ).
    \end{align*}
    By Assumption \ref{assumption: mu_estimate}, $\sqrt{n}\left(\hat{\mub}^{(n)} - \mub \right) = O_P(1) \one_{d\times1}$.
    Using the CLT, $\Delta_{f_0; \mu}^{(n)}(\mub) \xrightarrow{\mathcal{D}}N(\boldsymbol{0}, \Gamma_{g_0; \lambda})$ so $\forall \epsilon>0 \exists M>0: \Prob \left( \left| \Delta_{f_0; \mu}^{(n)}(\mub) \right|>M \right) < \epsilon$ as $n\rightarrow\infty$, i.e. $\Delta_{f_0; \mu}^{(n)}(\mub) = O_P(1) \one_{d\times1}$.
    Using Assumption \ref{assumption: mu_estimate}, Lemma 4.4 from \cite{kreiss_adaptive_1987} and \eqref{eq: asymptotic_linearity_under_g0_mu}, it holds that
    \begin{equation*}
        \Delta_{f_0; \mu}^{(n)} \left( \hat{\mub}^{(n)} \right) = \Delta_{f_0; \mu}^{(n)}(\mub) - 
            C^{f_0; g_0}_{\mu; \mu}
               \sqrt{n} \left( \hat{\mub}^{(n)} - \mub \right) + o_P(1) \one_{d \times 1}
    \end{equation*}
    so $\Delta_{f_0; \mu}^{(n)} (\hat{\mub}^{(n)} ) = O_P(1) \one_{d\times1}$ under $P^{(n)}_{\mub, \boldsymbol{0}; g_0}$ as $n\rightarrow\infty$. 
    Therefore, if we show that
    \begin{equation}\label{eq: convergence of matrices}
        \widehat{C}_{\mu; \lambda}
        \left( \widehat{C}^{f_0}_{\mu; \mu} \right)^{-1} -
        C^{g_0}_{\mu; \lambda}
            \left( C^{f_0; g_0}_{\mu; \mu} \right)^{-1} = o_P(1) \one_{d\times d}
    \end{equation}
    under $P^{(n)}_{\mub, \boldsymbol{0}; g_0}$ as $n\rightarrow\infty$, then the result follows.
    In order to show this, we write out the expressions of the matrices.
    For the following calculations, we use the fact that the elements of the inverse of a $d\times d$ matrix are given by the cofactors which are the determinants of  $(d-1)\times (d-1)$ matrices and those consist of $(d-1)!$ summands, each one being a product of $d-1$ elements of the original matrix.
    There exist (possibly different) functions $\ell_j(\cdot), h_j(\cdot) : \{1,\ldots,d-1\} \rightarrow \{2,\ldots,d\}$ for $j=1,\ldots,(d-1)!$ such the entry at position $[1,1]$ of the resulting matrix in \eqref{eq: convergence of matrices} is equal to
    \begin{align*}
        \sum_{j=1}^{(d-1)!} \Biggl\{  & \frac{\hat{I}_{\mu_1\lambda_1}}{\hat{d}^{f_0}_\mu} \prod_{i=1}^{d-1} (-1)^{j+1} \hat{I}^{f_0}_{\mu_{h_j(i)}\mu_{\ell_j(i)}}
        - \frac{I^{g_0}_{\mu_1\lambda_1}}{d^{f_0; g_0}_\mu} \prod_{i=1}^{d-1} (-1)^{j+1} I^{f_0; g_0}_{\mu_{h_j(i)}\mu_{\ell_j(i)}} \Biggr\} \\
        = \ & \sum_{j=1}^{(d-1)!} (-1)^{j+1} \Biggl\{ \frac{\hat{I}_{\mu_1\lambda_1}}{\hat{d}^{f_0}_\mu} \prod_{i=1}^{d-1} \hat{I}^{f_0}_{\mu_{h_j(i)}\mu_{\ell_j(i)}} 
        - \frac{\hat{I}_{\mu_1\lambda_1}}{d^{f_0; g_0}_\mu} \prod_{i=1}^{d-1} I^{f_0; g_0}_{\mu_{h_j(i)}\mu_{\ell_j(i)}}  \\
        & \hspace{3cm} 
        + \frac{\hat{I}_{\mu_1\lambda_1}}{d^{f_0; g_0}_\mu} \prod_{i=1}^{d-1} I^{f_0; g_0}_{\mu_{h_j(i)}\mu_{\ell_j(i)}} 
        - \frac{I^{g_0}_{\mu_1\lambda_1}}{d^{f_0; g_0}_\mu} \prod_{i=1}^{d-1} I^{f_0; g_0}_{\mu_{h_j(i)}\mu_{\ell_j(i)}} \Biggr\} \\
        = \ & \sum_{j=1}^{(d-1)!} (-1)^{j+1} \Biggl\{ \hat{I}_{\mu_1\lambda_1} \left( \frac{1}{\hat{d}^{f_0}_\mu} \prod_{i=1}^{d-1} \hat{I}^{f_0}_{\mu_{h_j(i)}\mu_{\ell_j(i)}}
        - \frac{1}{d^{f_0; g_0}_\mu} \prod_{i=1}^{d-1} I^{f_0; g_0}_{\mu_{h_j(i)}\mu_{\ell_j(i)}} \right)  \\
        & \hspace{3cm} 
        + \frac{1}{d^{f_0; g_0}_\mu} \prod_{i=1}^{d-1} I^{f_0; g_0}_{\mu_{h_j(i)}\mu_{\ell_j(i)}}
        \left( \hat{I}_{\mu_1\lambda_1} - I^{g_0}_{\mu_1\lambda_1} \right)
        \Biggr\} \\
        = \ & \sum_{j=1}^{(d-1)!} (-1)^{j+1} \Biggl\{ \frac{\hat{I}_{\mu_1\lambda_1}}{\hat{d}^{f_0}_\mu} \left( \prod_{i=1}^{d-1} \hat{I}^{f_0}_{\mu_{h_j(i)}\mu_{\ell_j(i)}}
        - \prod_{i=1}^{d-1} I^{f_0; g_0}_{\mu_{h_j(i)}\mu_{\ell_j(i)}} \right) \\
        &
        - \frac{\hat{I}_{\mu_1\lambda_1}}{\hat{d}^{f_0}_\mu d^{f_0; g_0}_\mu} \prod_{i=1}^{d-1} I^{f_0; g_0}_{\mu_{h_j(i)}\mu_{\ell_j(i)}} \left( \hat{d}^{f_0}_\mu - d^{f_0; g_0}_\mu \right) + \frac{1}{d^{f_0; g_0}_\mu} \prod_{i=1}^{d-1} I^{f_0; g_0}_{\mu_{h_j(i)}\mu_{\ell_j(i)}}
        \left( \hat{I}_{\mu_1\lambda_1} - I^{g_0}_{\mu_1\lambda_1} \right)
        \Biggr\}
    \end{align*}
    where $d^{f_0; g_0}_\mu = \det \left( C^{f_0; g_0}_{\mu; \mu} \right)$ and $\hat{d}^{f_0}_\mu = \det \left( \widehat{C}^{f_0}_{\mu; \mu} \right)$.
    Under $P^{(n)}_{\mub, \boldsymbol{0}; g_0}$ as $n\rightarrow\infty$, the quantity above is $o_P(1)$ which can be proved using Proposition \ref{prop: estimate quantities} for the first and last term, manipulations in the spirit of \eqref{eq: expectation_order_1n_term} combined with Lemma 4.4 of \cite{kreiss_adaptive_1987} for the difference of determinants in the second term, Slutsky's theorem and the fact that all quantities are bounded.
    The other entries of the matrix can be dealt with similarly.\\

    \noindent (ii) 
    Showing that 
    \begin{equation}\label{eq: var_estimated_mu}
        \hat{V}_{f_0} ( \hat{\mub}^{(n)} )  - V_{f_0} (\hat{\mub}^{(n)} ) 
        = o_P(1) \one_{d\times d}
    \end{equation} and
    \begin{equation} \label{eq: var_true_mu}
    V_{f_0} ( \hat{\mub}^{(n)} )  - V^{f_0}_{g_0} (\mub ) 
    = o_P(1) \one_{d\times d},
    \end{equation}
    under $P^{(n)}_{\mub, \boldsymbol{0}; g_0}$ as $n\rightarrow\infty$ proves the result, where
    \begin{equation*}
            V_{f_0}\left(\mub \right)
    = \frac{1}{n} \sum_{i=1}^n \Tilde{\Delta}_{f_0;g_0;\lambda;i}^{(n)} (\mub)\Tilde{\Delta}_{f_0;g_0;\lambda;i}^{(n)} (\mub)',
    \end{equation*} 
    for 
    \begin{equation}\label{eq: def_efficient_central_seq1}
        \Tilde{\Delta }_{f_0;g_0;\lambda; i}^{(n)} (\mub) = \Delta_{\lambda; i}(\mub) - C^{g_0}_{\mu; \lambda} \left( C^{f_0; g_0}_{\mu; \mu} \right) ^{-1} \Delta_{f_0; \mu; i}(\mub).
    \end{equation}
    Considering \eqref{eq: var_true_mu}, it can be proved along the same lines as for Proposition \ref{prop: estimate quantities} that
    \begin{align*}
        & \frac{1}{n} \sum_{i=1}^n \Delta_{\lambda; i} ( \hat{\mub}^{(n)} ) \Delta_{\lambda; i} ( \hat{\mub}^{(n)} )' 
        = \E_{g_0} \Bigl[ \Delta_{\lambda} \left( \mub \right) \Delta_{\lambda} \left( \mub \right)'  \Bigr] + o_P(1) \one_{d\times1} \\
        & \frac{1}{n} \sum_{i=1}^n \Delta_{f_0; \mu;i} ( \hat{\mub}^{(n)} ) \Delta_{f_0; \mu;i} ( \hat{\mub}^{(n)})' 
        = \E_{g_0} \Bigl[ \Delta_{f_0; \mu} ( \mub ) \Delta_{f_0; \mu} ( \mub)' \Bigr] + o_P(1) \one_{d\times1},
        \end{align*}
        and
        \begin{align*}
        & \frac{1}{n} \sum_{i=1}^n \Delta_{\lambda; i} ( \hat{\mub}^{(n)} ) \Delta_{f_0; \mu; i} ( \hat{\mub}^{(n)})' = \E_{g_0} \Bigl[ \Delta_{\lambda} ( \mub ) \Delta_{f_0; \mu} ( \mub )' \Bigr] + o_P(1) \one_{d\times1}
    \end{align*}
    under $P^{(n)}_{\mub, \boldsymbol{0}; g_0}$ as $n\rightarrow\infty$. This leads to
    \begin{align*}
        V_{f_0} ( \hat{\mub}^{(n)} )
        & = \frac{1}{n} \sum_{i=1}^n \left( \Delta_{\lambda; i}(\hat{\mub}^{(n)}) - C^{g_0}_{\mu; \lambda} \left( C^{f_0; g_0}_{\mu; \mu} \right) ^{-1} \Delta_{f_0; \mu; i}(\hat{\mub}^{(n)}) \right) \\
        & \hspace{3cm} \times \left( \Delta_{\lambda; i}(\hat{\mub}^{(n)}) - C^{g_0}_{\mu; \lambda} \left( C^{f_0; g_0}_{\mu; \mu} \right) ^{-1} \Delta_{f_0; \mu; i}(\hat{\mub}^{(n)}) \right)'\\
     & = \E_{g_0} \Bigl[ \tilde{\Delta}_{f_0; g_0; \lambda} ( \mub ) \tilde{\Delta}_{f_0; g_0; \lambda} ( \mub )' \Bigr] + o_P(1) \one_{d\times d}
    \end{align*}
    under $P^{(n)}_{\mub, \boldsymbol{0}; g_0}$ as $n\rightarrow\infty$. 
    We can conclude that \eqref{eq: var_true_mu} holds. \\
    
    \noindent Now, considering \eqref{eq: var_estimated_mu} and denoting $C := C^{g_0}_{\mu; \lambda} \left( C^{f_0; g_0}_{\mu; \mu} \right)^{-1} - \widehat{C}_{\mu; \lambda} \left( \widehat{C}^{f_0}_{\mu; \mu} \right)^{-1}$,
    \begin{align*}
    \hat{V}_{f_0} ( \hat{\mub}^{(n)} ) - & V_{f_0} (\hat{\mub}^{(n)} ) \\
        =  \frac{1}{n} \sum_{i=1}^n \Biggl\{ & \Delta_{ \lambda; i} ( \hat{\mub}^{(n)} ) \Delta_{f_0; \mu; i} (\hat{\mub}^{(n)})' C' + C \Delta_{f_0; \mu; i} (\hat{\mub}^{(n)}) \Delta_{\lambda; i} (\hat{\mub}^{(n)})' \\
        & - \widehat{C}_{\mu; \lambda} \left( \widehat{C}^{f_0}_{\mu; \mu} \right)^{-1} \Delta_{f_0; \mu; i} \Delta_{f_0; \mu; i}' C' - C \Delta_{f_0; \mu; i} \Delta_{f_0; \mu; i}' \left( C^{f_0; g_0}_{\mu; \mu} \right)^{-1} C^{g_0}_{\mu; \lambda} 
        \Biggr\}.
    \end{align*}
    By \eqref{eq: convergence of matrices}, $C = o_P(1) \one_{d\times d}$ and since all other involved quantities are $O_P(1)$ under $P^{(n)}_{\mub, \boldsymbol{0}; g_0}$ as $n\rightarrow\infty$ (following similar arguments as done in this and earlier proofs), we get that
    \eqref{eq: var_estimated_mu} is true. Hence the announced result of the Proposition follows. 
\end{proof}

\subsection{Proof of Theorem \ref{thm: optimality unspecified median}}\label{appendix: proof optimality unspecified median}

\begin{proof}
    \textbf{(i)} By the CLT, 
\begin{equation*}
   \Tilde{\Delta}^{(n)}_{f_0;g_0; \lambda} (\mub) \xrightarrow{\mathcal{D}} \mathcal{N}_d \left( \boldsymbol{0}, V_{g_0}^{f_0}(\mub) \right)
\end{equation*}
for $V_{g_0}^{f_0}(\mub)$ as defined in \eqref{eq: variance} under $P^{(n)}_{\mub, \boldsymbol{0}; g_0}$ as $n\rightarrow\infty$.
So, for $Z \sim \mathcal{N}_d\left( \boldsymbol{0}, I_d \right)$, $I_d$ the $d\times d$ identity matrix, using Proposition \ref{prop: unspecified_median_esimates} and Slutsky's lemma,
\begin{align*}
    Q^{\ast (n)}_{f_0} (\hat{\mub})
    & = \left(\Tilde{\Delta}_{f_0; \lambda}^{(n)\ast} (\hat{\mub}^{(n)}) \right)' \left(\hat{V}_{f_0}(\hat{\mub}^{(n)}) \right)^{-1} \Tilde{\Delta}_{f_0; \lambda}^{(n)\ast} (\hat{\mub}^{(n)}) \\
    & = \left( \Tilde{\Delta}_{f_0;g_0; \lambda}^{(n)} (\mub) \right)' \left(V_{g_0}^{f_0}(\mub) \right)^{-1} \Tilde{\Delta}_{f_0;g_0; \lambda}^{(n)} (\mub) + o_P(1) \\
    & = \left( \Tilde{\Delta}_{f_0;g_0; \lambda}^{(n)} (\mub) \right)' \left(V_{g_0}^{f_0} (\mub) \right)^{-1/2} \left(V_{g_0}^{f_0} (\mub) \right)^{-1/2} \Tilde{\Delta}_{f_0;g_0; \lambda}^{(n)} (\mub) + o_P(1) \\
    & = Z' Z + o_P(1) \xrightarrow{\mathcal{D}} \chi^2_d
\end{align*}
under $P^{(n)}_{\mub, \boldsymbol{0}; g_0}$ as $n\rightarrow\infty$.
Since this holds under any $g_0$, the result holds under $\mathcal{H}^{(n)}_0$.

\textbf{(ii)} Using Proposition \ref{prop: unspecified_median_esimates}(i) and the CLT, we readily have
\begin{equation*}
    \Tilde{\Delta}^{(n) \ast}_{f_0; \lambda} (\hat{\mub}^{(n)}) = \Tilde{\Delta}^{(n)}_{f_0;g_0; \lambda} (\mub) + o_P(1)
    \xrightarrow{\mathcal{D}} \mathcal{N}_d \Bigl( 
    \boldsymbol{0}, V_{g_0}^{f_0}(\mub)
    \Bigr)
\end{equation*}
under $P^{(n)}_{\mub, \boldsymbol{0}; g_0}$ as $n\rightarrow\infty$.
From the ULAN property we have
\begin{align*}
    \Lambda :& = \log \frac{dP^{(n)}_{\mub + n^{-1/2} \boldsymbol{\tau}^{(n)}_\mu, n^{-1/2} \boldsymbol{\tau}^{(n)}_\lambda; g_0}}{dP^{(n)}_{\mub,\boldsymbol{0};g_0}}  
    = \boldsymbol{\tau}^{(n)}{'}
    \Delta_{g_0}^{(n)}(\mub)
    - \frac{1}{2}
    \boldsymbol{\tau}^{(n)}{'}
    \Gamma_{g_0}
    \boldsymbol{\tau}^{(n)} 
    + o_P(1)
\end{align*}
under $P^{(n)}_{\mub, \boldsymbol{0}; g_0}$ as $n\rightarrow\infty$, where $\Gamma_{g_0}$ is defined in \eqref{eq: def: Fisher Information}.
Under $P^{(n)}_{\mub,\boldsymbol{0};g_0}$, it holds that

\begin{equation*}
    \Lambda \xrightarrow{\mathcal{D}} \mathcal{N} \Bigl(- \frac{1}{2}
    \boldsymbol{\tau}'
    \Gamma_{g_0}
    \boldsymbol{\tau} , \quad  
    \boldsymbol{\tau}'
    \Gamma_{g_0}
    \boldsymbol{\tau} \Bigr)
\end{equation*}
 as $n\rightarrow\infty$ thanks to the CLT, with $\boldsymbol{\tau}=\lim_{n\rightarrow\infty}\boldsymbol{\tau}^{(n)}$.
Using again Proposition \ref{prop: unspecified_median_esimates} and the fact that $\Tilde{\Delta}^{(n)}_{f_0;g_0; \lambda} (\mub)$ was constructed such that $\cov(\Tilde{\Delta}^{(n)}_{f_0;g_0; \lambda} (\mub), \Delta_{g_0; \mu}^{(n)}(\mub))$ is zero, we can evaluate 
$$  \text{Cov}\Bigl(\Tilde{\Delta}^{(n) \ast}_{f_0; \lambda} (\hat{\mub}^{(n)}), \Lambda \Bigr)
      \xrightarrow{n \rightarrow\infty} 
    C_{g_0}^{f_0}(\mub)
    \boldsymbol{\tau}_\lambda
$$
under $P^{(n)}_{\mub, \boldsymbol{0}; g_0}$, with $C_{g_0}^{f_0}(\mub) = \text{Cov} ( \Tilde{\Delta}^{(n)}_{f_0;g_0; \lambda} (\mub), \Delta_{\lambda}^{(n)}(\mub))$.
Thus, again by CLT the joint distribution of $\Tilde{\Delta}^{(n) \ast}_{f_0; \lambda} (\hat{\mub}^{(n)})$ and $\Lambda$ is given by
\begin{equation*}
    \begin{pmatrix}
        \Tilde{\Delta}^{(n) \ast}_{f_0; \lambda} (\hat{\mub}^{(n)}) \\ \Lambda
    \end{pmatrix}
        \xrightarrow{\mathcal{D}} \mathcal{N}_{d+1} \Biggl( 
        \begin{pmatrix}
            \boldsymbol{0} \\ - \frac{1}{2}
    \boldsymbol{\tau}'
    \Gamma_{g_0}
    \boldsymbol{\tau}
        \end{pmatrix}, \quad 
    \begin{pmatrix}
        V_{g_0}^{f_0}(\mub) & C_{g_0}^{f_0}(\mub)
    \boldsymbol{\tau}_\lambda \\
    C_{g_0}^{f_0}(\mub)
    \boldsymbol{\tau}_\lambda & \boldsymbol{\tau}' \Gamma_{g_0}
    \boldsymbol{\tau}
    \end{pmatrix}
        \Biggr).
\end{equation*}
under $P^{(n)}_{\mub, \boldsymbol{0}; g_0}$ as $n\rightarrow\infty$.
Now, since $P^{(n)}_{\mub, 0; g_0}$ and $P^{(n)}_{\mub + n^{-1/2} \boldsymbol{\tau}^{(n)}_\mu, n^{-1/2} \boldsymbol{\tau}^{(n)}_\lambda; g_0}$ are mutually contiguous, using the \textit{Third Le Cam Lemma}, which can be found in \cite{verdebout_modern_2017} (Proposition 5.2.2) and holds thanks to the ULAN property, we find
\begin{equation*}
    \Tilde{\Delta}^{(n) \ast}_{f_0; \lambda} (\hat{\mub}^{(n)})
    \xrightarrow{\mathcal{D}} \mathcal{N}_d
    \Bigl( 
    C_{g_0}^{f_0}(\mub)
    \boldsymbol{\tau}_\lambda, 
    V_{g_0}^{f_0}(\mub)
    \Bigr)
\end{equation*}
under $P^{(n)}_{\mub + n^{-1/2} \boldsymbol{\tau}^{(n)}_\mu, n^{-1/2} \boldsymbol{\tau}^{(n)}_\lambda; g_0}$ as $n \rightarrow \infty$. Thanks to Proposition \ref{prop: unspecified_median_esimates}(ii) and contiguity, it thus holds that
\begin{equation*}
    W = \hat{V}_{f_0}(\hat{\mub}^{(n)})^{-1/2} \Tilde{\Delta}^{(n) \ast}_{f_0; \lambda} (\hat{\mub}^{(n)})
    \xrightarrow{\mathcal{D}} \mathcal{N}_d \Biggl(V_{g_0}^{f_0}(\mub)^{-1/2}C_{g_0}^{f_0}(\mub)
    \boldsymbol{\tau}_\lambda, 
    I_d \Biggr)
\end{equation*}
under $P^{(n)}_{\mub + n^{-1/2} \boldsymbol{\tau}^{(n)}_\mu, n^{-1/2} \boldsymbol{\tau}^{(n)}_\lambda; g_0}$ as $n \rightarrow \infty$, where $\hat{V}_{f_0}(\hat{\mub}^{(n)})$ is defined in \eqref{eq: variance estimate}. 
Thus, 
\begin{equation*}
    Q^{\ast (n)}_{f_0}
    = W' W 
    \xrightarrow{\mathcal{D}} 
    \chi^2_d(\kappa)
\end{equation*}
under $P^{(n)}_{\mub + n^{-1/2} \boldsymbol{\tau}^{(n)}_\mu, n^{-1/2} \boldsymbol{\tau}^{(n)}_\lambda; g_0}$ as $n \rightarrow \infty$, where 
$$\kappa = \biggl( V_{g_0}^{f_0}(\mub)^{-1/2}C_{g_0}^{f_0}(\mub)
    \boldsymbol{\tau}_\lambda \biggr)' V_{g_0}^{f_0}(\mub)^{-1/2}C_{g_0}^{f_0}(\mub)
    \boldsymbol{\tau}_\lambda = \boldsymbol{\tau}'_\lambda C_{g_0}^{f_0}(\mub) V_{g_0}^{f_0}(\mub)^{-1} C_{g_0}^{f_0}(\mub)
    \boldsymbol{\tau}_\lambda.$$

\textbf{(iii)} In (i) we proved that $Q_{f_0}^{\ast(n)} = Q^{(n)}_{f_0} + o_P(1)$ under $P^{(n)}_{\mub, \boldsymbol{0}; g_0}$ as $n\rightarrow\infty$ so the result follows for all $f_0 \in \mathcal{F}$ from the optimality of the $f_0$-parametric test $\phi^{(n)}_{f_0}$.
\end{proof}

\subsection{Proof of Theorem \ref{theorem: stein's method unspecified}} \label{sec: proof steins method unspecified}

\begin{proof}
All expectations in the proof are with respect to $g_0$. 
We omit this from the notation for simplicity. 
Using the triangle inequality
\begin{align}
    \left|\E \left[ h ( Q_{f_0}^{\ast (n)} ) \right] - \E \left[ h\left(\chi^2_d\right) \right] \right| 
    \ \leq \ & \left|\E \left[ h(Q_{f_0}^{\ast (n)} ) \right] -\E \left[ h(Q_{f_0}^{\ast (n)} (\mub)) \right] \right| \label{eq: difference of statistics unspecified estimated mu} \\
    & + \left| \E \left[ h(Q_{f_0}^{\ast (n)} (\mub)) \right] -\E \left[ h(Q_{f_0; g_0}^{(n)} (\mub)) \right] \right| \label{eq: difference of statistics unspecified} \\
    & + \left|\E \left[ h(Q_{f_0; g_0}^{(n)} (\mub)) \right] - \E \left[ h\left(\chi^2_d\right) \right] \right| \label{eq: steins method parametric unspecified}
\end{align}
where {$Q_{f_0}^{\ast (n)} (\mub)$ is defined in \eqref{eq: test_statistic_unspecified_evaluated} }
and
\begin{equation*}
    Q_{f_0; g_0}^{(n)} (\mub) = \Tilde{\Delta}_{f_0;g_0;\lambda}^{(n)} (\mub)' 
    V_{g_0}^{f_0} \left( \mub \right)  ^{-1}
    \Tilde{\Delta}_{f_0;g_0;\lambda}^{(n)} (\mub)
\end{equation*}
is the parametric version of $Q_{f_0}^{\ast (n)}$.
The easiest term to bound is
\eqref{eq: steins method parametric unspecified}, which can be bounded using Theorem 2.4 of \cite{gaunt_rate_2023} by taking $g(\boldsymbol{W}) = \boldsymbol{W}' \boldsymbol{W}$, where $\boldsymbol{W} = (W_1, \ldots, W_d)'$ and $W_j = \frac{1}{\sqrt{n}} \sum_{l=1}^n X_{lj}$ for $j=1,\ldots,d$, with $X_{lj}$ as defined in the statement of Theorem \ref{theorem: stein's method unspecified}.
The bound obtained is 
\begin{equation}\label{eq: stein_unspecified_easy_bound}
    \left|\E \left[ h(Q_{f_0; g_0}^{(n)} (\mub)) \right] - \E \left[ h\left(\chi^2_d\right) \right] \right| \leq C_3/n
\end{equation}
with $C_3$ as given in the statement of Theorem \ref{theorem: stein's method unspecified}.\\

Considering the expression in the right-hand side of \eqref{eq: difference of statistics unspecified estimated mu}, using a first-order Taylor expansion of $h(Q_{f_0}^{\ast (n)})$ about $Q_{f_0}^{\ast (n)} (\mub)$, we obtain that $h(Q_{f_0}^{\ast (n)}) = h(Q_{f_0}^{\ast (n)} (\mub)) + \left( Q_{f_0}^{\ast (n)} - Q_{f_0}^{\ast (n)} (\mub) \right) h^{(1)} (Q^\dag)$, where $Q^\dag$ is a random variable between $Q_{f_0}^{\ast (n)}$ and $Q_{f_0}^{\ast (n)} (\mub)$. Therefore
\begin{align*}
    \left|\E \left[ h(Q_{f_0}^{\ast (n)}) \right] -\E \left[ h(Q_{f_0}^{\ast (n)} (\mub)) \right] \right|
    & = \left|\E \left[ \left( Q_{f_0}^{\ast (n)} - Q_{f_0}^{\ast (n)} (\mub) \right) h^{(1)} (Q^\dag) \right] \right| \\
    & \leq \lVert h^{(1)} \rVert \E \left| Q_{f_0}^{\ast (n)} - Q_{f_0}^{\ast (n)} (\mub) \right|.
\end{align*}
So, our goal is to find an upper bound for $\E \left| Q_{f_0}^{\ast (n)} - Q_{f_0}^{\ast (n)} (\mub) \right|$.
Using a first order Taylor expansion, it holds that 
\begin{equation*}
    Q_{f_0}^{\ast (n)} = Q_{f_0}^{\ast (n)} (\mub) + \left({ \left. \nabla_{\mub} Q_{f_0}^{\ast (n)} (\mub) \right|_{\mub = \mub^\ast}} \right)' (\hat{\mub}^{(n)} - \mub)
\end{equation*}
for some $\mub^{\ast}$ between $\hat{\mub}^{(n)}$ and $\mub$.
By Assumption \ref{assumption_on_f0_unspecified_median}, $Q_{f_0}^{\ast (n)} (\mub)$ is continuously differentiable with respect to $\mub$, with a bounded derivative as it is a continuous function on a bounded domain.
Finally, using the Cauchy-Schwarz inequality,
we get
\begin{align}\label{eq: rate estimated mu statistic}
    \E \left| Q_{f_0}^{\ast (n)} - Q_{f_0}^{\ast (n)} (\mub) \right| 
    & \leq \E \left[ \left| \left({ \left. \nabla_{\mub} Q_{f_0}^{\ast (n)} (\mub) \right|_{\mub = \mub^\ast}} \right)'(\hat{\mub}^{(n)} - \mub) \right| \right] \\
    & \leq \sqrt{\E \left[ \Big\lVert \left( { \left. \nabla_{\mub} Q_{f_0}^{\ast (n)} (\mub) \right|_{\mub = \mub^\ast}} \right)' \Big\rVert_2^2 \right]} \sqrt{\E \left[  \Big \lVert \hat{\mub}^{(n)} - \mub \Big\rVert_2^2 \right]}.\nonumber
\end{align} 

The expression in \eqref{eq: difference of statistics unspecified} requires more work. 
For ease of presentation, for the rest of the proof we denote $Q_{f_0}^{\ast (n)} (\mub)$ by $Q^{\ast}$ and $Q_{f_0;g_0}^{(n)} (\mub)$ by $Q$.
Firstly, $\exists \Tilde{Q}$ between $Q^{\ast}$ and $Q$ such that $h(Q^{\ast}) = h(Q) + \left( Q^{\ast} - Q \right) h^{(1)} (\Tilde{Q})$ and so \eqref{eq: difference of statistics unspecified} can be written as 
\begin{equation*}
    \left|\E \left[ h(Q^{\ast}) \right] -\E \left[ h(Q) \right] \right|
    = \left| \E \left[ \left( Q^{\ast} - Q \right) h^{(1)} (\Tilde{Q}) \right] \right|
    \leq \lVert h^{(1)} \rVert \E \left| Q^{\ast} - Q \right|.
\end{equation*}
So, our aim is to find an upper bound for $\E \left| Q^{\ast} - Q \right|$.
    We use the notations 
    \begin{equation*}
    \left( V_{g_0}^{f_0} \left( \mub \right) \right)^{-1} = \begin{pmatrix}
        \gamma_{11} & \cdots & \gamma_{1d} \\
        \vdots & \ddots & \vdots \\
        \gamma_{d1} & \cdots & \gamma_{dd}
    \end{pmatrix}, \qquad \left(\hat{V}_{f_0}\left(\mub \right) \right)^{-1} = \begin{pmatrix}
        \hat{\gamma}_{11} & \cdots & \hat{\gamma}_{1d} \\
        \vdots & \ddots & \vdots \\
        \hat{\gamma}_{d1} & \cdots & \hat{\gamma}_{dd}
    \end{pmatrix}
\end{equation*}
and
\begin{equation*}
        \left(C^{f_0; g_0}_{\mu; \mu} \right)^{-1} = \begin{pmatrix}
        \beta_{11} & \cdots & \beta_{1d} \\
        \vdots & \ddots & \vdots \\
        \beta_{d1} & \cdots & \beta_{dd}
    \end{pmatrix}, \qquad
    \left(\tilde{C}^{f_0}_{\mu; \mu} \right)^{-1} = \begin{pmatrix}
        \Tilde{\beta}_{11} & \cdots & \Tilde{\beta}_{1d} \\
        \vdots & \ddots & \vdots \\
        \Tilde{\beta}_{d1} & \cdots & \Tilde{\beta}_{dd}
    \end{pmatrix},
    \end{equation*}
where $\tilde{C}^{f_0}_{\mu; \mu}$ is the same matrix as $\hat{C}^{f_0}_{\mu; \mu}$ using the true value, $\mub$, instead of the estimated one.
The matrices are symmetric, so it holds that for $i, j \in \{1,\ldots,d\}$, $\gamma_{ij} = \gamma_{ji}$, $\hat{\gamma}_{ij} = \hat{\gamma}_{ji}$, $\beta_{ij} = \beta_{ji}$ and $\hat{\beta}_{ij} = \hat{\beta}_{ji}$.
Using this notation, we can write
    \begin{equation*}
        Q = \frac{1}{n} \sum_{i=1}^d \sum_{j=1}^d \gamma_{ij} \sum_{k=1}^n \sum_{\ell=1}^n \Tilde{\Delta}_{ik} \Tilde{\Delta}_{j\ell}, 
        \qquad
        Q^{\ast} = \frac{1}{n} \sum_{i=1}^d \sum_{j=1}^d \hat{\gamma}_{ij} \sum_{k=1}^n \sum_{\ell=1}^n \Tilde{\Delta}_{ik}^{\ast} \Tilde{\Delta}_{j\ell}^{\ast}
    \end{equation*}
    where, for simplicity of notation, we define 
    \begin{equation}\label{eq: semi_parametric_central_seq_ij}
    \begin{split} 
      \Tilde{\Delta}_{ij} 
      & := \Tilde{\Delta}_{f_0; g_0;\lambda;ij} (\mub)  
      = \sin(\theta_{ij} - \mu_i) - I_{\mu_i\lambda_i} \sum_{k=1}^d \beta_{ik} \phi^{f_0}_{kj}(\thetab - \mub), \\
        \Tilde{\Delta}_{ij}^{\ast} & := \Tilde{\Delta}_{f_0;\lambda;ij}^{\ast} (\mub)
        = \sin(\theta_{ij} - \mu_i) - \Tilde{I}_{\mu_i\lambda_i} \sum_{k=1}^d \Tilde{\beta}_{ik} \phi^{f_0}_{kj}(\thetab - \mub),
    \end{split}
    \end{equation} 
    and
    $\Tilde{I}_{\mu_i \lambda_i} = \frac{1}{n} \sum_{j=1}^n \cos(\theta_{ij} - \mu_i)$.
    For the rest of the proof we denote $\phi^{f_0}_{kj}(\thetab - \mub)$ by $\phi^{f_0}_{kj}$.
    Now, using the triangle and the Cauchy-Schwarz inequalities, we get
    \begin{align}
        \E \left| Q^{\ast} - Q \right| & = \ \E \Biggl| \frac{1}{n} \sum_{i=1}^d \sum_{j=1}^d \hat{\gamma}_{ij} \sum_{k=1}^n \sum_{\ell=1}^n \Tilde{\Delta}_{ik}^{\ast} \Tilde{\Delta}_{j\ell}^{\ast} - \frac{1}{n} \sum_{i=1}^d \sum_{j=1}^d \gamma_{ij} \sum_{k=1}^n \sum_{\ell=1}^n \Tilde{\Delta}_{ik} \Tilde{\Delta}_{j\ell} \Biggr| \nonumber \\
        & \leq \ \frac{1}{n} \sum_{i=1}^d \sum_{j=1}^d \E \Biggl| \hat{\gamma}_{ij} \sum_{k=1}^n \sum_{\ell=1}^n \Tilde{\Delta}_{ik}^{\ast} \Tilde{\Delta}_{j\ell}^{\ast} - \gamma_{ij} \sum_{k=1}^n \sum_{\ell=1}^n \Tilde{\Delta}_{ik} \Tilde{\Delta}_{j\ell} \Biggr| \nonumber \\
        & \leq \ \frac{1}{n} \sum_{i=1}^d \sum_{j=1}^d \left\{ 
        \left( \E \left[ \left( \hat{\gamma}_{ij} \right)^2 \right] \right)^{1/2}
        \left( \E \left[ \left( \sum_{k=1}^n \sum_{\ell=1}^n \left( \Tilde{\Delta}_{ik}^{\ast} \Tilde{\Delta}_{j\ell}^{\ast} - \Tilde{\Delta}_{ik} \Tilde{\Delta}_{j\ell} \right) \right)^2 \right] \right)^{1/2} \right. \nonumber \\
        & \hspace{3cm} \left. + \ \left(\E \left[ \left( \hat{\gamma}_{ij} - \gamma_{ij} \right)^2 \right] \right)^{1/2} 
        \left( \E \left[ \left( \sum_{k=1}^n \sum_{\ell=1}^n \Tilde{\Delta}_{ik} \Tilde{\Delta}_{j\ell} \right)^2 \right] \right)^{1/2} \right\}. \label{eq: stein_unspecified_terms_to_bound}
    \end{align}

    The statement of the theorem follows by combining \eqref{eq: stein_unspecified_easy_bound}, \eqref{eq: rate estimated mu statistic}, \eqref{eq: stein_unspecified_terms_to_bound}.
    


\end{proof}

{
\subsection{Details for the MSE of the MLE being $\mathcal{O}(n^{-1})$} \label{appendix: proof msal mle}
The discussion below Assumption \ref{assumption: MSE} in the main paper shows that the relevant $O(n^{-1})$ mean-squared error condition follows from an asymptotically linear representation of $\sqrt{n}(\boldsymbol{\hat{\mu}}^{(n)} - \mub)$ with an $L^2$-negligible remainder. In the following proposition, we prove that such a requirement is satisfied by a sufficiently regular maximum likelihood estimator (MLE). We highlight that the proposition is stated for generic i.i.d. random observations $\boldsymbol{X_1},\ldots,\boldsymbol{X_n}$, which could of course possibly be random vectors. We emphasize that $\boldsymbol{X_i}$ may be identified with the angular observation $\thetab_i\in[-\pi,\pi)^d$, but the general argument of Proposition \ref{prop:MLE_L2_asymptotic_linearity} is not specific to data in the $d$-dimensional torus that have been the focus in our main paper. The proposition gives explicit regularity conditions under which the vector MLE admits the required linear expansion as in \eqref{eq:asymptotic_linearity} of the main paper, and hence satisfies Assumption~\ref{assumption: MSE}.
\begin{proposition}
\label{prop:MLE_L2_asymptotic_linearity}
Let $\boldsymbol{X_1},\ldots,\boldsymbol{X_n}$ be i.i.d. random observations, possibly random vectors, with probability density (or mass) function
$f(\boldsymbol{X_1};\mub)$, where $\mub\in\Theta\subset\mathbb R^d$, and let
$\nabla(\log (f(\boldsymbol{X_1};\mub)))\in \mathbb{R}^{d\times 1}$ being the gradient vector and $J(\boldsymbol{X_k},\mub) = \left\{ \frac{\partial^2}{\partial\mu_i\partial\mu_j}\log(f(\boldsymbol{X_k};\mub))\right\}_{i,j = 1,\ldots,d}\in \mathbb{R}^{d\times d}$ the Hessian matrix associated with the $j$\textsuperscript{th} observation. Assume that:
\begin{enumerate}
\item[(a)] the true value \(\mub\) is identifiable and an interior point of \(\Theta\);

\item[(b)] the probability density (or mass) function $f$ is twice
continuously differentiable in \(\mub\);

\item[(c)] differentiation under the integral sign is valid and $\E[\nabla(\log (f(\boldsymbol{X_1};\mub)))]=\zerob$, while the expected Fisher information matrix
$$I(\mub) = \E\left[(\nabla(\log (f(\boldsymbol{X_1};\mub))))(\nabla(\log (f(\boldsymbol{X_1};\mub))))^{\prime})\right] = - \E\left[J(\boldsymbol{X_1},\mub)\right]$$
is finite and non-singular;
 \item[(d)] for $q>1$, $\sup_{n\geq1}
 \E\left\|\frac{1}{\sqrt n}\sum_{i=1}^n
 \nabla(\log (f(\boldsymbol{X_i};\mub)))\right\|^{2q}<\infty$;

 \item[(e)] If $\bar I_n = -\frac{1}{n}
 \sum_{i=1}^n J(\boldsymbol{X_i},\tilde{\mub}_n)$,
 where $\tilde{\mub}_n$ is between $\hat{\mub}_{\mathrm{MLE}}^{(n)}$ and $\mub$, then for $p>1$
 $$\E\left\|\bar I_n^{-1}-I(\mub)^{-1}\right\|^{2p} \to 0,$$
 where $p$ and $q$ from (d) are conjugate exponents, in the sense that $1/p + 1/q = 1$.
 \end{enumerate}
 Then, for $\mathbf R_n=\left\{\bar I_n^{-1}-I(\mub)^{-1}\right\}\frac1{\sqrt n}\sum_{i=1}^n\nabla(\log (f(\boldsymbol{X_i};\mub)))$, we have the representation
\begin{equation*}
\sqrt{n}\left(\hat{\mub}_{\mathrm{MLE}}^{(n)}-\mub\right) = I(\mub)^{-1} \frac{1}{\sqrt n}\sum_{i=1}^n\nabla(\log (f(\boldsymbol{X_i};\mub))) + \mathbf R_n,
\end{equation*}
where $\E\|\mathbf R_n\|^2=o(1)$. Consequently, for each \(j\in\{1,\ldots,d\}\),
\begin{equation*}
\E\left[\left(\hat\mu_{j,\mathrm{MLE}}^{(n)} - \mu_j\right)^2\right] = O(n^{-1}).
\end{equation*}
\end{proposition}
\begin{proof}
The definition of the MLE and a Taylor expansion of the score function (evaluated at the MLE) around $\mub$ yields
\begin{equation*}
\zerob = \sum_{i=1}^n\nabla(\log (f(\boldsymbol{X_i};\hat{\mub}_{\mathrm{MLE}}^{(n)}))) = \sum_{i=1}^n\nabla(\log (f(\boldsymbol{X_i};\mub))) + \sum_{i=1}^n J(\boldsymbol{X_i},\tilde{\mub}_n) \left(\hat{\mub}_{\mathrm{MLE}}^{(n)}-\mub\right),
\end{equation*}
where \(\tilde{\mub}_n\) lies between \(\hat{\mub}_{\mathrm{MLE}}^{(n)}\) and \(\mub\). Hence, for $\bar I_n =-\frac{1}{n}\sum_{i=1}^n
J(\boldsymbol{X_i},\tilde{\mub}_n)$, we have that
\begin{equation*}
\begin{aligned}
\sqrt{n}\left(\hat{\mub}_{\mathrm{MLE}}^{(n)}-\mub\right)&=\bar I_n^{-1}\frac1{\sqrt n}\sum_{i=1}^n\nabla(\log (f(\boldsymbol{X_i};\mub)))\\
& = I(\mub)^{-1}\frac1{\sqrt n}\sum_{i=1}^n\nabla(\log (f(\boldsymbol{X_i};\mub)))+\mathbf R_n,
\end{aligned}
\end{equation*}
with $\mathbf R_n$ as in the statement of the proposition. Now, by H\"older's inequality, with \(p^{-1}+q^{-1}=1\), we obtain that
\begin{equation*}
\E\|\mathbf R_n\|^2\leq\left(\E\left\|\bar I_n^{-1}-I(\mub)^{-1}\right\|^{2p}\right)^{1/p}\left(\E\left\|\frac1{\sqrt n}\sum_{i=1}^n\nabla(\log (f(\boldsymbol{X_i};\mub)))\right\|^{2q}\right)^{1/q},
\end{equation*}
where the first factor tends to zero by Assumption (e), while the second factor is bounded by Assumption (d). Therefore, $\E\|\mathbf R_n\|^2=o(1)$ showing that the scaled estimation error for the MLE can be represented in an asymptotically linear form as in \eqref{eq:asymptotic_linearity} of the main paper and therefore Assumption \ref{assumption: MSE} of the main paper (that the MSE is of order $\mathcal{O}(n^{-1})$) is satisfied for sufficiently regular MLEs.
\end{proof}
}

\subsection{Proof of Lemma \ref{lemma: rate of diff of gammas unspecified}} \label{appendix: proof lemma rate of diff of gammas unspecified}
Before presenting the proof of Lemma \ref{lemma: rate of diff of gammas unspecified}, we provide a lemma that is used in the proof.

\begin{lemma}\label{lemma: used_for_unspecified_steins}
    Define $\left(C^{f_0; g_0}_{\mu; \mu} \right)^{-1} = \begin{pmatrix}
        \beta_{11} & \cdots & \beta_{1d} \\
        \vdots & \ddots & \vdots \\
        \beta_{1d} & \cdots & \beta_{dd}
    \end{pmatrix}$ and $\left(\tilde{C}^{f_0}_{\mu; \mu} \right)^{-1} = \begin{pmatrix}
        \Tilde{\beta}_{11} & \cdots & \Tilde{\beta}_{1d} \\
        \vdots & \ddots & \vdots \\
        \Tilde{\beta}_{1d} & \cdots & \Tilde{\beta}_{dd}
    \end{pmatrix}$, where $\tilde{C}^{f_0}_{\mu; \mu}$ is the same matrix as $\hat{C}^{f_0}_{\mu; \mu}$ (as defined in \eqref{eq: def: cov_estimated}) but using the true value, $\mub$, instead of the estimated one and for this lemma we denote $\Tilde{I}_{\mu_i\lambda_i} = \frac{1}{n} \sum_{j=1}^n \cos(\theta_{ij} - \mu_i)$. 
    For $i, j, l_1, l_2 \in \{1,\ldots,d\}$, the following hold, where all expectations are taken with respect to $g_0$
    \begin{enumerate}
        \item[(i)] For any $X_{ik_1}, Y_{jk_2}$ such that $\E\left[ X_{ik_1} \right] = 0$ and $\E\left[ Y_{jk_2} \right] = 0$ $\forall k_1, k_2 \in \{1,\ldots,n\}$, we have $$\E \left[ \left( \sum_{k_1 = 1}^n \sum_{k_2 = 1}^n X_{ik_1} Y_{jk_2} \right)^2 \right] = O(n^2).$$

        \item[(ii)] For any $X_{ik_1}, Y_{jk_2}$ such that $\E\left[ X_{ik_1} \right] = 0$ and $\E\left[ Y_{jk_2} \right] = 0$ $\forall k_1, k_2 \in \{1,\ldots,n\}$, we have $$\E \left[ \left( \sum_{k_1 = 1}^n \sum_{k_2 = 1}^n X_{ik_1} Y_{jk_2} \right)^4 \right] = O(n^4).$$

        \item[(iii)] $\E \left[ \left( \Tilde{I}_{\mu_i \lambda_i} \Tilde{\beta}_{ij} - I^{g_0}_{\mu_i \lambda_i} \beta_{ij} \right)^2 \right] = O\left( \frac{1}{n} \right)$.
        
        \item[(iv)] $\E \left[ \left( \Tilde{I}_{\mu_i \lambda_i} \Tilde{\beta}_{ij} - I^{g_0}_{\mu_i \lambda_i} \beta_{ij} \right)^4 \right] = O\left( \frac{1}{n^2} \right)$.

        \item[(v)] $\E \left[ \left( \Tilde{I}_{\mu_i \lambda_i} \Tilde{I}_{\mu_j \lambda_j} \Tilde{\beta}_{il_1} \Tilde{\beta}_{il_2} - I^{g_0}_{\mu_i \lambda_i} I^{g_0}_{\mu_j \lambda_j} \beta_{il_1} \beta_{il_2} \right)^2 \right] = O\left( \frac{1}{n} \right)$.

        \item[(vi)] $\E \left[ \left( \Tilde{I}_{\mu_i \lambda_i} \Tilde{I}_{\mu_j \lambda_j} \Tilde{\beta}_{il_1} \Tilde{\beta}_{il_2} - I^{g_0}_{\mu_i \lambda_i} I^{g_0}_{\mu_j \lambda_j} \beta_{il_1} \beta_{il_2} \right)^4 \right] = O\left( \frac{1}{n^2} \right)$.

         \item[(vii)] For $k\in\{1,\ldots,n\}$, $\E \left[ \Tilde{\Delta}_{f_0; \lambda;ik}^{\ast} (\mub) \Tilde{\Delta}_{f_0; \lambda;jk}^{\ast} (\mub) - \Tilde{\Delta}_{f_0;g_0;\lambda;i1} (\mub) \Tilde{\Delta}_{f_0;g_0;\lambda;j1} (\mub) \right] = O\left( \frac{1}{\sqrt{n}} \right)$, where $\Tilde{\Delta}_{f_0; \lambda;ik}^{\ast} (\mub)$ and $\Tilde{\Delta}_{f_0;g_0;\lambda;i1} (\mub)$ denote the $k^\text{th}$ elements of the vectors $\Tilde{\Delta}_{f_0; \lambda;k}^{\ast} (\mub)$ and $\Tilde{\Delta}_{f_0;g_0;\lambda;1} (\mub)$, respectively. 
    \end{enumerate}
\end{lemma}
\begin{proof}
    (i) It holds that
    \begin{align*}
        \E & \left[ \left( \sum_{k_1 = 1}^n \sum_{k_2 = 1}^n X_{ik_1} Y_{jk_2} \right)^2 \right] = \sum_{k_1 = 1}^n \sum_{k_2 = 1}^n \sum_{k_3 = 1}^n \sum_{k_4 = 1}^n \E \left[ X_{ik_1} X_{ik_2} Y_{jk_3} Y_{jk_4}  \right] \\
        = & \sum_{k_1 = 1}^n \sum_{k_2 = 1}^n \sum_{k_3 = 1}^n \E \left[ X_{ik_1} X_{ik_2} Y^2_{jk_3} \right]
        + \sum_{k_1 = 1}^n \sum_{k_2 = 1}^n \sum_{k_3 = 1}^n \sum_{k_4 \neq k_3} \E \left[ X_{ik_1} X_{ik_2} Y_{jk_3} Y_{jk_4}  \right] \\
         = & \sum_{k_1 = 1}^n \sum_{k_2 = 1}^n \E \left[ X_{ik_1} X_{ik_2} Y^2_{jk_2}  \right] 
         + \sum_{k_1 = 1}^n \sum_{k_2 = 1}^n \sum_{k_3 \neq k_2} \E \left[ X_{ik_1} X_{ik_2} Y^2_{jk_3}  \right] \\
         & + \sum_{k_1 = 1}^n \sum_{k_2 = 1}^n \sum_{k_3 \neq k_2} \E \left[ X_{ik_1} X_{ik_2} Y_{jk_3} Y_{jk_2}  \right]
         + \sum_{k_1 = 1}^n \sum_{k_2 = 1}^n \sum_{k_3 = 1}^n \sum_{k_4 \neq k_3, k_2} \E \left[ X_{ik_1} X_{ik_2} Y_{jk_3} Y_{jk_4}  \right] \\
         = & \sum_{k_1 = 1}^n \E \left[ X^2_{ik_1} Y^2_{jk_1} \right] 
         + \sum_{k_1 = 1}^n \sum_{k_2 \neq k_1} \E \left[ X_{ik_1} \right] \E \left[ X_{ik_2} Y^2_{jk_2} \right] \\
         & + \sum_{k_1 = 1}^n \sum_{k_2 \neq k_1}^n \E \left[ X_{ik_1} Y^2_{jk_1} \right] \E \left[ X_{ik_2} \right]
         + \sum_{k_1 = 1}^n \sum_{k_2 = 1}^n \sum_{k_3 \neq k_2, k_1} \E \left[ X_{ik_1} X_{ik_2} \right] \E \left[ Y^2_{jk_3}  \right] \\
         & + \sum_{k_1 = 1}^n \sum_{k_2 \neq k_1} \E \left[ X_{ik_1} Y_{jk_1} \right] \E \left[X_{ik_2} Y_{jk_2} \right]
         + \sum_{k_1 = 1}^n \sum_{k_2 = 1}^n \sum_{k_3 \neq k_2, k_1} \E \left[ X_{ik_1} X_{ik_2} Y_{jk_2} \right] \E \left[ Y_{jk_3} \right] \\
         & + \sum_{k_1 = 1}^n \sum_{k_2 \neq k_1} \sum_{k_3 \neq k_1} \E \left[ X_{ik_1} Y_{jk_1} \right] \E\left[ X_{ik_2} Y_{jk_3} \right]
         + \sum_{k_1 = 1}^n \sum_{k_2 = 1}^n \sum_{k_3 = 1}^n \sum_{k_4 \neq k_3, k_2, k_1} \E \left[ X_{ik_1} X_{ik_2} Y_{jk_3} \right] \E \left[ Y_{jk_4} \right] \\
         = & \ n \E \left[ X^2_{i1} Y^2_{j1} \right] 
         + \sum_{k_1 = 1}^n \sum_{k_3 \neq k_1} \E \left[ X^2_{ik_1} \right] \E \left[ Y^2_{jk_3}  \right] 
         + \sum_{k_1 = 1}^n \sum_{k_2 \neq k_1} \sum_{k_3 \neq k_2, k_1} \E \left[ X_{ik_1} \right] \E \left[ X_{ik_2} \right] \E \left[ Y^2_{jk_3}  \right] \\
         & + n(n-1) \left( \E \left[ X_{i1} Y_{j1} \right] \right)^2
         + \sum_{k_1 = 1}^n \sum_{k_2 \neq k_1} \E \left[ X_{ik_1} Y_{jk_1} \right] \E\left[ X_{ik_2} Y_{jk_2} \right] \\
         & + \sum_{k_1 = 1}^n \sum_{k_2 \neq k_1} \sum_{k_3 \neq k_2, k_1} \E \left[ X_{ik_1} Y_{jk_1} \right] \E\left[ X_{ik_2}  \right] \E\left[ Y_{jk_3} \right]\\
         = & \ n \E \left[ X^2_{i1} Y^2_{j1} \right] 
         + n(n-1) \E \left[ X^2_{i1} \right] \E \left[ Y^2_{j2}  \right] 
         + 2n(n-1) \left( \E \left[ X_{i1} Y_{j1} \right] \right)^2 \\
         = & \ O(n^2).
    \end{align*}
    
    (ii) The result can be proved in the same way as for (i).

    (iii) For $i,j\in\{1,\ldots,d\}$, we can write
    $\Tilde{\beta}_{ij} = (-1)^{i+j}{\Tilde{d}^\mu_{ij}} /{\Tilde{d}^\mu}$ and ${\beta}_{ij} = (-1)^{i+j}{d^\mu_{ij}}/d^\mu$ where
    $\Tilde{d}^\mu = \det\left( \tilde{C}^{f_0}_{\mu; \mu} \right), \Tilde{d}^\mu_{ij} = \det \left( \tilde{C}^{f_0}_{\mu; \mu; ij} \right), 
    d^\mu = \det\left( C^{f_0; g_0}_{\mu; \mu}  \right)$ and $d^\mu_{ij} = \det \left( C^{f_0; g_0}_{\mu; \mu; ij} \right)$. 
    Here, $\tilde{C}^{f_0}_{\mu; \mu; ij}$ and $C^{f_0; g_0}_{\mu; \mu; ij}$ are matrices $\tilde{C}^{f_0}_{\mu; \mu}$ and $C^{f_0; g_0}_{\mu; \mu}$ without row $i$ and column $j$, respectively.
    Using this notation, 
    \begin{equation*}
        \E \left[ \left( \Tilde{I}_{\mu_i \lambda_i} \Tilde{\beta}_{ij} - I^{g_0}_{\mu_i \lambda_i} \beta_{ij} \right)^2 \right] = 
        \E \left[ \left( \Tilde{I}_{\mu_i \lambda_i} \frac{\Tilde{d}^\mu_{ij}}{\Tilde{d}^\mu} - I^{g_0}_{\mu_i \lambda_i} \frac{d^\mu_{ij}}{d^\mu} \right)^2 \right].
    \end{equation*}
     Applying the  Taylor series expansion to $\frac{1}{\Tilde{d}^\mu}$ about $\frac{1}{d^{\mu}}$ along with straightforward manipulations yields
    \begin{align*}
        \frac{\Tilde{I}_{\mu_i \lambda_i}  \Tilde{d}^\mu_{ij}}{\Tilde{d}^\mu}
        = & \frac{I^{g_0}_{\mu_i \lambda_i} d^\mu_{ij}}{d^\mu}
        + \frac{\Tilde{I}_{\mu_i \lambda_i}  \Tilde{d}^\mu_{ij} - I^{g_0}_{\mu_i \lambda_i} d^\mu_{ij}}{d^\mu}
        - \frac{I^{g_0}_{\mu_i \lambda_i} d^\mu_{ij} \left( \Tilde{d}^\mu - d^\mu \right) }{(d^\mu)^2} \\
        & \hspace{0.5cm} - \frac{\left( \Tilde{I}_{\mu_i \lambda_i}  \Tilde{d}^\mu_{ij} - I^{g_0}_{\mu_i \lambda_i} d^\mu_{ij} \right) \left( \Tilde{d}^\mu - d^\mu \right) }{(d^\mu)^2}
        + o \left( | \Tilde{d}^\mu - d^\mu | \right),
    \end{align*}
    so, using the inequality $\left( \sum_{i=1}^n x_i \right)^2 \leq n\sum_{i=1}^n x_i^2$, we get
    \begin{align}
        \E & \left[ \left( \Tilde{I}_{\mu_i \lambda_i} \frac{\Tilde{d}^\mu_{ij}}{\Tilde{d}^\mu} - I^{g_0}_{\mu_i \lambda_i} \frac{d^\mu_{ij}}{d^\mu} \right)^2 \right] \nonumber\\
        & \leq \ 4 \Biggl[ \frac{1}{\left( d^\mu \right)^2} \E \left[ \left( \Tilde{I}_{\mu_i \lambda_i} \Tilde{d}^\mu_{ij} - I^{g_0}_{\mu_i \lambda_i} d^\mu_{ij} \right)^2 \right]
        + \frac{ \left( I^{g_0}_{\mu_i \lambda_i} d^\mu_{ij} \right)^2}{(d^\mu)^4} \E \left[ \left( \Tilde{d}^\mu - d^\mu \right)^2 \right] \nonumber\\
        & \hspace{2cm} + \frac{1}{(d^\mu)^4} \E \left[ \left( \Tilde{I}_{\mu_i \lambda_i} \Tilde{d}^\mu_{ij} - I^{g_0}_{\mu_i \lambda_i} d^\mu_{ij} \right)^2 \left( \Tilde{d}^\mu - d^\mu \right)^2 \right] 
        + o \left( \E \left[ \left( \Tilde{d}^\mu - d^\mu \right)^2 \right] \right) \Biggr]. \label{eq: mid_step_equation_lemma_unspecified}
    \end{align}
    We will focus on proving that $\E \left[ \left( \Tilde{I}_{\mu_i \lambda_i} \Tilde{d}^\mu_{ij} - I^{g_0}_{\mu_i \lambda_i} d^\mu_{ij} \right)^2 \right] = O \left(\frac{1}{n} \right)$. 
    The proof that the second term of the bound in \eqref{eq: mid_step_equation_lemma_unspecified} is $O\left(\frac{1}{n}\right)$ is similar and easier than the proof for the first term, and so it is skipped.
    The fact that those two terms are $O \left(\frac{1}{n} \right)$ implies that the last two terms  of the bound in \eqref{eq: mid_step_equation_lemma_unspecified} can be upper-bounded by an $O\left( \frac{1}{n} \right)$ term, which then proves the desired result.

     The exact expression of the determinant $d^\mu_{ij}$ is not important.
     For $\alpha, \beta \in \{1,\ldots,d\}$
    we can write the determinant as
    \begin{equation*}
        d^\mu_{\alpha \beta} = \sum_{k = 1}^{(d-1)!} \prod_{i = 1}^{d-1} (-1)^k J_{ik}
    \end{equation*}
    where $J_{ik} = I^{g_0}_{\mu_{h_k(i)} \mu_{\ell_k(i)}}$, with $h_k(i)$ and $\ell_k(i)$ being (possibly different) linear functions of $i$ that take values in $\{1,\ldots,d\}$. 
    These functions are of the form $h_k(i) = \{c\pm i\}(\mod d) + 1$ for $c\in \{1,\ldots,d\}$.
    Similarly,
    \begin{equation*}
        \Tilde{d}^\mu_{\alpha \beta} 
        = \sum_{k = 1}^{(d-1)!} \prod_{i = 1}^{d-1} \frac{1}{n} \sum_{j=1}^n (-1)^k \hat{J}_{ik; j}
        = \frac{1}{n^{d-1}} \sum_{k =1}^{(d-1)!} \sum\limits_{ \substack{ \scriptscriptstyle{j_\eta=1} \\ \scriptscriptstyle{\eta \in B}}}^n \prod_{i \in B} (-1)^k \hat{J}_{ik; j_\eta}
    \end{equation*}
    where $\hat{J}_{ik; j} = \frac{\partial}{\partial\theta_{\ell_k(i)}} \phi^{f_0}_{h_k(i)j}$ and $B=\{1,\ldots,d-1\}$.
    It immediately holds that $\E \left[ \hat{J}_{ik; j} \right] = {J}_{ik}$.
    These general expressions for the determinant allow us to calculate the necessary orders without using the explicit functions of $h_k(i)$ and $\ell_k(i)$. Note that the sum over $\eta$ and $j_\eta$ is a product of $(d-1)$ sums of $n$ summands each. Define $B[-i]$ to be the set $B$ without the $i^\text{th}$ element, then
    \begin{align}
         \E & \left[ \left( \Tilde{I}_{\mu_\alpha \lambda_\alpha} \Tilde{d}^\mu_{\alpha\beta} - I^{g_0}_{\mu_\alpha \lambda_\alpha} d^\mu_{\alpha\beta} \right)^2 \right] \nonumber\\
        = & \ \E \left[ \left( \frac{1}{n} \sum_{\ell=1}^n \cos(\theta_{\alpha\ell} - \mu_\alpha) \frac{1}{n^{d-1}} \sum_{k =1}^{(d-1)!} \sum\limits_{ \substack{ \scriptscriptstyle{j_\eta=1} \nonumber\\ \scriptscriptstyle{\eta \in B}}}^n \prod_{i \in B} (-1)^k \hat{J}_{ik; j_\eta} - I^{g_0}_{\mu_\alpha \lambda_\alpha} d^\mu_{\alpha \beta} \right)^2 \right] \nonumber\\
        = & \ \frac{1}{n^{2d}} \sum_{\ell_1,\ell_2=1}^n \sum_{k_1, k_2 = 1}^{(d-1)!} \sum\limits_{ \substack{ \scriptscriptstyle{j^1_{\eta_1}, j^2_{\eta_2}=1} \\ \scriptscriptstyle{\eta_1, \eta_2 \in B}}}^n \E \left[ \prod_{y=1}^2 \left( \cos(\theta_{\alpha\ell_y} - \mu_\alpha) \prod_{i \in B} (-1)^{k_y} \hat{J}_{ik_y; j^y_{\eta_y}} - \frac{I^{g_0}_{\mu_\alpha \lambda_\alpha} d^\mu_{\alpha \beta}}{(d-1)!} \right) \right] \nonumber\\
        = & \ \frac{1}{n^{2d}} \sum_{\ell_1,\ell_2=1}^n \sum_{k_1, k_2 = 1}^{(d-1)!} \sum\limits_{ \substack{ \scriptscriptstyle{j^1_{\eta_1}, j^2_{\eta_2}=1} \\ \scriptscriptstyle{\eta_1, \eta_2 \in B} \\ \scriptscriptstyle{j^2_{1} \neq j^2_{2}}}}^n \E \left[ \prod_{y=1}^2 \left( \cos(\theta_{\alpha\ell_y} - \mu_\alpha) \prod_{i \in B} (-1)^{k_y} \hat{J}_{ik_y; j^y_{\eta_y}} - \frac{I^{g_0}_{\mu_\alpha \lambda_\alpha} d^\mu_{\alpha \beta}}{(d-1)!} \right) \right] \nonumber\\
        & + \ \frac{1}{n^{2d}} \sum_{\ell_1,\ell_2=1}^n \sum_{k_1, k_2 = 1}^{(d-1)!} \sum\limits_{ \substack{ \scriptscriptstyle{j^1_{\eta_1}, j^2_{\eta_2}=1} \\ \scriptscriptstyle{\eta_1 \in B} \\ \scriptscriptstyle{\eta_2 \in B[-1]} \\ \scriptscriptstyle{j^2_{1} = j^2_{2}}}}^n \E \left[ \prod_{y=1}^2 \left( \cos(\theta_{\alpha\ell_y} - \mu_\alpha) \prod_{i \in B} (-1)^{k_y} \hat{J}_{ik_y; j^y_{\eta_y}} - \frac{I^{g_0}_{\mu_\alpha \lambda_\alpha} d^\mu_{\alpha \beta}}{(d-1)!} \right) \right] .\label{eq: lemma_unspecified_steins_expectation_order_1n_term}
    \end{align}
    The second term in \eqref{eq: lemma_unspecified_steins_expectation_order_1n_term} is the product of $2d-1$ sums from $1$ up to $n$. 
    So, this term is $O\left(\frac{1}{n}\right)$. 
    Concerning the first term, we define 
    \begin{equation*}
        C := \left\{(j^1_{\eta_1}, j^2_{\eta_2}, \ell_1, \ell_2): \\ \eta_1, \eta_2 \in B, 
        \hspace{-1cm}\begin{array}{c}
           j^1_{1} \in \{1,\ldots,n\}, \\
           j^1_{2} \in \{1,\ldots,n\} \setminus \left\{j^1_{1}\right\}, \\
           \vdots \\
           j^1_{d-1} \in \{1,\ldots,n\} \setminus \left\{j^1_{1}, \ldots, j^1_{d-2}\right\}  \\
            j^2_{1} \in \{1,\ldots,n\} \setminus \left\{j^1_{1}, \ldots, j^1_{d-1} \right\}, \\ 
            \vdots \\
            j^2_{d-1} \in \{1,\ldots,n\} \setminus \left\{j^1_{1}, \ldots, j^1_{d-1}, j^2_{1}, \ldots, j^2_{d-2}\right\} \\
            \ell_1 \in \{1,\ldots,n\} \setminus \left\{j^1_{1}, \ldots, j^1_{d-1}, j^2_{1}, \ldots, j^2_{d-1}\right\} \\
            \ell_2 \in \{1,\ldots,n\} \setminus \left\{j^1_{1}, \ldots, j^1_{d-1}, j^2_{1}, \ldots, j^2_{d-1}, \ell_1\right\} 
        \end{array}
        \right\}.
    \end{equation*}
    Building upon the independence between random quantities (ensured thanks to taking different indices in $C$), we obtain
    \begin{align*}
        \E & \left[ \left( \Tilde{I}_{\mu_\alpha \lambda_\alpha} \Tilde{d}^\mu_{\alpha\beta} - I^{g_0}_{\mu_\alpha \lambda_\alpha} d^\mu_{\alpha\beta} \right)^2 \right] \nonumber\\
        = & \ \frac{1}{n^{2d}} \sum_{k_1, k_2 = 1}^{(d-1)!} \sum\limits_{ \substack{ \scriptscriptstyle{j^1_{\eta_1}, j^2_{\eta_2}, \ell_1, \ell_2 \in C} \\ \scriptscriptstyle{\eta_1, \eta_2 \in B}}}^n \E \left[ \prod_{y=1}^2 \left( \cos(\theta_{\alpha\ell_y} - \mu_\alpha) \prod_{i \in B} (-1)^{k_y} \hat{J}_{ik_y; j^y_{\eta_y}} - \frac{I^{g_0}_{\mu_\alpha \lambda_\alpha} d^\mu_{\alpha \beta}}{(d-1)!} \right) \right] + \ O \left( \frac{1}{n} \right) \\
        = & \ \frac{1}{n^{2d}} \sum_{k_1, k_2 = 1}^{(d-1)!} \sum\limits_{ \substack{ \scriptscriptstyle{j^1_{\eta_1}, j^2_{\eta_2}, \ell_1, \ell_2 \in C} \\ \scriptscriptstyle{\eta_1, \eta_2 \in B}}}^n \prod_{y=1}^2 \left( \E \left[ \cos(\theta_{\alpha\ell_y} - \mu_\alpha) \right] \prod_{i \in B} (-1)^{k_y} \E \left[  \hat{J}_{ik_y; j^y_{\eta_y}} \right] - \frac{I^{g_0}_{\mu_\alpha \lambda_\alpha} d^\mu_{\alpha \beta}}{(d-1)!} \right) + \ O \left( \frac{1}{n} \right) \\
        = & \ \frac{n(n-1)\cdots(n-2d+1)}{n^{2d}} \left( \sum_{k = 1}^{(d-1)!} \left( \E \left[ \cos(\theta_{\alpha 1} - \mu_\alpha) \right] \prod_{i \in B} (-1)^{k} J_{ik} - I^{g_0}_{\mu_\alpha \lambda_\alpha} d^\mu_{\alpha \beta} \right) \right)^2 + \ O \left( \frac{1}{n} \right),
    \end{align*}
    where the last equality follows from \eqref{eq:Imulambda}.
    So, $\E \left[ \left( \Tilde{I}_{\mu_i \lambda_i} \Tilde{d}^\mu_{ij} - I^{g_0}_{\mu_i \lambda_i} d^\mu_{ij} \right)^2 \right] = O \left( \frac{1}{n} \right)$ and the result follows.

    (iv), (v), (vi) The results can be proved in a similar way to (iii).

    (vii) For notational simplicity, in this proof we denote $\phi^{f_0}_{\ell i}(\thetab - \mub)$ by $\phi^{f_0}_{\ell i}$.
    Before providing the proof, we remind the reader of the notation used:
    \begin{align*}
        \Tilde{\Delta}_{f_0; \lambda;ik}^{\ast} (\mub) 
        & = \sin(\theta_{ik} - \mu_i) - \Tilde{I}_{\mu_i\lambda_i} \sum_{\ell = 1}^d \Tilde{\beta}_{i\ell} \phi^{f_0}_{\ell k}, \quad
        \Tilde{\Delta}_{f_0; g_0; \lambda;ik} (\mub) \\
        & = \sin(\theta_{ik} - \mu_i) - I^{g_0}_{\mu_i\lambda_i} \sum_{\ell = 1}^d \beta_{i\ell} \phi^{f_0}_{\ell k}.
    \end{align*}
    Using the triangle and Cauchy-Schwarz inequalities, we obtain
    \begin{align}
        \E & \left[ \Tilde{\Delta}_{f_0; \lambda;ik}^{\ast} (\mub) \Tilde{\Delta}_{f_0; \lambda;jk}^{\ast} (\mub) - \Tilde{\Delta}_{f_0;g_0;\lambda;i1} (\mub) \Tilde{\Delta}_{f_0;g_0;\lambda;j1} (\mub) \right] \nonumber \\
        = & \ \E \left[ \Tilde{\Delta}_{f_0; \lambda;ik}^{\ast} (\mub) \Tilde{\Delta}_{f_0; \lambda;jk}^{\ast} (\mub) - \Tilde{\Delta}_{f_0;g_0;\lambda;ik} (\mub) \Tilde{\Delta}_{f_0;g_0;\lambda;jk} (\mub) \right] \nonumber \\
        = & \ \E \left[ - \sum_{\ell=1}^d \sin(\theta_{jk} - \mu_j) \phi^{f_0}_{\ell k} \left( \Tilde{I}_{\mu_i \lambda_i} \Tilde{\beta}_{i\ell} - I^{g_0}_{\mu_i\lambda_i} \beta_{i\ell} \right) \right. \nonumber \\
        & \hspace{0.5cm} - \sum_{\ell=1}^d \sin(\theta_{ik} - \mu_i) \phi^{f_0}_{\ell k} \left( \Tilde{I}_{\mu_j \lambda_j} \Tilde{\beta}_{j\ell} - I^{g_0}_{\mu_j\lambda_j} \beta_{j\ell} \right)  \nonumber \\
        & \hspace{0.5cm} \left. + \sum_{\ell_1=1}^d \sum_{\ell_2=1}^d \phi^{f_0}_{\ell_1 i} \phi^{f_0}_{\ell_2 i} \left( \Tilde{I}_{\mu_j\lambda_j} \Tilde{I}_{\mu_i\lambda_i} \Tilde{\beta}_{j\ell_1} \Tilde{\beta}_{i\ell_2} - I^{g_0}_{\mu_j\lambda_j} I^{g_0}_{\mu_i\lambda_i} \beta_{j\ell_1} \beta_{i\ell_2} \right) \right] \nonumber \\
        \leq & \ \sum_{\ell=1}^d \E \biggl| \sin(\theta_{jk} - \mu_j) \phi^{f_0}_{\ell k} \left( \Tilde{I}_{\mu_i \lambda_i} \Tilde{\beta}_{i\ell} - I^{g_0}_{\mu_i \lambda_i} \beta_{i\ell} \right) \biggr| \nonumber\\
        & \hspace{0.5cm} + \sum_{\ell=1}^d \E \biggl| \sin(\theta_{ik} - \mu_i) \phi^{f_0}_{\ell k} \left( \Tilde{I}_{\mu_j\lambda_j} \Tilde{\beta}_{j\ell} - I^{g_0}_{\mu_j\lambda_j} \beta_{j\ell} \right) \biggr| \nonumber \\
        & \hspace{0.5cm} + \sum_{\ell_1=1}^d \sum_{\ell_2=1}^d \E \biggl| \phi^{f_0}_{\ell_1 k} \phi^{f_0}_{\ell_2 k} \left( \Tilde{I}_{\mu_j\lambda_j} \Tilde{I}_{\mu_i\lambda_i} \Tilde{\beta}_{j\ell_1} \Tilde{\beta}_{i\ell_2} - I^{g_0}_{\mu_j\lambda_j} I^{g_0}_{\mu_i\lambda_i} \beta_{j\ell_1} \beta_{i\ell_2} \right) \biggr| \nonumber \\
        \leq & \ \sum_{\ell=1}^d \left( \E \left[ \left( \sin(\theta_{jk} - \mu_j) \phi^{f_0}_{\ell k} \right)^2 \right] \right)^{1/2} \left( \E \left[ \left( \Tilde{I}_{\mu_i\lambda_i} \Tilde{\beta}_{i\ell} - I^{g_0}_{\mu_i\lambda_i} \beta_{i\ell} \right)^2 \right] \right)^{1/2} \nonumber \\
        & \hspace{0.5cm}
        + \sum_{\ell=1}^d \left( \E \left[ \left( \sin(\theta_{ik} - \mu_i) \phi^{f_0}_{\ell k} \right)^2 \right] \right)^{1/2} \left( \E \left[ \left( \Tilde{I}_{\mu_j\lambda_j} \Tilde{\beta}_{j\ell} - I^{g_0}_{\mu_j\lambda_j} \beta_{j\ell} \right)^2 \right] \right)^{1/2} \nonumber \\
        & \hspace{0.5cm} + \sum_{\ell_1=1}^d \sum_{\ell_2=1}^d \left( \E \left[ \left( \phi^{f_0}_{\ell_1 k} \phi^{f_0}_{\ell_2 k} \right)^2 \right] \right)^{1/2} \left( \E \left[ \left( \Tilde{I}_{\mu_j\lambda_j} \Tilde{I}_{\mu_i\lambda_i} \Tilde{\beta}_{j\ell_1} \Tilde{\beta}_{i\ell_2} - I^{g_0}_{\mu_j\lambda_j} I^{g_0}_{\mu_i\lambda_i} \beta_{j\ell_1} \beta_{i\ell_2} \right)^2 \right] \right)^{1/2} \nonumber 
    \end{align}
    Noting that $\E \left[ \left( \sin(\theta_{jk} - \mu_j) \phi^{f_0}_{\ell k} \right)^2 \right] = O(1)$, $\E \left[ \left( \phi^{f_0}_{\ell_1 k} \phi^{f_0}_{\ell_2 k} \right)^2 \right] = O(1)$ and using (iii) and (v), the result follows.
\end{proof}

\begin{proof}[Proof of Lemma \ref{lemma: rate of diff of gammas unspecified}]
All expectations in the proof are with respect to $g_0$. 
We omit this from the notation for simplicity. \\
    
\noindent (i) By Assumption \ref{assumption_on_f0_unspecified_median}, $Q_{f_0}^{\ast (n)} (\mub)$ is continuously differentiable with respect to $\mub$, with a bounded derivative as it is a continuous function on a bounded domain.
So, $\E \left[ \Big\lVert \left( { \left. \nabla_{\mub} Q_{f_0}^{\ast (n)} (\mub) \right|_{\mub = \mub^\ast}} \right)' \Big\rVert_2^2 \right]$ can be bounded by the constant 
\begin{equation}\label{eq: bound for grad Q}
    {C_4 = \sup_{\mub} \Bigl\lVert \nabla_{\mub} Q_{f_0}^{\ast (n)} (\mub) \Bigr\rVert_2^2.}
\end{equation}

\noindent (ii) 
    It holds that $\E \left[ \Tilde{\Delta}_{ik} \right] = 0$, so, using Lemma~\ref{lemma: used_for_unspecified_steins}(i), $$\E \left[ \left( \sum_{k=1}^n \sum_{\ell=1}^n \Tilde{\Delta}_{ik} \Tilde{\Delta}_{j\ell} \right)^2 \right] \leq C_5 n^2$$
    for some universal constant $C_5$ that depends on $f_0$ and $g_0$.\\

\noindent (iii) 
    Changing the indices of the sums for convenience, we find
    \begin{align}
        \E & \left[ \left( \sum_{k_1=1}^n \sum_{k_2=1}^n \left( \Tilde{\Delta}_{ik_1}^{\ast} \Tilde{\Delta}_{jk_2}^{\ast} - \Tilde{\Delta}_{ik_1} \Tilde{\Delta}_{jk_2} \right) \right)^2 \right] \nonumber \\
        & = \ \E \left[ \left( -\sum_{k_1=1}^n \sum_{k_2=1}^n \sum_{\ell=1}^d \sin(\theta_{ik_1} - \mu_i) \phi^{f_0}_{\ell k_2} \left( \Tilde{I}_{\mu_j\lambda_j} \Tilde{\beta}_{j\ell} - I_{\mu_j\lambda_j} \beta_{j\ell} \right) \right.\right. \nonumber\\
        & \hspace{1.5cm} -\sum_{k_1=1}^n \sum_{k_2=1}^n \sum_{\ell=1}^d \sin(\theta_{jk_2} - \mu_j) \phi^{f_0}_{\ell k_1} \left( \Tilde{I}_{\mu_i\lambda_i} \Tilde{\beta}_{i\ell} - I_{\mu_i\lambda_i} \beta_{i\ell} \right) \nonumber\\
        & \hspace{1.5cm} \left. \left. + \sum_{k_1=1}^n \sum_{k_2=1}^n \sum_{\ell_1=1}^d \sum_{\ell_2=1}^d \phi^{f_0}_{\ell_1 k_1} \phi^{f_0}_{\ell_2 k_2} \left( \Tilde{I}_{\mu_i\lambda_i} \Tilde{I}_{\mu_j\lambda_j} \Tilde{\beta}_{i\ell_1} \Tilde{\beta}_{j\ell_2} - I_{\mu_i\lambda_i} I_{\mu_j\lambda_j} \beta_{i\ell_1} \beta_{j\ell_2} \right) \right)^2 \right] \nonumber\\
        & \leq \ (2d+d^2) \left\{ \E \left[ \left(\sum_{k_1=1}^n \sum_{k_2=1}^n \sin(\theta_{ik_1} - \mu_i) \phi^{f_0}_{\ell k_2} \left( \Tilde{I}_{\mu_j\lambda_j} \Tilde{\beta}_{j\ell} - I_{\mu_j\lambda_j} \beta_{j\ell} \right) \right)^2 \right] \right. \label{eq: term_B_term_A}\\
        & \hspace{1.5cm} + \E \left[ \left( \sum_{k_1=1}^n \sum_{k_2=1}^n \sin(\theta_{jk_2} - \mu_j) \phi^{f_0}_{\ell k_1} \left( \Tilde{I}_{\mu_i\lambda_i} \Tilde{\beta}_{i\ell} - I_{\mu_i\lambda_i} \beta_{i\ell} \right) \right)^2 \right] \label{eq: term_B_term_B}\\
        & \hspace{1.5cm} + \left. \E \left[ \left( \sum_{k_1=1}^n \sum_{k_2=1}^n \phi^{f_0}_{\ell_1 k_1} \phi^{f_0}_{\ell_2 k_2} \left( \Tilde{I}_{\mu_i\lambda_i} \Tilde{I}_{\mu_j\lambda_j} \Tilde{\beta}_{i\ell_1} \Tilde{\beta}_{j\ell_2} - I_{\mu_i\lambda_i} I_{\mu_j\lambda_j} \beta_{i\ell_1} \beta_{j\ell_2} \right) \right)^2 \right] \right\}. \label{eq: term_B_term_C}
    \end{align}
    Considering the first term of the sum, \eqref{eq: term_B_term_A}, by the Cauchy-Schwarz inequality we readily obtain
    \begin{align*}
        \E \Biggl[ \Biggl( & \sum_{k_1=1}^n \sum_{k_2=1}^n \sin(\theta_{ik_1} - \mu_i) \phi^{f_0}_{\ell k_2} \left( \Tilde{I}_{\mu_j\lambda_j} \Tilde{\beta}_{j\ell} - I_{\mu_j\lambda_j} \beta_{j\ell} \right) \Biggr)^2 \Biggr] \\
        & \leq \left( \E \left[ \left(\sum_{k_1=1}^n \sum_{k_2=1}^n \sin(\theta_{ik_1} - \mu_i) \phi^{f_0}_{\ell k_2} \right)^4 \right] \right)^{1/2} 
        \left( \E \left[ \left( \Tilde{I}_{\mu_j\lambda_j} \Tilde{\beta}_{j\ell} - I_{\mu_j\lambda_j} \beta_{j\ell} \right)^4 \right] \right)^{1/2} \\
        & = \ O(n),
    \end{align*}
    where we used Lemma~\ref{lemma: used_for_unspecified_steins}(ii) and (iv).
    Similarly, it can be proved that $\eqref{eq: term_B_term_B} = O(n)$.
    Considering \eqref{eq: term_B_term_C} and using Lemma~\ref{lemma: used_for_unspecified_steins}(ii) and (vi) yields
    \begin{align*}
        \E & \left[ \left( \sum_{k_1=1}^n \sum_{k_2=1}^n \phi^{f_0}_{\ell_1 k_1} \phi^{f_0}_{\ell_2 k_2} \left( \Tilde{I}_{\mu_i\lambda_i} \Tilde{I}_{\mu_j\lambda_j} \Tilde{\beta}_{i\ell_1} \Tilde{\beta}_{j\ell_2} - I_{\mu_i\lambda_i} I_{\mu_j\lambda_j} \beta_{i\ell_1} \beta_{j\ell_2} \right) \right)^2 \right] \\
        & \leq \left( \E \left[ \left( \sum_{k_1=1}^n \sum_{k_2=1}^n \phi^{f_0}_{\ell_1 k_1} \phi^{f_0}_{\ell_2 k_2} \right)^4 \right] \right)^{1/2} \left(\E \left[\left( \Tilde{I}_{\mu_i\lambda_i} \Tilde{I}_{\mu_j\lambda_j} \Tilde{\beta}_{i\ell_1} \Tilde{\beta}_{j\ell_2} - I_{\mu_i\lambda_i} I_{\mu_j\lambda_j} \beta_{i\ell_1} \beta_{j\ell_2} \right)^4 \right] \right)^{1/2} \\
        & = \ O(n).
    \end{align*}
    Putting all ends together, we get
    \begin{equation*}
        \E \left[ \left( \sum_{k=1}^n \sum_{\ell=1}^n \left( \Tilde{\Delta}_{ik}^{\ast} \Tilde{\Delta}_{j\ell}^{\ast} - \Tilde{\Delta}_{ik} \Tilde{\Delta}_{j\ell} \right) \right)^2 \right] \leq C_6 n
    \end{equation*}
    for some universal constant $C_6$ that depends on $f_0$ and $g_0$.\\

    \noindent (iv) 
    Since the $\gamma_{ij}$s are the entries of the inverse of a matrix, we can write
    $\hat{\gamma}_{ij} = (-1)^{i+j}{\hat{d}^v_{ij}} /{\hat{d}^v}$ and ${\gamma}_{ij} = (-1)^{i+j}{d^v_{ij}}/d^v$ for
    $\hat{d}^v = \det\left( \hat{V}_{f_0}\left(\mub \right) \right), \hat{d}^v_{ij} = \det \left( \hat{V}_{f_0}\left(\mub \right)_{(ij)}\right), 
    d^v = \det\left( V_{g_0}^{f_0} \left( \mub \right) \right)$ and $d^v_{ij} = \det \left( V_{g_0}^{f_0} \left( \mub \right)_{(ij)} \right)$, where $A_{(ij)}$ denotes the matrix $A$ after removing the $i^\text{th}$ row and $j^\text{th}$ column.
    Thus
    \begin{equation*}
        \E \left[ \left( \hat{\gamma}_{ij} - \gamma_{ij} \right)^2 \right] 
        = \E \left[ \left( (-1)^{i+j} \frac{\hat{d}^v_{ij}}{\hat{d}^v} - (-1)^{i+j} \frac{d^v_{ij}}{d^v} \right)^2 \right].
    \end{equation*}
    Applying the Taylor expansion to $1/\hat{d}^v$ about $d^v$ along with straightforward manipulations gives us
    \begin{equation}\label{eq: gamma_taylor_unspecified_steins}
        \frac{\hat{d}^v_{ij}}{\hat{d}^v}
        = \frac{d^v_{ij}}{d^v}
        + \frac{\hat{d}^v_{ij} - d^v_{ij}}{d^v}
        - \frac{d^v_{ij}(\hat{d}^v - d^v)}{(d^v)^2}
        - \frac{(\hat{d}^v_{ij} - d^v_{ij})(\hat{d}^v - d^v)}{(d^v)^2}
        + o \left( |\hat{d}^v - d^v| \right),
    \end{equation}
    and hence
    \begin{align}
        \E \left[ \left( \frac{\hat{d}^v_{ij}}{\hat{d}^v} - \frac{d^v_{ij}}{d^v} \right)^2 \right]
        \leq 4 \Biggl[ &
        \frac{1}{\left( d^v \right)^2} \E\left[ \left( \hat{d}^v_{ij} - d^v_{ij} \right)^2 \right] 
        + \frac{\left( d^v_{ij} \right)^2}{\left( d^v \right)^4} \E\left[ \left( \hat{d}^v - d^v \right)^2 \right] \nonumber\\
        & + \frac{1}{\left( d^v \right)^4} \E\left[ \left( \hat{d}^v_{ij} - d^v_{ij} \right)^2 \left( \hat{d}^v - d^v \right)^2 \right] 
        + o \left( \E\left[ \left( \hat{d}^v - d^v \right)^2 \right] \right)
        \Biggr]. \label{eq: mid_step_equation_unspecified}
    \end{align}
    Similarly to the proof of Theorem \ref{thm: steins method specified}, we will focus on proving that $\E\left[ \left( \hat{d}^v_{ij} - d^v_{ij} \right)^2 \right] = O \left(\frac{1}{n} \right)$, as the result follows from there. 
    Indeed, apart from showing the order of the first term of the bound in \eqref{eq: mid_step_equation_unspecified}, such a result yields two more outcomes. 
    Firstly, it also implies that the second and last term is $O\left(\frac{1}{n}\right)$. Secondly, this means that the third  term of the bound in \eqref{eq: mid_step_equation_unspecified} 
    can be upper-bounded by an $O\left( \frac{1}{n} \right)$ term, which then proves the required result.\\

\noindent In our calculations now, the focus is put on the order of the determinant, and not on its exact expression, because the last term of \eqref{eq: mid_step_equation_unspecified} already prevents us from calculating an explicit constant for the bound of \eqref{eq: difference of statistics unspecified}.
The determinant of a $d\times d$ matrix is given by $d!$ summands, each being a product of $d$ elements of the matrix.
For $\alpha, \beta \in \{1,\ldots,d\}$
we can write 
\begin{equation}
\label{eq:d_ab_unspecified}
    d^{v}_{\alpha \beta} = \sum_{k = 1}^{(d-1)!} \prod_{i=1}^{d-1} (-1)^{k} J_{ik}
\end{equation}
where $J_{ik} = \E \left[ \Tilde{\Delta}_{h_k(i)1} \Tilde{\Delta}_{\ell_k(i)1} \right]$ 
with $h_k(i)$ and $\ell_k(i)$ being (possibly different) linear functions of $i$ that take values in $\{1,\ldots,d\}$. 
These functions are of the form $h_k(i) = \{c\pm i\}\mod d + 1$ for $c\in \{1,\ldots,d\}$.
Similarly, 
\begin{align*}
    \hat{d}^{v}_{\alpha \beta} 
    = & \sum_{k=1}^{(d-1)!} \prod_{i = 1}^{d-1} \frac{1}{n} \sum_{j=1}^n (-1)^{k} \hat{J}_{ik;j}
    = \ \frac{1}{n^{d-1}} \sum_{k = 1}^{(d-1)!}
    \sum\limits_{ \substack{ \scriptscriptstyle{j_\eta=1} \\ \scriptscriptstyle{\eta \in B} }}^n \prod_{i \in B} (-1)^{k} \hat{J}_{ik;j_\eta}
\end{align*}
where $\hat{J}_{ik; j} = \Tilde{\Delta}_{h_k(i)j}^{\ast} \Tilde{\Delta}_{\ell_k(i)j}^{\ast}$ and $B = \{1,\ldots,d-1\}$.
    By Lemma~\ref{lemma: used_for_unspecified_steins} (vii), it holds that $\E \left[ \hat{J}_{ik; j} \right] - {J}_{ik} = O \left( \frac{1}{\sqrt{n}} \right)$, which is the biggest difference of this proof compared to that of Theorem \ref{thm: steins method specified}.
    It also holds that $\Tilde{\Delta}_{kj}^{\ast} = O(1)$.
The above general expressions for the determinant allow us to calculate the necessary orders without using the explicit functions of $h_k(i)$ and $\ell_k(i)$.
Note that the sum over $\eta$ and $j_\eta$ is a product of $(d-1)$ sums of $n$ summands each.
Define $B[-i]$ to be the set $B$ without the $i^\text{th}$ element,
then, following the same steps as in Theorem \ref{thm: steins method specified}, we obtain
\begin{align}
    & \E \left[\left(\hat{d}^{v}_{\alpha\beta} - d^{v}_{\alpha\beta} \right)^2\right] \nonumber\\
    & = \ \frac{1}{n^{2d-2}} \sum_{k_1,k_2 = 1}^{(d-1)!} 
    \sum\limits_{ \substack{\scriptscriptstyle{j_{\eta_1},l_{\eta_2}=1} \\ 
    \scriptscriptstyle{\eta_1, \eta_2 \in B} \\ 
    \scriptscriptstyle{l_1 \neq l_{2}}}}^n
    \E \Biggl[ \left( \prod_{i\in B} (-1)^{k_1} \hat{J}_{ik_1;j_{\eta_1}} 
    - \frac{d^{v}_{\alpha\beta}}{(d-1)!} \right) 
    \left( \prod_{i\in B} (-1)^{k_2} \hat{J}_{ik_2;l_{\eta_2}} 
    - \frac{d^{v}_{\alpha\beta}}{(d-1)!} \right) \Biggr] \nonumber \\
    & \hspace{0.2cm} + \frac{1}{n^{2d-2}} \sum_{k_1,k_2 = 1}^{(d-1)!}
    \sum\limits_{\substack{ \scriptscriptstyle{j_{\eta_1},l_{\eta_2}=1} \\ \scriptscriptstyle{\eta_1 \in B} \\ \scriptscriptstyle{\eta_2 \in B[-1]} \\ \scriptscriptstyle{l_{1} = l_{2}}}}^n
    \E \Biggl[ \left( \prod_{i\in B} (-1)^{k_1} \hat{J}_{ik_1;j_{\eta_1}} 
    - \frac{d^{v}_{\alpha\beta}}{(d-1)!} \right)
    \left( \prod_{i\in B} (-1)^{k_2} \hat{J}_{ik_2;l_{\eta_2}} 
    - \frac{d^{v}_{\alpha\beta}}{(d-1)!} \right) \Biggr] \label{eq: expectation_order_1n_term_unspecified} 
\end{align}
The second term in \eqref{eq: expectation_order_1n_term_unspecified} is the product of $2d-3$ sums from $1$ up to $n$. 
So, this term is $O\left(\frac{1}{n}\right)$. 
Similarly, for all other terms that involve indices that are equal to each other, we see that they can contain at most $2d-3$ sums from $1$ up to $n$ and so are at most $O(\frac{1}{n})$. 
It remains to consider the case where all indices are not equal to each other. To this end, working with the same set as in \eqref{eq: set_c}, having recourse to independence under different indices and using the above-mentioned fact that we have that $\E \left[ \hat{J}_{ik; j} \right] - {J}_{ik} = O \left( \frac{1}{\sqrt{n}} \right)$, lengthy calculations (similar to previous proofs, hence we spare the details to the reader) give
\begin{align*}
    & \E \left[\left(\hat{d}^{v}_{\alpha\beta} - d^{v}_{\alpha\beta} \right)^2\right] \nonumber \\
    & = \ \frac{n(n-1) \ldots (n-2d+3)}{n^{2d-2}} 
    \left( \sum_{k = 1}^{(d-1)!} \prod_{i\in B} \left( (-1)^k J_{ik} + O \left( \frac{1}{\sqrt{n}} \right) \right)
    - d^{v}_{\alpha\beta} \right)^2 + O \left( \frac{1}{n} \right) \nonumber \\
    & = \ \frac{n(n-1) \ldots (n-2d+3)}{n^{2d-2}} 
    \left( \sum_{k = 1}^{(d-1)!} \prod_{i\in B} (-1)^k J_{ik}
    - d^{v}_{\alpha\beta} + O \left( \frac{1}{\sqrt{n}} \right) \right)^2 + O \left( \frac{1}{n} \right) \nonumber \\
    & = O\left( \frac{1}{n} \right),
\end{align*}
where the last equality is a result of the definition of $d^{v}_{\alpha\beta}$ as in \eqref{eq:d_ab_unspecified}. Thus, it holds that
\begin{equation*}
    \E_{g_0} \left[ \left( \hat{\gamma}_{ij} - \gamma_{ij} \right)^2 \right] \leq C_7/n
\end{equation*}
for some universal constant $C_7$ that depends on $f_0$ and $g_0$.\\

\noindent (v)
Direct manipulations yield
$$
        \E \left[ \hat{\gamma}^2_{ij} \right]
         \leq 2  \E \left[ \gamma^2_{ij} \right] +2 \E \left[ \left(\hat{\gamma}^2_{ij}-\gamma^2_{ij}\right) \right]     
        \leq C_8
$$ 
which follows from (iv) and the boundedness of $\gamma_{ij}$, for some universal constant $C_8$ that depends on $f_0$ and $g_0$.
\end{proof}

\section{Further simulation results} \label{sec: supp_simulations}

In this section, we present further simulation results. 
The set-up of the simulations is the same as explained in Section \ref{sec: simulations} of the main paper.

\begin{table}[htbp]\centering
\caption{Percentage of rejections for $\phi^{\ast (n); \mu}$ when $d=2$. The data are generated using the model mentioned, $n$ represents the sample size and $\lambdab$ the value of the skewness parameter.\label{tab: simulations_2_known_supp}}
\begin{tabular}{@{}rrrrrrrr@{}} 
\toprule
& \multicolumn{1}{c}{$\lambdab$} & {$(0,0)$} & {$(0.1,0)$} &
{$(0.1,0.1)$} & {$(0.2,0.1)$} & {$(0.2,0.2)$}\\ 
{Model} & \multicolumn{1}{c}{$n$} & \\ \midrule
& 200 & 0.042 & 0.097 & 0.115 & 0.286 & 0.334 \\
$I_{0.6;0.9}$ & 500 & 0.046 & 0.185 & 0.232 & 0.648 & 0.724 \\
& 1000 & 0.034 & 0.319 & 0.434 & 0.918 & 0.954 \\
\\
& 200 & 0.047 & 0.109 & 0.153 & 0.435 & 0.542 \\
$S_{1; 5; 0.3}$ & 500 & 0.048 & 0.219 & 0.326 & 0.823 & 0.917 \\
& 1000 & 0.047 & 0.427 & 0.633 & 0.989 & 0.997 \\
\\
& 200 & 0.047 & 0.116 & 0.140 & 0.382 & 0.512 \\
$C_{1; 5; 0.3}$ & 500 & 0.056 & 0.231 & 0.329 & 0.816 & 0.898 \\
& 1000 & 0.051 & 0.462 & 0.612 & 0.980 & 0.994 \\
\\
& 200 & 0.061 & 0.125 & 0.231 & 0.518 & 0.777 \\
$BWC_{0.1; 0.1; 0.1}$ & 500 & 0.045 & 0.251 & 0.573 & 0.919 & 0.993 \\
& 1000 & 0.058 & 0.506 & 0.857 & 1.000 & 1.000 \\
\bottomrule
\end{tabular}
\end{table}

\begin{table}[htbp]\centering
\caption{Percentage of rejections for $\phi^{\ast (n); \mu}$ when $d=3$. The data are generated using the model mentioned, $n$ represents the sample size and $\lambdab$ the value of the skewness parameter.\label{tab: simulations_3_known_supp}}
\begin{tabular}{@{}rrrrrrrr@{}} 
\toprule
& \multicolumn{1}{c}{$\lambdab$} & {$(0,0,0)$} & {$(0.1,0,0)$} &
{$(0.2,0.1,0)$} & {$(0.2,0.2,0.2)$} \\ 
{Model} & \multicolumn{1}{c}{$n$} & \\ \midrule
& 200 & 0.043 & 0.097 & 0.399 & 0.805 \\
$I_{0.1; 0.2; 0.3}$ & 500 & 0.054 & 0.251 & 0.854 & 0.999 \\
& 1000 & 0.044 & 0.435 & 0.993 & 1.000 \\
\\
& 200 & 0.066 & 0.120 & 0.216 & 0.378 \\
$TWC_{1; 1.2; 0.5}^{0.1; 0.2; 0.3}$ & 500 & 0.058 & 0.246 & 0.548 & 0.794 \\
& 1000 & 0.045 & 0.452 & 0.845 & 0.983\\
\bottomrule
\end{tabular}
\end{table}

\begin{table}[htbp]\centering
\caption{Percentage of rejections for $\phi^{\ast (n); \mu}$ when $d=4,6$. The data are generated using the model mentioned, $n$ represents the sample size and $\lambdab$ the value of the skewness parameter. Note that the missing entries correspond to non-allowed settings where the sum of the components of $\lambdab$ would exceed 1.\label{tab: simulations_456_known_supp}}
\begin{tabular}{@{}rrrrrrrr@{}} 
\toprule
& & \multicolumn{1}{c}{$\lambdab$} & {$(0,\ldots,0)$} & {$(0.05,\ldots,0.05)$} &
{$(0.1,\ldots,0.1)$} & {$(0.15,\ldots,0.15)$} \\ 
{Model} & $d$ & \multicolumn{1}{c}{$n$} & \\ \midrule
& & 200 & 0.037 & 0.086 & 0.309 & 0.638 \\
$I_{0.1; \ldots; 0.1}$ & 4 & 500 & 0.046 & 0.191 & 0.722 & 0.978 \\
& & 1000 & 0.050 & 0.391 & 0.959 & 1.000 \\
\\
& & 200 & 0.050 & 0.077 & 0.213 & 0.455 \\
$I_{0.6; \ldots; 0.6}$ & 4 & 500 & 0.043 & 0.152 & 0.485 & 0.863\\
& & 1000 & 0.050 & 0.242 & 0.834 & 0.989 \\
\\
& & 200 & 0.046 & 0.107 & 0.313 & 0.639 \\
$MNNTS_4$ & 4 & 500 & 0.052 & 0.196 & 0.709 & 0.984 \\
& & 1000 & 0.050 & 0.379 & 0.968 & 1.000 \\
\\
& & 200 & 0.047 & 0.110 & 0.401 & 0.775 \\
$I_{0.1; \ldots; 0.1}$ & 6 & 500 & 0.047 & 0.231 & 0.827 & 0.998 \\
& & 1000 & 0.041 & 0.503 & 0.993 & 1.000 \\
\\
& & 200 & 0.037 & 0.071 & 0.252 & 0.568 \\
$I_{0.6; \ldots; 0.6}$ & 6 & 500 & 0.062 & 0.176 & 0.599 & 0.961 \\
& & 1000 & 0.062 & 0.285 & 0.918 & 1.000 \\
\\
& & 200 & 0.045 & 0.118 & 0.393 & 0.766 \\
$MNNTS_6$ & 6 & 500 & 0.054 & 0.259 & 0.830 & 0.998 \\
& & 1000 & 0.031 & 0.492 & 0.992 & 1.000 \\
\bottomrule
\end{tabular}
\end{table}

\begin{table}[htbp]\centering
\caption{Percentage of rejections for $\phi^{\ast (n)}_{f_0}$ when $d=2$. The data are generated using the model $g_0$ and the test statistic is evaluated using the model $f_0$, $n$ represents the sample size and $\lambdab$ the value of the skewness parameter.\label{tab: simulations_2_unknown_ind}}
\begin{tabular}{@{}rrrrrrrr@{}} 
\toprule
\multicolumn{2}{c}{$g_0 = I_{0.1; 0.1}$}\\
\midrule
& \multicolumn{1}{c}{$\lambdab$} & {$(0,0)$} & {$(0.1,0)$} &
{$(0.1,0.1)$} & {$(0.2,0.1)$} & {$(0.2,0.2)$}\\ 
$f_0$ & \multicolumn{1}{c}{$n$} & \\
\midrule
& 200 & 0.054 & 0.123 & 0.236 & 0.482 & 0.705 \\
$I_{0.1;0.1}$ & 500 & 0.057 & 0.270 & 0.506 & 0.876 & 0.984 \\
& 1000 & 0.050 & 0.485 & 0.797 & 0.996 & 1.000 \\
\\
& 200 & 0.059 & 0.134 & 0.249 & 0.510 & 0.743 & \\
$I_{0.6;0.9}$ & 500 & 0.063 & 0.271 & 0.514 & 0.895 & 0.987 \\
& 1000 & 0.048 & 0.503 & 0.819 & 0.999 & 1.000 \\
\midrule
& \multicolumn{1}{c}{$\lambdab$} & {$(0,0)$} & {$(0.2,0)$} &
{$(0.2,0.2)$} & {$(0.4,0.2)$} & {$(0.4,0.4)$}\\ 
\midrule
& 200 & 0.054 & 0.145 & 0.273 & 0.541 & 0.746 \\
& 500 & 0.049 & 0.184 & 0.330 & 0.617 & 0.822 \\
$S_{0.5; 0.5; 0.1}$ & 1000 & 0.054 &  0.205 & 0.363 & 0.666 & 0.914 \\
& 5000 & 0.036  & 0.213 & 0.462 & 0.943 & 1.000 \\
\\
& 200 & 0.055 & 0.184 & 0.207 & 0.498 & 0.590 \\
$BWC_{0.1; 0.1; 0.1}$ & 500 & 0.052 & 0.225 & 0.224 & 0.512 & 0.595 \\
& 1000 & 0.048 & 0.285 & 0.263 & 0.538 & 0.614 \\
& 5000  & 0.039 & 0.342 & 0.302 & 0.685 & 0.7050 \\
\\
& 200 & 0.049 & 0.274 & 0.280 & 0.572 & 0.639 \\
$BWC_{0.1; 0.5; 0.3}$ & 500 & 0.046 & 0.439 & 0.290 & 0.647 & 0.664 \\
& 1000 & 0.052 & 0.625 & 0.304 & 0.727 & 0.656 \\
& 5000 & 0.061 & 0.832 & 0.336 & 0.825 & 0.666 \\
\bottomrule
\end{tabular}
\end{table}

\begin{table}[htbp]\centering
\caption{Percentage of rejections for $\phi^{\ast (n)}_{f_0}$ when $d=2$. The data are generated using the the model $g_0$ and the test statistic is evaluated using the model $f_0$, $n$ represents the sample size and $\lambdab$ the value of the skewness parameter.\label{tab: simulations_2_unknown_bwc}}
\begin{tabular}{@{}rrrrrrrr@{}} 
\toprule
\multicolumn{2}{c}{$g_0 = BWC_{0.5; 0.5; 0.3}$}\\
\midrule
& \multicolumn{1}{c}{$\lambdab$} & {$(0,0)$} & {$(0.1,0)$} &
{$(0.1,0.1)$} & {$(0.2,0.1)$} & {$(0.2,0.2)$}\\ 
$f_0$ & \multicolumn{1}{c}{$n$} & \\
\midrule
& 500 & 0.053 & 0.091 & 0.121 & 0.272 & 0.365 \\
$I_{0.1;0.1}$ & 1000 & 0.057 & 0.114 & 0.182 & 0.461 & 0.674 \\
& 5000 & 0.052 & 0.424 & 0.780 & 0.991 & 1.000 \\
\\
& 500 & 0.045 & 0.087 & 0.123 & 0.264 & 0.428 \\
$BWC_{0.1; 0.1; 0.1}$ & 1000 & 0.056 & 0.158 & 0.211 & 0.533 & 0.707 \\
& 5000 & 0.046 & 0.582 & 0.826 & 0.997 & 1.000 \\
\\
& 500 & 0.061 & 0.104 & 0.180 & 0.357 & 0.519 \\
$BWC_{0.5; 0.5; 0.3}$ & 1000 & 0.058 & 0.193 & 0.272 & 0.613 & 0.811 \\
& 5000 & 0.049 & 0.719 & 0.900 & 1.000 & 1.000 \\
\midrule
& \multicolumn{1}{c}{$\lambdab$} & {$(0,0)$} & {$(0.2,0)$} &
{$(0.2,0.2)$} & {$(0.4,0.2)$} & {$(0.4,0.4)$}\\ 
\midrule
& 500 & 0.046 & 0.150 & 0.100 & 0.276 & 0.420 \\
$S_{0.5; 0.5; 0.1}$ & 1000 & 0.045 & 0.219 & 0.134 & 0.520 & 0.747 \\
& 5000 & 0.057 & 0.802 & 0.599 & 0.996 & 1.000 \\
\bottomrule
\end{tabular}
\end{table}

\begin{table}[htbp]\centering
\caption{Percentage of rejections for $\phi^{\ast (n)}_{f_0}$ when $d=2$. The data are generated using the the model $g_0$ and the test statistic is evaluated using the model $f_0$, $n$ represents the sample size and $\lambdab$ the value of the skewness parameter.\label{tab: simulations_2_unknown_mnnts}}
\begin{tabular}{@{}rrrrrrrr@{}} 
\toprule
\multicolumn{2}{c}{$g_0 = MNNTS_2$}\\
\midrule
& \multicolumn{1}{c}{$\lambdab$} & {$(0,0)$} & {$(0.1,0)$} &
{$(0.1,0.1)$} & {$(0.2,0.1)$} & {$(0.2,0.2)$}\\ 
$f_0$ & \multicolumn{1}{c}{$n$} & \\
\midrule
& 500 & 0.061 & 0.074 & 0.106 & 0.213 & 0.320 \\
$I_{0.1;0.1}$ & 1000 & 0.060 & 0.113 & 0.199 & 0.350 & 0.542 \\
& 5000 & 0.047 & 0.346 & 0.608 & 0.940 & 0.992 \\
\\
& 500 & 0.093 & 0.331 & 0.571 & 0.909 & 0.991 \\
$I_{0.6; 0.9}$ & 1000 & 0.066 & 0.532 & 0.829 & 0.995 & 1.000 \\
& 5000 & 0.050 & 0.992 & 1.000 & 1.000 & 1.000 \\
\\
& 500 & 0.043 & 0.063 & 0.125 & 0.223 & 0.359 \\
$BWC_{0.1; 0.1; 0.1}$ & 1000 & 0.040 & 0.090 & 0.198 & 0.372 & 0.623 \\
& 5000 & 0.051 & 0.240 & 0.701 & 0.971 & 1.000 \\
\\
& 200 & 0.067 & 0.075 & 0.063 & 0.131 & 0.157 \\
$BWC_{0.5; 0.5; 0.3}$ & 500 & 0.051 & 0.088 & 0.098 & 0.167 & 0.241 \\
& 1000 & 0.045 & 0.085 & 0.149 & 0.240 & 0.443 \\
& 5000 & 0.047 & 0.201 & 0.540 & 0.846 & 0.979 \\
\midrule
& \multicolumn{1}{c}{$\lambdab$} & {$(0,0)$} & {$(0.2,0)$} &
{$(0.2,0.2)$} & {$(0.4,0.2)$} & {$(0.4,0.4)$}\\ 
\midrule
& 500 & 0.048 & 0.046 & 0.064 & 0.103 & 0.329 \\
$S_{0.5; 0.5; 0.1}$ & 1000 & 0.035 & 0.065 & 0.072 & 0.137 & 0.486 \\
& 5000 & 0.048 & 0.067 & 0.084 & 0.355 & 0.964 \\
\bottomrule
\end{tabular}
\end{table}

\begin{table}[htbp]\centering
\caption{Percentage of rejections for $\phi^{\ast (n)}_{f_0}$ when $d=3$. The data are generated using the model $g_0$ and the test statistic is evaluated using the model $f_0$,  $n$ represents the sample size and $\lambdab$ the value of the skewness parameter.\label{tab: simulations_3_unknown_iid}}
\begin{tabular}{@{}rrrrrrrr@{}} 
\toprule
\multicolumn{2}{c}{$g_0 = I_{0.1; 0.1; 0.1}$}\\
\midrule
& \multicolumn{1}{c}{$\lambdab$} & {$(0,0,0)$} & {$(0.1,0,0)$} &
{$(0.2,0.1,0)$} & {$(0.2,0.2,0.2)$} \\ 
$f_0$ & \multicolumn{1}{c}{$n$} & \\ \midrule
& 200 & 0.055 & 0.121 & 0.444 & 0.835 \\
$I_{0.1; 0.1; 0.1}$ & 500 & 0.047 & 0.234 & 0.843 & 0.999 \\
& 1000 & 0.061 & 0.422 & 0.992 & 1.000 \\
\\
& 200 & 0.051 & 0.121 & 0.428 & 0.766 \\
$I_{0.1; 0.2; 0.3}$ & 500 & 0.039 & 0.232 & 0.829 & 0.986 \\
& 1000 & 0.061 & 0.432 & 0.991 & 1.000 \\
\\
& 200 & 0.046 & 0.091 & 0.342 & 0.777 \\
$TWC_{5; 2; 0.1}^{0.1; 0.1; 0.1}$ & 500 & 0.054 & 0.196 & 0.808 & 0.999 \\
& 1000 & 0.043 & 0.367 & 0.984 & 1.000 \\
\\
& 200 & 0.060 & 0.105 & 0.307 & 0.727 \\
$TWC_{1; 1.2; 0.5}^{0.1; 0.2; 0.3}$ & 500 & 0.061 & 0.145 & 0.548 & 0.982 \\
& 1000 & 0.051 & 0.244 & 0.753 & 1.000 \\
\bottomrule
\end{tabular}
\end{table}

\begin{table}[htbp]\centering
\caption{Percentage of rejections for $\phi^{\ast (n)}_{f_0}$ when $d=3$. The data are generated using the model $g_0$ and the test statistic is evaluated using the model $f_0$,  $n$ represents the sample size and $\lambdab$ the value of the skewness parameter.\label{tab: simulations_3_unknown_supp}}
\begin{tabular}{@{}rrrrrrrr@{}} 
\toprule
\multicolumn{2}{c}{$g_0 = MNNTS_3$}\\
\midrule
& \multicolumn{1}{c}{$\lambdab$} & {$(0,0,0)$} & {$(0.1,0,0)$} &
{$(0.2,0.1,0)$} & {$(0.2,0.2,0.2)$} \\ 
$f_0$ & \multicolumn{1}{c}{$n$} & \\ 
\midrule
& 200 & 0.060 & 0.077 & 0.099 & 0.216 \\
$I_{0.1; 0.2; 0.3}$ & 500 & 0.053 & 0.074 & 0.161 & 0.394 \\
& 1000 & 0.050 & 0.095 & 0.233 & 0.693 \\
\\
& 200 & 0.065 & 0.082 & 0.250 & 0.572 \\
$TWC_{5; 2; 0.1}^{0.1; 0.1; 0.1}$ & 500 & 0.053 & 0.114 & 0.543 & 0.924 \\
& 1000 & 0.051 & 0.167 & 0.859 & 0.999 \\
\\
\bottomrule
\end{tabular}
\end{table}

\begin{table}[htbp]\centering
\caption{Percentage of rejections for $\phi^{\ast (n)}_{f_0}$ when $d=4,5,6$. The data are generated using the model $g_0$ and the test statistic is evaluated using the model $f_0$,  $n$ represents the sample size and $\lambdab$ the value of the skewness parameter. \label{tab: simulations_456_unknown_mnnts}}
\begin{tabular}{@{}rrrrrrrr@{}} 
\toprule
\multicolumn{3}{c}{$g_0 = MNNTS_d$}\\
\midrule
& & \multicolumn{1}{c}{$\lambdab$} & {$(0,\ldots,0)$} & {$(0.05,\ldots,0.05)$} &
{$(0.1,\ldots,0.1)$} & {$(0.15,\ldots,0.15)$} & {$(0.2,\ldots,0.2)$} \\ 
{$f_0$} & $d$ & \multicolumn{1}{c}{$n$} & \\
\midrule
& & 200 & 0.078 & 0.058 & 0.065 & 0.075 & 0.108 \\
$I_{0.1; \ldots; 0.1}$ & 4 & 500 & 0.049 & 0.052 & 0.102 & 0.156 & 0.257 \\
& & 1000 & 0.071 & 0.080 & 0.131 & 0.362 & 0.617 \\
\\
& & 200 & 0.093 & 0.074 & 0.177 & 0.469 & 0.801 \\
$I_{0.6; \ldots; 0.6}$ & 4 & 500 & 0.069 & 0.134 & 0.595 & 0.945 & 0.999 \\
& & 1000 & 0.056 & 0.320 & 0.937 & 1.000 & 1.000 \\
\\
& & 200 & 0.075 & 0.066 & 0.059 & 0.072 & 0.119 \\
$I_{0.1; \ldots; 0.1}$ & $5$ & 500 & 0.064 & 0.057 & 0.099 & 0.176 & 0.287 \\
& & 1000 & 0.067 & 0.051 & 0.166 & 0.396 & 0.692 \\
\\
& & 200 & 0.084 & 0.067 & 0.181 & 0.517 & 0.841 \\
$I_{0.6; \ldots; 0.6}$ & $5$ & 500 & 0.063 & 0.154 & 0.667 & 0.975 & 1.000 \\
& & 1000 & 0.061 & 0.385 & 0.971 & 1.000 & 1.000 \\
\\
& & 200 & 0.078 & 0.066 & 0.070 & 0.068 & - \\
$I_{0.1; \ldots; 0.1}$ & 6 & 500 & 0.065 & 0.061 & 0.083 & 0.182 & - \\
& & 1000 & 0.054 & 0.062 & 0.182 & 0.424 & - \\
\\
& & 200 & 0.090 & 0.079 & 0.236 & 0.560 & - \\
$I_{0.6; \ldots; 0.6}$ & 6 & 500 & 0.053 & 0.159 & 0.716 & 0.989 & - \\
& & 1000 & 0.065 & 0.403 & 0.982 & 1.000 & - \\
\bottomrule
\end{tabular}
\end{table}

\begin{table}[htbp]\centering
\caption{Percentage of rejections for $\phi^{\ast (n)}_{f_0}$ when $d=4,5,6,10,20$. The data are generated using the model $g_0$ and the test statistic is evaluated using the model $f_0$,  $n$ represents the sample size and $\lambdab$ the value of the skewness parameter. \label{tab: simulations_456_unknown_ind}}
\begin{tabular}{@{}rrrrrrrr@{}} 
\toprule
\multicolumn{3}{c}{$g_0 = I_{0.6; \ldots; 0.6}$}\\
\midrule
& & \multicolumn{1}{c}{$\lambdab$} & {$(0,\ldots,0)$} & {$(0.05,\ldots,0.05)$} &
{$(0.1,\ldots,0.1)$} & {$(0.15,\ldots,0.15)$} & {$(0.2,\ldots,0.2)$} \\ 
{$f_0$} & $d$ & \multicolumn{1}{c}{$n$} & \\
\midrule
& & 200 & 0.049 & 0.097 & 0.280 & 0.622 & 0.891 \\
$I_{0.1; \ldots; 0.1}$ & 4 & 500 & 0.038 & 0.166 & 0.675 & 0.972 & 1.000 \\
& & 1000 & 0.061 & 0.387 & 0.943 & 1.000 & 1.000 \\
\\
& & 200 & 0.073 & 0.061 & 0.133 & 0.309 & 0.525 \\
$I_{0.6; \ldots; 0.6}$ & 4 & 500 & 0.063 & 0.098 & 0.384 & 0.812 & 0.978 \\
& & 1000 & 0.054 & 0.204 & 0.786 & 0.989 & 1.000 \\
\\
& & 200 & 0.061 & 0.098 & 0.318 & 0.638 & 0.929 \\
$I_{0.1; \ldots; 0.1}$ & $5$ & 500 & 0.044 & 0.193 & 0.720 & 0.989 & 0.999 \\
& & 1000 & 0.051 & 0.457 & 0.977 & 1.000 & 1.000 \\
\\
& & 200 & 0.076 & 0.072 & 0.139 & 0.356 & 0.615 \\
$I_{0.6; \ldots; 0.6}$ & $5$ & 500 & 0.052 & 0.124 & 0.427 & 0.867 & 0.992 \\
& & 1000 & 0.050 & 0.226 & 0.840 & 0.997 & 1.000 \\
\\
& & 200 & 0.053 & 0.101 & 0.322 & 0.741 & - \\
$I_{0.1; \ldots; 0.1}$ & 6 & 500 & 0.064 & 0.220 & 0.805 & 0.997 & - \\
& & 1000 & 0.046 & 0.486 & 0.994 & 1.000 & - \\
\\
& & 200 & 0.090 & 0.063 & 0.117 & 0.345 & - \\
$I_{0.6; \ldots; 0.6}$ & 6 & 500 & 0.058 & 0.131 & 0.467 & 0.905 & - \\
& & 1000 & 0.053 & 0.227 & 0.872 & 1.000 & - \\
\\
& & 200 & 0.052 & 0.127 & 0.450 & - & - \\
$I_{0.1; \ldots; 0.1}$ & 10 & 500 & 0.051 & 0.257 & 0.941 & - & - \\
& & 1000 & 0.046 & 0.590 & 1.000 & - & - \\
\\
& & 200 & 0.102 & 0.067 & 0.191 & - & - \\
$I_{0.6; \ldots; 0.6}$ & 10 & 500 & 0.064 & 0.124 & 0.649 & - & - \\
& & 1000 & 0.068 & 0.341 & 0.979 & - & - \\
\\
& & 200 & 0.048 & 0.133 & - & - & - \\
$I_{0.1; \ldots; 0.1}$ & 20 & 500 & 0.052 & 0.426 & - & - & - \\
& & 1000 & 0.052 & 0.829 & - & - & - \\
& & 5000 & 0.052 & 1.000 & - & - & - \\
\\
& & 200 & 0.084 & 0.079 & - & - & - \\
$I_{0.6; \ldots; 0.6}$ & 20 & 500 & 0.064 & 0.185 & - & - & - \\
& & 1000 & 0.060 & 0.504 & - & - & - \\
& & 5000 & 0.044 & 1.000 & - & - & - \\
\bottomrule
\end{tabular}
\end{table}

\end{document}